\documentclass{article}
\usepackage{graphicx} 
\usepackage[utf8]{inputenc}
\usepackage{amsfonts,amssymb} 
\usepackage{epstopdf}
\usepackage{inputenc}
\usepackage{amsthm}
\usepackage{enumitem}
\usepackage{cite}
\usepackage{mathtools}
\usepackage{graphicx}
\usepackage{mathtools}
\usepackage{authblk}
\usepackage{stackengine} 
\usepackage{hyperref}
\usepackage{booktabs}
\usepackage{siunitx}
\usepackage{caption}
\usepackage{mdframed}
\usepackage[scr=esstix,cal=boondox]{mathalfa} 
\usepackage{lipsum}
\usepackage{aligned-overset} 
\usepackage[stdsubgroups,nocfg]{nomencl}

\renewcommand{\nomgroup}[1]{%
\ifthenelse{\equal{#1}{A}}{%
      \item[\textbf{Skew products}]}{%
        \ifthenelse{\equal{#1}{K}}{%
        \item[\textbf{Pesin theory}]}{%
          \ifthenelse{\equal{#1}{P}}{%
          \item[\textbf{Reduction steps}]}{%
            \ifthenelse{\equal{#1}{X}}{%
            \item[\textbf{General parameters}]}{%
              {}}}}}}
\usepackage{setspace}
    \makenomenclature
\hypersetup{
hidelinks,
colorlinks=true,
linkcolor=red,
citecolor=blue,
urlcolor=black
}

\makeatletter
\def\namedlabel#1#2{\begingroup
    #2%
    \def\@currentlabel{#2}%
    \phantomsection\label{#1}\endgroup
}
\makeatother

\begin{document}

\newtheorem{maintheorem}{Theorem}
\renewcommand*{\themaintheorem}{\Alph{maintheorem}}

\newcommand{\bbN}{\mathbb{N}}
\newcommand{\bbR}{\mathbb{R}}
\newcommand{\bbZ}{\mathbb{Z}}
\newcommand{\M}{{\mathcal{M}}}
\newcommand{\im}{{\mathrm{im}}}
\newcommand{\diam}{{\mathrm{diam}}}
\newcommand{\Jac}{\mathop{\mathrm{Jac}}\nolimits}
\def\cA{{\cal A}}
\def\cB{{\cal B}}
\def\cC{{\cal C}}
\def\cE{{\cal E}}
\def\cF{{\cal F}}
\def\cH{{\cal H}}
\def\cL{{\cal L}}
\def\cO{{\cal O}}
\def\cR{{\cal R}}
\def\cS{{\cal S}}
\def\cT{{\cal T}}

\newtheorem{maincorollary}[maintheorem]{Corollary}

\newtheorem{theorem}{Theorem}[section]
\newtheorem{definition}[theorem]{Definition} 
\newtheorem{remark}[theorem]{Remark} 
\newtheorem{conjecture}{Conjecture}
\newtheorem{proposition}[theorem]{Proposition}
\newtheorem{lemma}[theorem]{Lemma}
\newtheorem{assumption}{Assumption}
\newtheorem{corollary}[theorem]{Corollary}
\newtheorem{question}{Question}
\newtheorem{example}{Example}[section]
\newmdtheoremenv{slo}{Slogan}[section]

\title{
  Strong Positive Recurrence for Strongly Entropy Expanding Maps \\
  \large }
\author{Alexandre Delplanque, Hengyi Li}
\date{May, 2026}
\maketitle

\begin{abstract}
    We introduce \emph{expanding strong positive recurrence} (expanding-SPR) for non-invertible systems with critical points. We then 
    give a sufficient condition for maps to be SPR: \emph{strong entropy expansion} (SEE).
    We prove that expanding-SPR implies good thermodynamical properties, including exponential decay of correlations of H\"older observables for the measures of maximal entropy. The class of 
    maps that are SEE includes all $C^{\infty}$-interval maps with $h_{\rm{top}}(f)>0$ and all Viana maps.
\end{abstract}

\tableofcontents

\section{Introduction}

A key phenomenon in dynamical systems is that deterministic dynamics may give rise to chaotic behavior bearing the appearance of randomness. One way to investigate this phenomenon is via ergodic theory, where the notion of entropy plays a key role. One class of objects related to entropy, the \emph{measures of maximal entropy} (MME) of a dynamical system, carries rich dynamical information and has far-reaching interactions in geometric, topological, as well as symbolic dynamics.

MMEs are best understood in the context of uniformly hyperbolic dynamics, thanks to the works of Sinai-Ruelle-Bowen in which they introduced thermodynamical formalism to study the statistical properties of smooth dynamical systems \cite{Sinai} \cite{Ruelle} \cite{Bo}. Then, their pioneering ideas are carried to dynamical systems whose hyperbolicity are not as well-behaved as the uniformly hyperbolic ones \cite{GST} \cite{Cli} \cite{BCS3}. A key notion is \emph{strong positive recurrence}, adopted from the theory of Markov shifts \cite{G1} \cite{S3}.


In this paper, we introduce \emph{expanding strong positive recurrence} (expanding-SPR) for systems that are non-invertible. We show that one class of non-invertible systems, which we call \emph{strongly entropy expanding} (SEE), are expanding-SPR. This is a class of maps characterized by high topological entropy, which includes systems in any dimension, with arbitrary critical sets. This answers a question posed by Buzzi for general classes of SPR systems beyond surfaces \cite{Bu4}. Viana maps \cite{Viana} \cite{AlvesViana} also lie in this class. Recall that a Viana map is defined as $f=f_{a_0,d,\alpha}: \mathbb{S}^1\times \mathbb{R}\rightarrow \mathbb{S}^1\times \mathbb{R} $,
$$
(\theta,t)\longmapsto (d\cdot\theta, a_0+\alpha\cdot \sin(2\pi\theta)-t^2).
$$
For $\alpha$ sufficiently small, there exists a compact interval $I_0\subseteq (-2,2)$ so that $f(S^1\times I_0)\subseteq \operatorname{int} S^1\times I_0$, so we may consider only this attractor and assume that we work on compact manifolds \cite{AlvesViana} \cite{LDP}. Regarding Viana maps, our theory on expanding-SPR systems has the following consequence.

\begin{theorem}[Exponential mixing of the MME for Viana maps]\label{thm-Viana}
    Any Viana map admits a unique measure of maximal entropy $\mu$, such that for every $\beta>0$, there exist $C>0$ and $0<\theta<1$ such that for all $\psi,\varphi:\mathcal{M}\rightarrow \mathbb{R}$ $\beta$-H\"older functions, for all $n\geq 0$,
    $$
    |\int \psi\cdot(\varphi\circ f^n)d\mu-\int \psi\cdot\varphi d\mu|<C||\psi||_{\beta}||\varphi||_{\beta}\theta^n,
    $$
    where $||\varphi||_{\beta}:=\sup|\varphi|+\sup\{ \frac{|\varphi(x)-\varphi(y)|}{d(x,y)^{\beta}}: x,y\in \mathcal{M} \}$.
\end{theorem}
\begin{proof}
    This is Theorem \ref{thm:main1} and Theorem \ref{thm:main2} combined, together with the uniqueness of the measure of maximal entropy \cite{LKC}.

    Indeed, Viana maps are $C^{\infty}$ by definition. Then, as proved in \cite{ALP} (Section 8), Viana maps satisfy $h_{\rm{top}}(f)>R(f)$, so are SEE (Definition \ref{def:SEE}) and hence expanding-SPR. Lastly, Viana maps are topologically mixing (Theorem C from \cite{AlvesViana}). 
\end{proof}

\subsection{Preliminaries}\label{sec:1-pre}

We let $\bbZ = \{..., -1, 0, 1, ... \}$, $\bbN = \{ 0, 1, 2, ... \}$ and $\bbN^* = \bbN \backslash \{ 0 \}$.

Let $X$ be a compact metric space, $T:X\rightarrow X$ be a continuous map. Denote by $\mathbb{P}(T)$ the space of $T$-invariant Borel probability measures on $X$, equipped with the weak-$\ast$ topology. Denote by $\mathbb{P}_{\rm{erg}}(T)$ the subspace of ergodic measures in $\mathbb{P}(T)$. We denote by $(\hat{X},\hat{T})$ the natural extension of $(X,T)$. Given $\hat{x}=(x_n)_{n\in \mathbb{Z}}\in \hat{X}$, denote by $\vartheta_n(\hat{x}):=x_n\in \mathcal{M}$ the $n$-th projection of $\hat{x}$ into $\mathcal{M}$. For each $\mu\in \mathbb{P}(T)$, denote by $\hat{\mu}\in \mathbb{P}(\hat{T})$ the unique lift of $\mu$ to the natural extension (see \cite{SarigNotes} for a precise definition and properties).

For $x\in X$, denote by $B_T(x,n,\varepsilon):=\{ y\in X: d(f^kx,f^ky)<\varepsilon, 0\leq k\leq n-1 \}$, the Bowen ball at $x$. Denote by $h_{\rm{top}}(T)$ the topological entropy of $(X,T)$. For $\mu\in \mathbb{P}(T)$, denote by $h_T(\mu)$ the metric entropy of $(X,T,\mu)$. Variational principle implies that
$$
h_{\rm{top}}(T)=\sup_{\mu\in \mathbb{P}(T)}\{ h_T(\mu) \}.
$$
A measure $\mu\in \mathbb{P}(T)$ is called a measure of maximal entropy (MME) if $h_T(\mu)=h_{\rm{top}}(T)$. Harmonicity of the metric entropy implies that if $\mu$ is an MME of $(X,T)$, so are almost every one of its ergodic components. This also implies that $h_{\rm{top}}(T)=\sup_{\mu\in \mathbb{P}_{\rm{erg}}(T)}\{ h_T(\mu) \}$. More details could be found in \cite{PW}. When the system $(X,T)$ is clear from context, we abbreviate the entropy $h_T(\mu)$ by $h(\mu)$.

\bigskip

Throughout this paper, we always assume that $\mathcal{M}$ is a closed manifold of dimension $d$. 
For $\alpha \in \bbR$, denote by $\lfloor \alpha \rfloor$ the largest integer smaller than or equal to $\alpha$; $\lceil \alpha \rceil$ the minimal integer larger than or equal to $\alpha$.
Let $f:\mathcal{M}\rightarrow \mathcal{M}$ be a $C^1$-differentiable map. The norm of $d_xf: T_xM\rightarrow T_{fx}M$ is denoted by $||d_xf||$, its conorm by $m(d_xf):=\min \{ ||d_xf\cdot v||: v\in T_xM, ||v||=1 \}$. Denote by $\operatorname{Jac}_xf$ the Jacobian of $f$ at $x$. 
Denote by $R(f):=\lim_{n\rightarrow\infty}\frac{1}{n}\log^+ ||d f^n||_{\infty}$ the spectral radius of $f$, where $\log^+(x) = \max(\log x, 0)$, for $x > 0$.

A map $f:\mathcal{M}\rightarrow \mathcal{M}$ is $C^r$-smooth, with $r\in [1,\infty)\cup\{ \infty \}$, if $f$ is $\lfloor r\rfloor$-differentiable and $d^rf$ is $r-\lfloor r\rfloor$-H\"older continuous, we denote by $||d^r f ||_{\infty}$ the H\"older constant. We define the $C^r$-norm of $f$ by $|| f ||_r = \max(|| f ||_k, || d^r f ||_{\infty})$, $1\leq k\leq \lfloor r\rfloor$.

Let $C(f) \subset \M$ be the set of critical points of $f$. Given $x\in \mathcal{M}\setminus C(f)$, $v\in T_x\mathcal{M}\setminus\{0 \}$, let
$$
\lambda_{f}(x,v):=\limsup_{n\rightarrow \infty}\frac{1}{n}\log||d_xf^n(v)||.
$$
For all $x\in \mathcal{M}$, $r(x):=\sharp\{ \lambda_f(x,v): v\in T_x\mathcal{M} \}<\infty$. Moreover, write $\{ \lambda_f(x,v): v\in T_xM \}=\{ \lambda_{r(x)}(x)<\cdots< \lambda_1(x) \}$, there exists a filtration $\{0\}\subseteq V_{r(x)}\subseteq V_{r(x)-1}\subseteq \cdots\subseteq V_1(x)$, with $V_i(x):=\{ u\in T_xM: \lambda_f(x,u)\leq \lambda_i(x) \}$. The maps $x\mapsto r(x)$ and $x\mapsto V_i(x)$ are measurable.

Let $(\hat{\mathcal{M}},\hat{f})$ be the natural extension of $(\mathcal{M},f)$. Define the tangent bundle of $\hat{\mathcal{M}}$ to be
$$
T\hat{\mathcal{M}}:=\bigsqcup_{\hat{x}\in \hat{\mathcal{M}}\setminus \bigcup_{n\in \mathbb{Z}}\hat{f}^{-n}\vartheta_0^{-1}(C(f))} T_{x_0} \mathcal{M}.
$$
And we may identify $T_{\hat{x}}\hat{\mathcal{M}}=T_{x_0}\mathcal{M}$. Define the cocycle of differential of $\hat{f}$ for $\hat{x}\in \hat{\mathcal{M}}\setminus \bigcup_{n\in \mathbb{Z}}\hat{f}^{-n}\vartheta_0^{-1}(C(f))$ to be
$$
d_{\hat{x}}\hat{f}^n:=\begin{cases}
    d_{x_0}f^n, & n\geq 0\\
    (d_{x_{-n}}f)^{-1}\cdot \cdots \cdot (d_{x_{-1}}f)^{-1} & n<0.
\end{cases}
$$

Oseledet's multiplicative ergodic theorem \cite{V1} implies that there exists an $\hat{f}$-invariant subset $\operatorname{REG}_f\subseteq \hat{\mathcal{M}}$ such that
\begin{itemize}
    \item[1.] for all $\hat{x}\in \operatorname{REG}_f$, there is an $f$-invariant decomposition $T_x\mathcal{M}=E^1 (\hat{x})\oplus\cdots\oplus E^{r(x_0)}(\hat{x})$ such that  for all $v\in E_i(\hat{x})\setminus \{0\}$,
    $$
    \lim_{n\rightarrow\pm \infty}\frac{1}{|n|}\log||d_{\hat{x}}\hat{f}^n(v)||=\lambda_i(x_0),
    $$
    and
    $$
    \operatorname{dim} E_i(\hat{x})=\dim V_i(x_0)-\dim V_{i+1}(x_0);
    $$
    \item[2.] for all $\mu\in \mathbb{P}(f)$ with $\log||df||\in L^1(\mu)$, $\hat{\mu}(\operatorname{REG}_f)=1$.
    
\end{itemize}
For $x\in \vartheta_0(\operatorname{REG}_f)$, we call elements in $\{ \lambda_i(x) : 1\leq i\leq r(x) \}$ the Lyapunov exponents of $x$. For $x\in \vartheta_0(\operatorname{REG}_f)$, we denote $\Lambda_f(x):=\limsup_{n\rightarrow +\infty}\frac{1}{n}\log|\operatorname{Jac}_xf^n|$. Then for $\nu\in \mathbb{P}(f)$,
$$
\Lambda_f(\nu):=\int \log|\operatorname{Jac}_xf|d\nu.
$$
When $\nu\in \mathbb{P}_{\rm{erg}}(f)$, $\Lambda_f(\nu)=\Lambda_f(x)$ for $\nu$-a.e. $x\in \mathcal{M}$. For more details regarding Lyapunov exponents and especially for non-invertible systems, see \cite{V1} \cite{SETE}.

For $\hat{\mu}\in \mathbb{P}_{\rm{erg}}(\hat{f})$, as the unique lift of $\mu\in \mathbb{P}(f)$, its Lyapunov exponents are the Lyapunov exponents of $\mu$.

\bigskip

We see that for $\mu\in \mathbb{P}_{\rm{erg}}(f)$, $r(x)$, $\operatorname{dim} V_i(x)$ and $\lambda_i(x)$ are constant $\mu$-a.e.. So we may denote by $r(\mu):=\int r(x) d\mu$, $\operatorname{dim} V_i(\mu)=\int \operatorname{dim} V_i(x)d\mu $ and $\lambda_i(\mu):=\int \lambda_i(x)d\mu$. For $\mu\in \mathbb{P}_{\rm{erg}}(f)$, the $u$-index of $\mu$ is $\min\{ d-\operatorname{dim} V_i(\mu): \lambda_i(\mu)\leq 0 \}$. Given $\chi>0$, a measure $\mu\in \mathbb{P}_{\rm{erg}}(f)$ is $\chi$-expanding, if $\lambda_{r(\mu)}(\mu)>\chi$. If $\mu$ is $\chi$-expanding, then the $u$-index of $\mu$ is $d$. A measure $\mu\in \mathbb{P}(f)$ is $\chi$-expanding if its ergodic components are $\chi$-expanding $\mu$-a.e. A measure $\hat{\mu}\in \mathbb{P}_{\rm{erg}}(\hat{f})$, the unique lift of $\mu\in \mathbb{P}_{\rm{erg}}(f)$, is $\chi$-expanding if $\mu$ is $\chi$-expanding. So is $\hat{\mu}\in \mathbb{P}(\hat{f})$.

Given $\chi>0$, denote by $\operatorname{NUE}_{\chi}$ the set of $\chi$-non-uniformly expanding points in $\hat{\mathcal{M}}\setminus \bigcup_{n\in \mathbb{Z}}\hat{f}^{-n}\vartheta_0^{-1}(C(f))$, consisting of points $\hat{x}$ such that for all $v\in T_{\hat{x}}\hat{\mathcal{M}} \setminus  \{0\}$,
\begin{equation}
\label{eqn:NUE}
\tag{NUE}
\limsup_{n\rightarrow \infty}\frac{1}{n}\log||d_{\hat{x}}\hat{f}^{-n}\cdot v||<-\chi, \ \liminf_{n\rightarrow \infty}\frac{1}{n}\log||d_{\hat{x}}\hat{f}^n\cdot v||>0.
\end{equation}
Note that for all $\chi$-expanding measures $\mu\in \mathbb{P}_{\rm{erg}}(f)$, $\hat{\mu}(\operatorname{NUE}_{\chi})=1$ (\cite{LDP}). 

\subsection{Statements of main results}

We start by introducing the notion of expanding strong positive recurrence. One may compare the formulation in Definition \ref{def:exp-SPR} to that in \cite{BCS3}. 

\begin{definition}[Pesin block]\label{def:Pesinblock1}
    Take $C>0$, $\chi>0$, $\varepsilon>0$. The expanding-Pesin-block, denoted by $\Lambda_{C,\varepsilon,\chi}$, is defined to be the subset of $\hat{\mathcal{M}}$, the natural extension of $(\mathcal{M},f)$ (see Section \ref{sec:1-pre}), consisting of all the points $\hat{x}\in \hat{\mathcal{M}}$ such that for all $k\in \mathbb{N}$,  $n\in \mathbb{Z}$,
        $$||d_{\hat{f}^n \hat{x}}\hat{f}^{-k}||<C\cdot\exp(|n|\varepsilon-k\cdot \chi).$$
\end{definition}

\begin{remark}
    We remark that the definition immediately implies the 'slow approaching to critical set': for all $\hat{x}\in \Lambda_{C,\varepsilon,\chi}$,
    \begin{equation}
    \label{eqn:criticalappro}
        m(d_{x_n}f)\geq C^{-1}\cdot\exp(-|n|\varepsilon),
    \end{equation}
    for all $n\in \mathbb{Z}$.
\end{remark}

\begin{definition}[Expanding-SPR]\label{def:exp-SPR}
    The system $(\mathcal{M},f)$ is \emph{expanding strongly positively recurrent} (expanding-SPR), if there exist $C,\chi>0$, such that for all 
    $\varepsilon>0$, there exist $h_0>0,\tau>0$, such that for the Pesin block $\Lambda_{C,\varepsilon,\chi}$, for all $\mu\in \mathbb{P}_{\rm{erg}}(f)$ with $h(\mu)>h_0$,
    $$
    \hat{\mu}(\Lambda_{C,\varepsilon,\chi})>\tau.
    $$
\end{definition}

The attribute \emph{expanding} before SPR in Definition \ref{def:exp-SPR} is important, and it is due to our definition of Pesin block (Definition \ref{def:Pesinblock1}). Our notion of Pesin block is expanding with index $d$ (compare with \cite{BCS3}, \cite{LDP}). So the corresponding SPR property detects measures with $u$-index $d$. This, in a sense, makes the critical set the only contributor to the \emph{non-uniformity} of these measures. For more discussions regarding this issue, see Section \ref{sec:1-idea}.

Similarly to \cite{BCS3}, expanding-SPR systems enjoy great stochastic properties.

\begin{maintheorem}\label{thm:main2}
    Let $(\mathcal{M},f)$ be a $C^{\infty}$ expanding-SPR system which is also topologically mixing. Then $(\mathcal{M},f)$ admits a finite number of MMEs. Moreover, each MME satisfies 
    \begin{itemize}
        \item[1.] exponentially decay of correlations of H\"older observables,
        \item[2.] asymptotic Laplace transform and asymptotic variance for Birkhoff sums,
        \item[3.] central limit theorem.
    \end{itemize}
\end{maintheorem}

Item $1$ is proved in Theorem \ref{thm:EM}; items $2$ and $3$ are given explicitly in Theorem \ref{thm:AV-BCS} and Corollary \ref{cor:CLT-BCS}.

\bigskip

The next question is which systems are expanding-SPR. 

\begin{definition}\label{def:SEE}
Let $\mathcal{M}$ be a closed Riemannian manifold. A $C^{\infty}$-map $f:\mathcal{M}\rightarrow \mathcal{M}$ is called \emph{strongly entropy expanding} (SEE) if $h_{\rm{top}}(f)>(d-1)\cdot R(f)$. 
\end{definition}

The terminology comes from Buzzi's notion of entropy expanding maps from \cite{B-dm} (of course, strongly entropy expanding implies entropy expanding).
By Proposition \ref{prop:SEE-open}, the set of SEE maps is open among $C^{\infty}$-maps in the $C^1$-topology.

One important part of this work is the entropic-continuity of the sum of Lyapunov exponents for SEE-maps. This is a consequence of a $C^r$-perturbative version (Theorem \ref{th:main-cont}).
For $r < \infty$, we emphasize that this $C^r$ version does not give the entropic continuity of the Lyapunov exponents, since there is an additional error in the entropy inequality (item $2)$ of Theorem \ref{th:main-cont}).
We state this simpler version for $C^{\infty}$ regularity.

\begin{maintheorem}[Entropic-continuity of exponents]\label{thm:main-con-1}
Let $f : \M \to \M$ be a $C^{\infty}$-map.
Let $(\nu_k)$ be a sequence of ergodic measures, with $u$-index $d = \dim \M$, converging to $\mu$.
There exist $\mu_0, \mu_1$ $f$-invariant probability measures and $\beta \in [0,1]$ such that $\mu = \beta \mu_1 + (1-\beta) \mu_0$ and
\begin{itemize}
\item[1)] if $\beta > 0$, then $\mu_1$ is expanding, i.e. it has $u$-index $d$, 
\item[2)] $\limsup\limits_{k \to +\infty} h_f(\nu_k) \leq \beta h_f(\mu_1) + (1-\beta) (d-1) R(f)$,
\item[3)] $\limsup\limits_{k \to +\infty} \Lambda_f(\nu_k) \geq \beta \Lambda_f(\mu_1)$.
\end{itemize}
\end{maintheorem}

Using Theorem \ref{thm:main-con-1}, we prove that all SEE maps are expanding-SPR.

\begin{maintheorem}\label{thm:main1}
    Let $f:\mathcal{M}\rightarrow \mathcal{M}$ be an SEE map, then $(\mathcal{M},f)$ is expanding-SPR.
\end{maintheorem}


\subsection{Consequences}\label{sec:1-FR}

We first give an account of previously studied families of systems that are SEE, (so expanding-SPR by Theorem \ref{thm:main1}). By Proposition \ref{prop:SEE-open} SEE maps are open in the $C^1$-topology. So Theorem \ref{thm-Viana} can be strengthened to a neighborhood of Viana maps.

\begin{proposition}\label{prop:Vianaper}
    There is a $C^1$-neighborhood of Viana maps 
    such that all $C^{\infty}$-maps in this neighborhood are expanding-SPR, so satisfy all the properties in Theorem \ref{thm:main2}.
\end{proposition}

For $C^{\infty}$-maps $f:I\rightarrow I$, $I=[0,1]$ or $S^1$, the dimension of $I$ is $d=1$ and $(d-1)\cdot R(f)=0$. So as long as $h_{\rm{top}}(f)>0$, $f$ is SEE, so expanding-SPR.

\begin{proposition}\label{prop:1-dim}
    Let $I=[0,1]$ or $S^1$, and $f:I\rightarrow I$ be a $C^{\infty}$-map, with $h_{\rm{top}}(f)>0$. Then $(I,f)$ is expanding-SPR.
\end{proposition}

Another class of SEE maps comes from linear maps of the torus.

\begin{proposition}\label{prop:linear}
    If a linear map $A:\mathbb{T}^d\rightarrow \mathbb{T}^d$ satisfies
    $$
    \sum_{i=1}^d\log^+ |\lambda_i| > (d-1)\cdot \max\{ \log^+|\lambda_i| \},
    $$
    then $(\mathbb{T}^d,A)$ is expanding-SPR.
    In particular, in dimension $2$, for all linear maps $A:\mathbb{T}^2\rightarrow \mathbb{T}^2$, as long as $\lambda_i>1$, $i=1,2$, $(\mathbb{T}^2,A)$ is expanding-SPR.
\end{proposition}

Again, we allow perturbations.

\begin{proposition}\label{prop:tori}
    There is a $C^1$-neighborhood of the linear toral maps in Proposition \ref{prop:linear} such that all $C^{\infty}$-maps in this neighborhood are expanding-SPR. This open neighborhood contains maps which are not uniformly hyperbolic.
\end{proposition}

In \cite{OV1} \cite{OV2}, a class of local diffeomorphism is studied as smooth perturbations of affine expanding maps, where they showed the existence and uniqueness of measure of maximal entropy. In dimension $2$, their systems are always SEE. In higher dimensions, the two classes of maps overlaps but neither one includes the other. 


\bigskip
Finally, notice that SPR properties for Markov shifts are open for the space of potentials \cite{CS} (for further explanations on terminologies, see Section \ref{sec:3}). This implies that we may also generalize the notion of SPR to non-zero potentials. In particular, this allows us to treat the unique equilibrium states studied in \cite{LKC}.

\begin{proposition}\label{prop:non0}
    Let $(\mathcal{M},f)$ be an expanding-SPR system. Then there exists $\varepsilon>0$ such that for all potentials $\varphi:\mathcal{M}\rightarrow \mathbb{R}$ with $|| \varphi ||_{\infty}< \varepsilon$, there exist 
    finitely many equilibrium states for $\varphi$, each satisfying the properties in Theorem \ref{thm:main1}.
\end{proposition}

\subsection{Contexts of previous works}\label{sec:Intro-CPW}

SPR theory for diffeomorphisms has been established by Buzzi-Crovisier-Sarig \cite{BCS2} \cite{BCS3}. The theory has far reaching applications, including surface diffeomorphisms \cite{BCS3} and certain partially hyperbolic systems \cite{MPP}.

The current work, sharing the same philosophy, addresses SPR phenomena from a different direction. In our approach as well as in that of Buzzi-Crovisier-Sarig's, the defect of SPR is traced to the defect of continuity of exponents at MMEs (Corollary \ref{cor:UI} and Section 3.2 in \cite{BCS3}). In the current work, this is then traced to the loss of mass towards the critical set. This is a first major difference in difficulties from \cite{BCS2} \cite{BCS3}, whose difficulties are caused mostly by the non-continuity of unstable/stable distributions. 

Theorem \ref{thm:main1} is a consequence of Theorem \ref{thm:main-con-1}, which provides significant dynamical information in itself (certain implications apart from expanding-SPR are given in Sections \ref{sec:2-UI} and \ref{sec:2-con}). In a sense, Theorem \ref{th:main-cont} shows that high topological entropy prevents measures losing exponents by losing weight towards \emph{non-hyperbolic/non-expanding regions} when approaching MMEs. Research works in this direction include \cite{BCS2} \cite{B}, which deal with surface diffeomorphisms, and \cite{L} \cite{DL}, which deal with maps in dimension $1$, with \cite{BCS2} \cite{B} \cite{DL} using reparametrization techniques to control the entropy along \emph{neutral} components, where we don't see expansion. Reparametrization techniques are also used to build SRB measures in various contexts \cite{Bur2} \cite{A} \cite{BB} \cite{A2}. 

The current work, as well as \cite{BCS3}, takes ideas from the SPR theory for Markov shifts \cite{G1} \cite{S3}. We note that there are other ways to incorporate this idea of SPR into smooth dynamical systems \cite{GST} \cite{Cli}.


\bigskip

As in the uniform hyperbolic case \cite{Bowen}, a powerful tool used in the current work is the \emph{symbolic model} $(\Sigma,\sigma)$ of the system $(\mathcal{M},f)$. The model and the projection $\pi:\Sigma\rightarrow \hat{\M}$ is called a coding. The codings that we use are \cite{LDP} (see \cite{ALP} as well) and \cite{BLY}. A comparison of the two codings can be seen at Remark \ref{re-Acondi}. The theory starts from \cite{S1} for surface diffeomorphisms, which is then extended to non-uniformly hyperbolic diffeomorphisms in any dimension \cite{SBO} and for non-invertible systems \cite{L1} \cite{ALP}.

\bigskip

The definition of SEE maps is motivated by \cite{DL} \cite{BB} from reparametrization theory, and by \cite{B-dm} from entropy theory. Much work has been done regarding certain subclasses of SEE maps. For $1$-dimensional systems, the existence of MMEs for $1$-dimensional systems is studied \cite{BurguetExistMME} \cite{BuzziRuette}, as well as finiteness \cite{Buzzi1997} and exponential mixing of the MMEs \cite{HK} \cite{Li2} \cite{BS}. The uniqueness of MME for Viana maps \cite{Viana} is proved in \cite{Santana} \cite{LKC}. The existence follows from the general work of Newhouse \cite{NewhouseContEntropy}. The existence of SRB measures for Viana maps, and their fast mixing, is proved in \cite{Alves}. Pertubations of affine torus maps, regarding existence of MMEs, are studied in \cite{OV1} \cite{OV2}. 

We remark that, as applications of Theorem \ref{thm:main1} and Theorem \ref{thm:main2}, Theorem \ref{thm-Viana} is unknown until this work, so is Proposition \ref{prop:Vianaper} in the perturbation version; Proposition \ref{prop:1-dim}, for exponential mixing and all other properties, is new in the sense that previous works assume additional properties regarding critical sets, be them unimodal \cite{BS}, piecewise monotone \cite{HK}, so requiring $f$ to have finitely many monotone branches; as far as we know, Proposition \ref{prop:tori} is also new in the sense of exponential mixing and other thermodynamical properties regarding MMEs, when the system studied is non-uniformly expanding. Proposition \ref{prop:non0} then extends the thermodynamical properties to equilibrium states of non-zero potentials \cite{LKC}.










\subsection{Ideas of the proof}\label{sec:1-idea}

The proof consists of two major parts. The first one establishes the theory of expanding-SPR systems via non-uniformly hyperbolic dynamics and symbolic dynamics (Theorem \ref{thm:main2}, Theorem \ref{thm:main1}); the second studies how critical sets, exponents and the entropy of the system interact via Yomdin theory \cite{Yomdin1987VolumeGA} (Theorem \ref{thm:main-con-1}). 

\bigskip

The goal of the first part is to show that our \emph{geometric definition} of SPR (Definition \ref{def:exp-SPR}) implies the symbolic model of the system being SPR \emph{as Markov shifts} (Definition \ref{def:1-sSPR}). This resembles the approach in \cite{BCS3}. The difference is that, since we choose a rather restrictive definition of SPR, we exclude all the other causes of loss of uniformity in hyperbolicity apart from the critical set of the system. Then, this defect of uniformity caused by the critical set is quantified via a defect of \emph{uniform integrability} ($\gamma$ in Corollary \ref{cor:UI}). One may compare our version of Pesin's tempered envelope lemma (Lemma \ref{temper}) to Lemma 2.20 in \cite{BCS3}, where ours has an additional error term caused by defect of uniform integrability.


\bigskip

The goal of the second part is to find a sufficient condition to control the defect of uniform integrability. For $C^{\infty}$-surface diffeomorphisms, as well as $C^{\infty}$-maps in dimension $1$, the condition is $h_{\rm{top}}(f)>0$ \cite{BCS2} \cite{BCS3} \cite{DL}. In our setting it is SEE condition. 

In spirit, the proof of Theorem \ref{thm:main-con-1} is a higher-dimensional generalization of the techniques in \cite{B} and \cite{DL}. Here our condition SEE forces expansion in all direction (Lemma \ref{lem:chi-expan}), so also reduces the contribution of defect of exponents continuity to critical set. However, there are substantial differences between the $1$-dimensional version and the SEE-version.

For SEE maps, the hyperbolic times considered are expanding on average directionwise, that is, it is built from the Jacobian, not an expanding observable like that of Ben Ovadia-Burguet \cite{BB}. This is enough to get the continuities (entropy and exponent), but not the expansion. We need one further step to get expansion, when the limit entropy goes to $h_{\rm{top}}(f)$.

Since we work in higher dimensions, neutral times can carry entropy (unlike in \cite{BCS2}, \cite{B}, \cite{DL}), but in at most $(d-1)$-directions. This justifies our definition of SEE: the $(d-1)\cdot R(f)$ term comes from our reparametrization along the neutral times, and we recover uniform integrability at MMEs when the MMEs themselves has entropy greater than this 'error'. 
This reparametrization technique requires two results.
The first one is Lemma \ref{lem:iter}, which is a non-invertible version of Lemma 4.1 from Ben Ovadia-Burguet \cite{BB}.
Then we iterate this lemma and obtain Lemma \ref{lem:repara}, which is an adaptation of the semi-local reparametrization lemma by Burguet \cite{B}.

Contrarily to \cite{B} and \cite{DL}, we do not obtain the entropic continuity of the Lyapunov exponents in $C^r$ regularity, $r < \infty$.
The reason is that 'bounded couples' (Definition \ref{def:bounded-couples}), which are the higher-dimensional version of 'bounded curves' \cite{B}, do not have as good geometric properties, nor as good shapes.
Technically, we have an additional error term coming from the geometry of the intersection between a small ball and a bounded couple.


\subsection{Acknowledgments}

The authors thank David Burguet and J\'er\^ome Buzzi for their suggestions. The authors also thank Davi Obata, Lima Yuri and Mat\'eo Ghezal for the fruitful discussions.

No AI has been used to produce this paper.

\section{From continuity of the exponents to SPR}\label{sec:2}

\subsection{Uniform integrability}
\label{sec:2-UI}

\begin{corollary}[Uniform integrability]\label{cor:UI}
Let $f : \M \to \M$ be an SEE map. 
For any $\varepsilon > 0$, there exist $h_0 \in (0, h_{\rm{top}}(f))$ and $\eta > 0$ such that, for any ergodic measure $\mu$ with $h(\mu) > h_0$, we have
$$
\int_{\{ |\Jac f| < \eta \}} -\log |\Jac f| d \mu < \varepsilon.
$$
\end{corollary}

\begin{proof}
By contradiction, assume that there is a sequence $(\nu_k, \eta_k, \delta_k)$ with
\begin{itemize}
\item[i)] $\nu_k$ is ergodic and converges to some $\mu$,
\item[ii)] $\eta_k, \delta_k \to 0$,
\item[iii)] $h(\nu_k) \geq h_{\rm{top}}(f) - \delta_k$,
\item[iv)] $\int_{\{ |\Jac f| < \eta_k \}} - \log |\Jac f| d \nu_k \geq \varepsilon$,
\item[v)] $(h(\nu_k))$ converges.
\end{itemize}
Let
$$
\gamma = \lim\limits_{\eta \to 0} \limsup\limits_{k \to +\infty} \int_{\{ |\Jac f| < \eta \}} - \log |\Jac f| d \nu_k.
$$
Let $\alpha = (d-1) R(f) < h_{\rm{top}}(f)$ and $\beta, \mu_0, \mu_1$ be given by Theorem \ref{thm:main-con-1}.
Notice that $\gamma \geq \varepsilon$, and
\begin{align*}
\gamma
&= \int \log |\Jac f| d \mu - \liminf\limits_{k \to +\infty} \int \log |\Jac f| d \nu_k\\
&\leq \Lambda(\mu) - \beta \Lambda(\mu_1)\\
&= (1-\beta) \Lambda(\mu_0)\\
&\leq (1- \beta) \log || \Jac f ||_{\infty}.
\end{align*}
The first inequality does not come from Theorem \ref{thm:main-con-1}, which gives a $\limsup$ and not a $\liminf$ inequality, but item $3')$ from Theorem \ref{th:main2}.
Thus $\beta \leq 1 - \frac{\gamma}{\log || \Jac f ||_{\infty}}$.
On the other hand we have the following
\begin{align*}
h_{\rm{top}}(f)
&\leq \lim\limits_{k \to +\infty} h(\nu_k) + \delta_k\\
&\leq \beta h(\mu_1) + (1-\beta) \alpha\\
&= \beta (h(\mu_1) - \alpha) + \alpha\\
&\leq \beta (h_{\rm{top}}(f) - \alpha) + \alpha\\
&\leq \left ( 1-\frac{\gamma}{\log || \Jac f ||_{\infty}} \right ) h_{\rm{top}}(f) + \frac{\gamma}{\log || \Jac f ||_{\infty}} \alpha.
\end{align*}
Therefore $\gamma = 0$ and we cannot have $\gamma \geq \varepsilon$.
\end{proof}

\begin{corollary}[Existence of measures of maximal entropy and entropic continuity of the Lyapunov exponent]
\label{cor:Cont-expo}
Let $f : \M \to \M$ be an SEE map. 
Let $(\nu_k)_{k\in \mathbb{N}}$ be a sequence of ergodic measures, $\nu_k\rightarrow \mu$ in the weak-$\ast$ topology and $\lim_{k\rightarrow \infty}h_f(\nu_k) = h_{\rm{top}}(f)$.
Then $\mu$ is a measure of maximal entropy for $f$ and
$$
\lim_{k\rightarrow \infty} \sum_{i=1}^d\lambda^+_i(\nu_k)=\sum_{i=1}^d\lambda_i(\mu).
$$
\end{corollary}
\begin{proof}
We first mention that for $k$ large enough, the ergodic measure $\nu_k$ has only positive Lyapunov exponents.
Indeed, by Ruelle's inequality \cite{Ruelle1978}, if $d_u(\nu_k)$ is the number of positive exponents of $\nu_k$, then $h_f(\nu_k) \leq d_u(\nu_k) R(f)$, thus $d_u(\nu_k) > d-1$ and $d_u(\nu_k) = d$.\\

    Let $\nu_k\rightarrow \mu$ be as stated in the corollary. Notice that
    $$
    \lim_{\eta\rightarrow 0}\limsup_{k\rightarrow \infty}\int_{\{ |\Jac f| < \eta \}} -\log |\Jac f| d \nu_k=\sum_{i=1}^d\lambda_i(\mu)-\liminf_{k\rightarrow \infty} \sum_{i=1}^d\lambda^+_i(\nu_k).
    $$
    As Corollary \ref{cor:UI} implies the left hand of the equation being $0$, we have the continuity as stated.

    Moreover, by the same proof as in \cite{B} and \cite{DL}, Theorem \ref{th:main-cont}, applied to $\nu_k\rightarrow \mu$, implies $\beta=1$. So $\mu$ is an MME.
\end{proof}

\subsection{Lower semi-continuity of the topological entropy}
\label{sec:2-con}

We show that the topological entropy is lower semi-continuous at SEE maps and that the set of SEE maps is open.
This does not use Theorem \ref{thm:main-con-1}, and we in fact have the following stronger statement.

\begin{proposition}
Let $f : \M \to \M$ be a $C^r$-map, $r>1$, satisfying the following restrictive variational principle
$$
h_{\rm{top}}(f) = \sup\limits_{\substack{
\mu \: f\text{-invariant s.t.}\\
\mu \text{ has only positive exponents}}}
h_f(\mu).
$$
Then the topological entropy is lower semi-continuous at $f$ in the $C^1$-topology.
\end{proposition}

\begin{proof}
We first use the result by Gelfert \cite{Gelfert}, which is an analog for non-uniformly expanding maps of the horseshoe approximation result by Katok \cite{Katok} for surface diffeomorphisms, then we use the structural stability result by Shub \cite{Shub} regarding uniformly expanding compact sets.
\end{proof}

Since any SEE map satisfies the above restrictive variational principle, this gives the following:

\begin{proposition}
\label{prop:lsc-top-ent}
The topological entropy is lower semi-continuous at SEE maps for the $C^1$-topology.
\end{proposition}

This has the following consequence:
\begin{proposition}
\label{prop:SEE-open}
The set of SEE maps is open in the $C^1$-topology.
\end{proposition}

\begin{proof}
By definition of SEE maps and Proposition \ref{prop:lsc-top-ent}, it remains to check that $f \mapsto R(f)$ is upper semi-continuous, which it is as an infimum of continuous maps: $R(f) = \inf\limits_{n \geq 1} \frac{1}{n} \log^+ ||d f^n ||_{\infty}$.
\end{proof}

\subsection{SPR property}\label{section:SPRSEE}

In this section we prove Theorem \ref{thm:main1}. Our approach resembles the one given in \cite{BCS3}. The idea is to first find a Pliss set in $\hat{\mathcal{M}}$ with measure lower bounded for measures of high entropy. Then we apply a modified version of Pesin's tempered envelope lemma to compare the measure of the Pliss set with that of a Pesin block, uniformly with respect to measures with large entropy.

We first prove that $f:\mathcal{M}\rightarrow \mathcal{M}$ being SEE implies that all measures with sufficiently high entropy are $\chi$-expanding.


\begin{lemma}\label{lem:chi-expan}
    Let $f:\mathcal{M}\rightarrow \mathcal{M}$ be a $C^1$-map, $h_{\rm{top}}(f)>(d-1)R(f)$. Then there exist $\chi>0$, $h_0>0$ such that for all $\mu\in \mathbb{P}_{\rm{erg}}(f)$ with $h(\mu)>h_0$, $\mu$ is $\chi$-expanding.
\end{lemma}
\begin{proof}
    Take $0<\chi<h_{\rm{top}}(f)-(d-1)R(f)$. Then set $h_0:=(d-1)R(f)+\chi$. For all measure $\mu\in \mathbb{P}_{\rm{erg}}(f)$ with $h(\mu)>\chi$, we have
    $$
    (d-1)R(f)+\chi<h(\mu)\overset{(a)}{\leq} \sum_{i=0}^d \lambda^+_i(\mu)\overset{(b)}{\leq} (d-1)R(f)+\lambda^+_{r(\mu)}(\mu),
    $$
    where $\lambda^+_i(\mu):=\max\{ \lambda_i(\mu),0 \}$. Inequality $(a)$ is Ruelle inequality; inequality $(b)$ follows from $\lambda_i(\mu)\leq R(f)$. Cancel the $(d-1)R(f)$ in the left and right hand sides, we may conclude.
\end{proof}





\bigskip

Firstly, we collect some technical lemmas that will be used in this section.

\begin{lemma}[Pliss lemma, Lemma 2.18 \cite{BCS3}]\label{Pliss}
Let $(X,T)$ be an invertible topological system and $\nu\in \mathbb{P}(T)$. Let $\varphi\in L^1(\nu)$, upper-bounded by $A \in \bbR$ $\nu$-a.e.
Then for every $\beta<A$,
$$
\nu(\{ x:\forall j\geq1, \sum_{i=0}^{j-1}\varphi\circ T^i(x)\geq \beta\cdot j \})\geq \frac{\int \varphi d\nu-\beta}{A-\beta}.
$$
\end{lemma}

In the end, we introduce some notations. Given $T:X\rightarrow X$ and $0<h<h_{\rm {top}}(T)$, we denote by
$$
\begin{aligned}
&\mathbb{P}^h(T):=\{ \mu\in \mathbb{P}(T): h(\mu)>h \};\\
&\mathbb{P}^h_{\rm{erg}}(T):=\{\mu\in \mathbb{P}_{\rm{erg}}(T): h(\mu)>h  \}.
\end{aligned}
$$

\subsubsection{Constructing Pliss sets}

Throughout the remainder of Section \ref{section:SPRSEE}, we always consider SEE maps, that is, $f:\mathcal{M}\rightarrow \mathcal{M}$, $C^{\infty}$, $h_{\rm{top}}(f)>\left(d-1\right)R(f)$.

\begin{definition}\label{def:Plissset}
Given $C>0$, $\chi>0$, $n_0\in \mathbb{N}$, \emph{the $(C, \chi,n_0)$-Pliss set} is defined to be $P_{C, \chi,n_0}\subseteq \hat{\mathcal{M}}$, consists of all $\hat{x}\in \hat{\mathcal{M}}$ such that
\begin{itemize}
\item[1.] $||d_{\hat{x}}\hat{f}^{-j\cdot n_0}||\leq \exp(-\chi\cdot j\cdot n_0)$ for all $j\geq 0$,
\item[2.] $m(d_{x_0}f)\geq C^{-1}$. 
\end{itemize}
Without condition $2)$, we say that the set is \emph{the $(\chi,n_0)$-Pliss set}, denoted by $P_{\chi,n_0}$.
\end{definition}

\begin{definition}\label{def:Eset}
Given $C>0$, $\chi>0$, \emph{the $(C,\chi)$-expanding set} $P_{C,\chi}\subseteq \hat{\mathcal{M}}$, consists of all $\hat{x}\in \hat{\mathcal{M}}$ such that
\begin{itemize}
\item[1.] $||d_{\hat{x}}\hat{f}^{-k}||\leq C\cdot \exp(-k\cdot \chi)$ for all $k\geq 0$,
\item[2.] $m(d_{x_0}f)\geq C^{-1}$.
\end{itemize}
\end{definition}

In this section, we prove the following proposition.

\begin{proposition}\label{prop:uniformPliss}
There exists $\chi_0>0$ such that for all $0<\tau<1$, there exist $C_0>0$ and $h_0>0$, satisfying for all $\hat{\mu}\in \mathbb{P}^{h_0}(\hat{f})$, 
$$
\hat{\mu}(P_{C_0,\chi_0})>\tau.
$$
\end{proposition}

\bigskip

Denote by $\mathcal{L}\subseteq \mathbb{P}(\hat{f})$ the set of accumulation points of sequences $(\hat{\nu}_k)_{k\in \mathbb{N}}$ such that $\hat{\nu}_k\in \mathbb{P}_{\rm{erg}}(\hat{f})$, $h(\hat{\nu}_k)\rightarrow h_{\rm{top}}(f)$. We remark that 1. $\mathcal{L}\subseteq \mathbb{P}(\hat{f})$ is closed as a set of accumulation points; 2. elements $\hat{\mu}\in \mathcal{L}$ are not necessarily ergodic. 

\begin{lemma}\label{lem:Lcompact}
There exists $0<\chi<h_{\rm{top}}(f)$ such that $\mathcal{L}\subseteq \{ \hat{\nu}\in \mathbb{P}(\hat{f}): \text{$\nu$ is $\chi$-expanding} \}$ and is compact. Moreover, all ergodic components $\mu_x$ of $\mu$ whose lift $\hat{\mu}\in \mathcal{L}$ are $\chi$-expanding.
\end{lemma}
\begin{proof}
    We know that $\mathbb{P}(\hat{f})$ is compact, and $\mathcal{L}$ is closed, so compact in $\mathbb{P}(\hat{f})$ under the weak-$\ast$ topology. Take $\hat{\mu}\in \mathcal{L}$, by definition it suffices to show that for $\hat{\mu}$-a.e. ergodic component $\hat{\mu}_{\hat{x}}$ of $\hat{\mu}$, $\lambda_{\tau(\hat{\mu}_{\hat{x}})}(\hat{\mu}_{\hat{x}})>\chi$.

    First notice that metric entropy is upper semi-continuous as $f$ is $C^{\infty}$. So $\lim_{k\rightarrow \infty}h(\hat{\mu}_k)=h_{\rm{top}}(f)$ implies that $h(\hat{\mu})=h_{\rm{top}}(f)$, thus $h(\hat{\mu}_{\hat{x}}) = h_{\rm{top}}(f)$ for $\hat{\mu}$-a.e. ergodic component $\hat{\mu}_{\hat{x}}$ of $\hat{\mu}$. Lemma \ref{lem:chi-expan} implies that there exist $h_0>0$ and $\chi>0$ such that for all $\hat{\nu}\in \mathbb{P}_{\rm{erg}}^{h_0}(\hat{f})$, $\nu$ is $\chi$-expanding. 

    Fix this $\chi$. Indeed, $\hat{\mu}$-a.e. ergodic component $\hat{\mu}_{\hat{x}}$ has entropy $h_{\rm{top}}(f)>h_0$. So they are $\chi$-expanding. This concludes the lemma.
\end{proof}

Given $N\in \mathbb{N}^{\ast}$, denote by
$$
U_N:=\{ \hat{x}\in \hat{\mathcal{M}}: ||d_{\hat{x}}\hat{f}^{-N}||<e^{-\chi\cdot N} \}.
$$

\begin{lemma}\label{uniformbigexpan}
Let $\chi>0$, $\mathcal{K}\subseteq \{\hat{\nu}\in \mathbb{P}(\hat{f}): \text{ $\mu$ is $\chi$-expanding}\}$ a compact subset, and $0<\delta<\frac{1}{2}$. Then there exist $N\in \mathbb{N}^{\ast}$ and $\mathcal{V}$, an open neighborhood of $\mathcal{K}$, such that for all $\hat{\mu}\in \mathcal{V}$, there exists a positive integer $N_1=N_1(\hat{\mu})\leq N$, with
$$
\hat{\mu}(U_{N_1})>1-\delta^2.
$$
\end{lemma}
\begin{proof}
    See Lemma 3.13 in \cite{BCS3}. The proof is exactly the same.
\end{proof}

Now we are ready to prove Proposition \ref{prop:uniformPliss}.  Before the proof, we first record here certain calculations which will be used repeatedly throughout the later sections. 
By linear algebra, for all $x\in \M$,
\begin{equation}\label{eqn:Jacconor}
    m(d_xf)^d\leq |\operatorname{Jac}_xf|\leq m(d_xf)\cdot ||d_xf||^{d-1} \leq m(d_xf)\cdot||df||_{\infty}^{d-1};
\end{equation}
\begin{equation}\label{eqn:Jacnor}
    |\operatorname{Jac}_xf|\leq ||d_x f||^d.
\end{equation}
For two integers $a \leq b \in \bbZ$, we let $[\![ a , b]\!] = [a,b] \cap \bbZ$ and $[\![ a, b [\![ = [a,b) \cap \bbZ = [\![a , b-1 ]\!]$.

\begin{proof}[Proof of Proposition \ref{prop:uniformPliss}]
    Fix $\tau>0$. Fix $0<\chi<h_{\rm{top}}(f)$, then $0<\chi_0<\chi$.
    
    Fix $0<\gamma<\min\{\frac{1}{4},\frac{1}{\chi},\frac{1-\tau}{3}\}$ such that $ (1-\frac{3}{2}\gamma)\chi>\chi_0$.
    Take $\varepsilon=2\gamma^2$ in Corollary \ref{cor:UI}, then it implies that there exist $\eta>0$, $h_1>0$ such that for all $\hat{\mu}\in \mathbb{P}^{h_1}(\hat{f})$ (recall how measures are lifted to the natural extension),
    \begin{equation}
    \label{eqn:UIphi'}
    \int_{\hat{\mathcal{M}}} 1_{\{ |\operatorname{Jac}_{x_{-1}}f|<\eta ||df||^{d-1}_{\infty} \}} \log |\operatorname{Jac}_{x_{-1}}f| d\hat{\mu}>-\frac{1}{2}\chi\gamma^2.
    \end{equation}

    By $\eqref{eqn:Jacconor}$, we have $\{ m(d_{x_{-1}}f)<\eta \}\subseteq \{ |\operatorname{Jac}_{x_{-1}}f|<\eta ||df||^{d-1}_{\infty}\}$, and 
    $$
    \begin{aligned}
        \int_{\hat{\mathcal{M}}} 1_{\{ m(d_{x_{-1}f})<\eta \}}\log m(d_{x_{-1}}f) d\hat{\mu} &\overset{(a)}{\geq} \int_{\hat{\mathcal{M}}} 1_{\{ |\operatorname{Jac}_{x_{-1}}f|<\eta ||df||^{d-1}_{\infty} \}} \log m(d_{x_{-1}}f) d\hat{\mu} \\
        &\overset{(b)}{\geq} \int_{\hat{\mathcal{M}}} 1_{\{ |\operatorname{Jac}_{x_{-1}}f|<\eta ||df||^{d-1}_{\infty} \}} \left ( \log |\operatorname{Jac}_{x_{-1}}f|-(d-1)\log||df||_{\infty} \right ) d\hat{\mu}\\
        &\geq 2\int_{\hat{\mathcal{M}}} 1_{\{ |\operatorname{Jac}_{x_{-1}}f|<\eta ||df||^{d-1}_{\infty} \}} \log |\operatorname{Jac}_{x_{-1}}f| d\hat{\mu}\overset{(c)}{>}-2\gamma^2.
    \end{aligned}
    $$
    Inequality $(a)$ follows from the set inclusion we stated and by choosing $\eta<||df||_{\infty}^{-(d-1)}$. Inequality $(b)$ follows from the right hand side of $\eqref{eqn:Jacconor}$. Inequality $(c)$ is just $\eqref{eqn:UIphi'}$.

    That is,
    \begin{equation}
    \label{eqn:UIphi}
    \int_{\hat{\mathcal{M}}} 1_{\{ m(d_{x_{-1}}f)<\eta \}}\log m(d_{x_{-1}}f) d\hat{\mu}>-\frac{1}{2}\chi\gamma^2,
    \end{equation}
    and define $\varphi(\hat{x}):=1_{\{ m(d_{x_{-1}}f)<\eta \}}\log m(d_{x_{-1}}f)$.

    For all $\hat{\mu}\in \mathbb{P}^{h_1}(\hat{f})$, apply Lemma \ref{Pliss} to $(\hat{\mathcal{M}},\hat{f},\hat{\mu})$, with $\varphi$ taken as above and $A=0$, $\beta=-\chi\gamma$. Then 
    $$
    \hat{\mu}(\{ \hat{x}: \forall j\geq 1 \sum_{i=0}^{j-1}\varphi(\hat{f}^i\hat{x})\geq -\chi\gamma\cdot j \})\geq \frac{\int \varphi d\hat{\mu}-(-\chi\gamma)}{0-(-\chi\gamma)}\geq 1-\frac{1}{2}\gamma.
    $$

    Apply Lemma \ref{uniformbigexpan} with $\mathcal{K}$ taken as $\mathcal{L}\subseteq \mathbb{P}_{\chi}(\hat{f})$, and 
    $$
    \delta:=\min\left\{ \frac{\gamma\cdot \chi}{-4\log\eta+4}, \frac{1}{2}\gamma \right \}.
    $$
    This gives $N\in \mathbb{N}^*$ and $\mathcal{V}$, open neighborhood of $\mathcal{L}$ such that for all measures $\hat{\nu}\in \mathcal{V}$, there exists $N_1(\hat{\nu})\leq N$, satisfying $\hat{\nu}(U_{N_1})>1-\delta^2$. As $\mathbb{P}(\hat{f})$ is compact, and $\mathcal{L}$ is the set of accumulation points of measures of entropy increasing to $h_{\rm{top}}(f)$, we may find $h_2<h_{\rm{top}}(f)$ such that $\mathbb{P}^{h_2}(\hat{f})\subseteq \mathcal{V}$. Define $h:=\max\{ h_1,h_2 \}$.

    Take one $\hat{\mu}\in \mathbb{P}^h(\hat{f})$. Since $\hat{\mu}\in \mathcal{V}$, there exists an $N_1\leq N$. Apply Lemma \ref{Pliss} again, this time taking $\varphi':=1_{U_{N_1}}$, $A=1$, $\beta=1-\delta$ . Since $\hat{\mu}(U_{N_1})>1-\delta^2$, we get 
    $$
    \begin{aligned}
    \hat{\mu}(\{ \hat{x}: \forall n\geq 1, \sharp\{ 0\leq k<n: x_{-k}\in U_{N_1} \}\geq n\cdot (1-\delta)\})&\geq \frac{1-\delta^2-(1-\delta)}{1-(1-\delta)}\\
    &=1-\delta>1-\frac{1}{2}\gamma.
    \end{aligned}
    $$ 

    Now consider the subset $\hat{H} \subseteq \hat{\mathcal{M}}$ such that for all $\hat{x}\in \hat{H}$,
    \begin{itemize}
    \item[1.] for all $n\in \mathbb{N}$, $\sum_{i=0}^{n-1}\varphi (\hat{f}^i\hat{x})\geq - \chi\gamma\cdot n$;
    \item[2.] for all $n\in \mathbb{N}$, $\sharp\{ 0\leq k<n: \hat{f}^{-k}\hat{x}\in U_{N_1} \}\geq n (1-\delta) \geq n\cdot \left(1-\frac{\gamma\cdot \chi}{-4\log\eta+4} \right)$.
    \end{itemize}
    The set $\hat{H}$ satisfies $\hat{\mu}(\hat{H})>1-3\gamma$ by our above observation. 

    Take $\hat{x}\in \hat{H}$. For all $n\geq 0$, we define $b_0:=0$, then
    $$
    a_{i+1}:=\min\{ k\geq b_i: \hat{f}^{-k}\hat{x}\in U_{N_1} \}, \ b_i:=a_i+N_1.
    $$

    We set $J:=\max\{ i\geq 1: b_i\leq n \}$ and $V_1:=\bigcup_{i=1}^J [\![a_i,b_i[\![$. So
    $$
    \begin{aligned}
    &\sharp V_1\geq n\cdot (1-\delta)-N_1\geq n\cdot (1-\frac{\gamma}{4})-N;\\
    &\sharp \left([\![0,n[\![\setminus V_1\right)\leq n\cdot (1-(1-\delta))+N_1\leq n\cdot \frac{\gamma\cdot \chi}{-4\log\eta+4}+N.
    \end{aligned}
    $$
    
    Now take $\hat{x}\in \hat{H}$, we have
    $$
    \begin{aligned}
    ||d_{\hat{x}}\hat{f}^{-n}||&\leq\prod_{i=1}^{J} e^{-\chi\cdot (b_i-a_i)}\cdot \prod_{k\notin V_1: m(d_{x_{-k-1}}f)<\eta} m(d_{x_{-k-1}}f)^{-1}\cdot \prod_{k\notin V_1:m(d_{x_{-k-1}}f) \geq \eta}  m(d_{x_{-k-1}}f)^{-1}\\
    &\leq \prod_{i=1}^{J} e^{-\chi\cdot (b_i-a_i)}\cdot \prod_{k\notin V_1: m(d_{x_{-k-1}}f)<\eta} m(d_{x_{-k-1}}f)^{-1}\cdot \prod_{k\notin V_1:m(d_{x_{-k-1}}f)\geq \eta} \eta^{-1}\\
    &\overset{(a)}{\leq} e^{-\chi\cdot \sharp V_1}\cdot e^{\chi\gamma\cdot n}\cdot e^{-\log\eta\cdot\sharp([\![0,n[\![\setminus V_1)}\\
    &\leq e^{-(1-\frac{\gamma}{4})\cdot n\cdot \chi+N\cdot \chi}\cdot e^{\chi\gamma\cdot n}\cdot e^{-\log \eta\cdot(n\cdot \frac{\gamma\cdot \chi}{-4\log\eta+4}+N)}\\
    &\leq e^{-(1-\frac{\gamma}{2}-\gamma)\chi\cdot n}\cdot e^{(\chi-\log \eta)\cdot N}\\
    &\leq C_0 \cdot e^{-\chi_0\cdot n},
    \end{aligned}
    $$
    where $C_0:=e^{(\chi-\log\eta)\cdot N}$ and inequality $(a)$ follows from the fact that $\sum_{i=0}^{n-1}\varphi (\hat{f}^i\hat{x})\geq - \chi\gamma\cdot n$.
    Note that $\gamma$ was taken small so that $-(1-\frac{\gamma}{2}-\gamma)\cdot \chi<-\chi_0$. Meanwhile, take $n=1$ in item $1$ of properties of $\hat{H}$ we see we may enlarge $C_0$ so that $C_0^{-1}=\gamma$ if necessary to conclude $|f'(x_0)|\geq C_0^{-1}$. In this way, $\hat{H}\subseteq P_{C_0,\chi_0}$. By our assumption $\gamma<\frac{1-\tau}{3}$, we get for all $\hat{\mu}\in \mathbb{P}^h(\hat{f})$, $\hat{\mu}(P_{C_0},\chi_0)\geq \hat{\mu}(\hat{H})>1-3\gamma>\tau$.
\end{proof}

\subsubsection{From Pliss set to Pesin blcok}

Given an invertible topological system $T:X\rightarrow X$, a function $\Pi: X\rightarrow (0,\infty)$ and $\varepsilon>0$, the $\varepsilon$-tempered envelope of $\Pi$ is the function
$$
\Pi_{\varepsilon}(x):= \sup_{n\in \mathbb{Z}} \{ e^{-|n|\varepsilon}\Pi(T^{n} x) \}.
$$

\begin{lemma}\label{temper}
Consider an invertible topological system $(X,T)$, $\nu\in \mathbb{P}_{\rm{erg}}(T)$. Let $\Pi: \Omega\rightarrow (0,\infty)$ be a measurable function, and $\varepsilon>0$. Then
\begin{itemize}
\item[1.] if $\log\frac{\Pi\circ T}{\Pi}\in L^1(\nu)$, then $\Pi_{\varepsilon}$ is finite for $\nu$-a.e. $x\in X$,
\item[2.] if $\log\frac{\Pi\circ T}{\Pi}\in L^1(\nu)$, for any $\varepsilon<c_0<1$, for any $t>0$,
$$
\begin{aligned}
\nu(\{ x\in \Omega: \Pi_{\varepsilon}(x)>t \})\leq &4\max\{ \frac{c_0}{\varepsilon},1 \}\cdot \nu(\{ x\in X: \Pi(x)>t \})\\&+\frac{4}{\varepsilon}\int_{\{\log\frac{\Pi\circ T}{\Pi}>c_0\}} \log\frac{\Pi\circ T}{\Pi} d\nu.
\end{aligned}
$$
\end{itemize}
\end{lemma}
\begin{proof}
    The proof of item 1 is exactly the same as the corresponding item in Lemma 2.20, \cite{BCS3}.

    Now we proceed to prove item 2. Observe that $\Pi_{\varepsilon}(x)=\max\{ \Pi^{-}_{\varepsilon}(x), \Pi^{+}_{\varepsilon}(x) \}$, where
    $$
    \Pi^{-}_{\varepsilon}(x):=\sup_{n\geq 0} \{ e^{-n\varepsilon}\Pi(T^n x) \}, \Pi^{+}_{\varepsilon}(x):=\sup_{n\geq 0} \{ e^{-n\varepsilon}\Pi(T^{-n} x) \}.
    $$
    In the remainder of the proof, we show that given $c_0>0$,
    \begin{equation}
    \nu(\{\ x: \Pi^{-}_{\varepsilon}(x)>t\})\leq \frac{2c_0}{\varepsilon}\nu(\{ x: \Pi(x)>t \})+\frac{2}{\varepsilon}\int_{\{\log\frac{\Pi\circ T}{\Pi}>c_0\}} \log\frac{\Pi\circ T}{\Pi} d\nu.
    \label{eqn:5.3.1}
    \end{equation}
    Indeed, the following proof gives the same inequality for $\Pi^{+}_{\varepsilon}(x)$. We also assume that $\nu(\{x:\Pi^{-}_{\varepsilon}(x)>t\})\neq 0$, since otherwise $\eqref{eqn:5.3.1}$ holds trivially.

    Poincar\'e recurrence theorem implies that we may assume without loss of generality that $x$ is a point such that $\Pi^{-}_{\varepsilon}(T^nx)>t$ for infinitely many iterates $n>0$. We make the following definitions. 
    \begin{itemize}
    \item[i.] $n_1(x):=\min \{ n\geq 0:  \Pi^{-}_{\varepsilon}(T^nx)>t\}$, 
    $$
    l_1:=\min\{ l\geq 0: e^{-l\varepsilon}\Pi(T^{n_1+l}x)>t \};
    $$
    \item[ii.] then, inductively, let
    $$
    \begin{aligned}
    n_{i+1}(x)&:=\min \{ n\geq n_i+l_i: \Pi^{-}_{\varepsilon}(T^nx)>t \}\\
    l_{i+1}(x)&:=\min\{ l\geq 0: e^{-l\varepsilon}\Pi(T^{n_{i+1}+l}x)>t \}.
    \end{aligned}
    $$
    \end{itemize}

    \bigskip

    {\bf Claim.} For one such $x$,
    \begin{itemize}
    \item[1.] $\{ n: \Pi^{-}_{\varepsilon}(T^n x)>t \}\subseteq \bigsqcup^{\infty}_{i=1} [\![n_i,n_i+l_i]\!]$;
    \item[2.] $\{ n: \Pi(T^n x)>t \}\supset \bigsqcup^{\infty}_{i=1} [\![n_i+l_i-\max\{\frac{\varepsilon l_i}{c_0}-g_i(x),0\},n_i+l_i]\!]$, where
    $$
    g_i:= \frac{1}{c_0} \sum_{k\in [\![n_i+l_i-\frac{\varepsilon l_i}{c_0},n_i+l_i]\!], \log\frac{\Pi\circ T}{\Pi}(T^k x)>c_0} \left(\log\frac{\Pi\circ T}{\Pi}(T^k x)-c_0 \right).
    $$
    \end{itemize}

    \begin{proof}[Proof of the Claim]

    Assertion $1$ of the claim is obvious.

    Now we prove assertion $2$. By the definition of $l_i$, $\Pi(T^{n_i+l_i}x)>e^{\varepsilon l_i}t>t$. Define $L_i$ by
    \begin{equation}
        \label{eqn:L}
        L_i:=\max\{ L\in[\![0,l_i]\!]:\frac{\Pi(T^{n_i+l_i}x)}{\Pi(T^{n_i+l_i-L}x)}<e^{\varepsilon l_i} \}.
    \end{equation}
    One sees from definition that for all $n\in [\![n_i+l_i-L_i,n_i+l_1]\!]$, $ \Pi(T^n x)>t$ by our distortion, that is,
    $$
    \{ n: \Pi^{-}_{\varepsilon}(T^n x)>t \}\supset [\![n_i+l_i-L_i,n_i+l_i]\!].
    $$
    Hence, what remains is to shrink the integer interval $[\![n_i+l_i-L_i,n_i+l_i]\!]$ until we reach the desired integer interval in assertion $2$.

    Define $L'_i$ by the following property: the maximal integer $L\in [\![0,l_i]\!]$ satisfying
    \begin{equation}
    c_0\cdot L + \sum_{k=n_i+l_i-L}^{n_i+l_i}1_{\log\frac{\Pi\circ T}{\Pi}(T^{k} x)>c_0} \left(\log\frac{\Pi\circ T}{\Pi}(T^{k_1} x)-c_0\right) -\varepsilon l_i<0.
    \label{eqn:3.3}
    \end{equation}
    One sees that $L'_i\leq L_i$: take $0\leq L\leq L'_i$,
    $$
    \begin{aligned}
        \log\frac{\Pi(T^{n_i+l_i}x)}{\Pi(T^{n_i+l_i-L}x)}&=\sum_{k\in [\![n_i+l_i-L,n_i+l_i]\!]}\log\frac{\Pi\circ T}{\Pi}(T^kx)\\
        &\overset{(a)}{\leq} \sum_{k\in [\![n_i+l_i-L,n_i+l_i]\!]}\log\frac{\Pi\circ T}{\Pi}(T^kx)+c_0\cdot (L'_i-L)\\
        &\leq \sum_{k\in [\![n_i+l_i-L'_i,n_i+l_i]\!], \log\frac{\Pi\circ T}{\Pi}(T^kx)<c_0}c_0+\sum_{k, \log\frac{\Pi\circ T}{\Pi}(T^kx)>c_0}\log\frac{\Pi\circ T}{\Pi}(T^kx)\\
        &=c_0\cdot L'_i+\sum_{k\in [\![n_i+l_i-L'_i,n_i+l_i]\!], \log\frac{\Pi\circ T}{\Pi}(T^kx)>c_0}\left(\log\frac{\Pi\circ T}{\Pi}(T^kx)-c_0\right)\\
        &<\varepsilon l_i.
    \end{aligned}
    $$
    Inequality $(a)$ follows from $c_0>0$.

    Finally, define $L''_i$ by the following property: the maximal integer $L\in [\![0,l_i]\!]$ satisfying
    \begin{equation}
    c_0\cdot L-\varepsilon l_i+\sum_{k=n_i+l_i-\frac{\varepsilon l_i}{c_0}}^{n_i+l_i}(1_{\log\frac{\Pi\circ T}{\Pi}(T^{k} x)>c_0} \log\frac{\Pi\circ T}{\Pi}(T^{k_1} x)-c_0)<0.
    \label{eqn:3.4}
    \end{equation}
    One sees that by definition $\frac{\varepsilon l_i}{c_0}>L''_i$, 
    so
    $$
    \begin{aligned}
    &c_0\cdot L''_i-\varepsilon l_i+\sum_{k=n_i+l_i-L''_i}^{n_i+l_i}(1_{\log\frac{\Pi\circ T}{\Pi}(T^{k} x)>c_0} \log\frac{\Pi\circ T}{\Pi}(T^{k_1} x)-c_0)\\
    &\leq c_0\cdot L''_i-\varepsilon l_i+\sum_{k=n_i+l_i-\frac{\varepsilon l_i}{c_0}}^{n_i+l_i}(1_{\log\frac{\Pi\circ T}{\Pi}(T^{k} x)>c_0} \log\frac{\Pi\circ T}{\Pi}(T^{k_1} x)-c_0)<0.
    \end{aligned}
    $$
    Thus, by our requirement on $L'_i$ being maximal satisfying $\eqref{eqn:3.4}$, $L''_i\leq L'_i\leq L_i$. Notice that by definition,
    $$
    \begin{aligned}
    L''_i&=\frac{\varepsilon l_i}{c_0}-\left(\frac{\log\frac{\Pi\circ T}{\Pi}(T^{k_1} x)}{c_0}-1\right)^+-\cdots-\left(\frac{\log\frac{\Pi\circ T}{\Pi}(T^{k_s} x)}{c_0}-1\right)^+\\
    &=\frac{\varepsilon l_i}{c_0}-g_i,
    \end{aligned}
    $$
    where $A^+:=\max\{ A,0 \}$. Note that the assumption $c_0>\varepsilon$ guarantees $n_i+l_i-(\frac{\varepsilon l_i}{c_0}-g_i)>n_i$. We take $\max\{ \frac{\varepsilon l_i}{c_0}-g_i,0 \}$ to guarantee that $[\![n_i+l_i-\max\{\frac{\varepsilon l_i}{c_0}-g_i(x),0\},n_i+l_i]\!]$ is a well-defined integer interval.
    \end{proof}

    Let's return to the proof of item $2$ of the lemma. Since $\nu$ is $T$-ergodic, we may take $x\in X$ such that
    $$
    \begin{aligned}
    &1. \ \lim_{N\rightarrow \infty}\frac{1}{N}\sum_{k=0}^{N-1}1_{\{ \Pi^{-}_{\varepsilon}(x)>t \}}(T^kx)=\nu(\{ \Pi^{-}_{\varepsilon}(x)>t \}),\\
    &2. \ \lim_{N\rightarrow \infty}\frac{1}{N}\sum_{k=0}^{N-1}1_{\{ \Pi (x)>t \}}(T^kx)=\nu(\{ \Pi (x)>t \}),\\
    &3.\  \lim_{N\rightarrow \infty}\frac{1}{N}\sum_{k=0}^{N-1} 1_{\{ \log \frac{\Pi\circ T}{\Pi}( x)>c_0\}} \log \frac{\Pi\circ T}{\Pi}(T^k x)=\int_{\{ \log \frac{\Pi\circ T}{\Pi}( x)>c_0\}} \log \frac{\Pi\circ T}{\Pi}( x) d\nu.
    \end{aligned}
    $$
    Note that this third item is possible since $\log \frac{\Pi\circ T}{\Pi}( x)\in L^1(\nu)$ by assumption.

    For this $x$, abbreviate $n_i=n_i(x)$, $l_i=l_i(x)$ and $g_i=g_i(x)$. Then
    $$
    \begin{aligned}
    \nu(\{ \Pi^{-}_{\varepsilon}(x)>t \})&\overset{(a)}{=}\lim_{N\rightarrow \infty}\frac{1}{N} \sharp\{ 1\leq k\leq N: \Pi^{-}_{\varepsilon}(T^kx)>t \}\\
    &\overset{(b)}{\leq} \lim_{i\rightarrow \infty} \frac{1}{n_i+l_i} \sum_{j=1}^i (l_j+1)\\
    &\leq \frac{2c_0}{\varepsilon }\limsup_{i\rightarrow \infty} \frac{1}{n_i+l_i} \sum_{j=1}^i (\frac{\varepsilon l_j}{c_0})+1\\
    &\leq \frac{2c_0}{\varepsilon }\limsup_{i\rightarrow \infty} \frac{1}{n_i+l_i} \sum_{j=1}^i (\max\{\frac{\varepsilon l_j}{c_0}-g_i,0\}+1)+\frac{2}{\varepsilon}\limsup_{i\rightarrow \infty} \frac{1}{n_i+l_i}\sum_{j=1}^i  c_0\cdot g_i\\
    &\overset{(c)}{\leq} \frac{2c_0}{\varepsilon} \lim_{N\rightarrow \infty }\frac{1}{N}\sum_{k=0}^{N_1}1_{\{ \Pi (x)>t \}}(T^kx)+\frac{2}{\varepsilon}\lim_{N\rightarrow \infty}\frac{1}{N}\sum_{k=0}^{N-1} 1_{\{ \log \frac{\Pi\circ T}{\Pi}( x)>c_0\}} \log \frac{\Pi\circ T}{\Pi}(T^k x)\\
    &\overset{(d)}{=}\frac{2c_0}{\varepsilon}\nu(\{ x: \Pi(x)>t \})+\frac{2}{\varepsilon}\int_{\log\frac{\Pi\circ T}{\Pi}>c_0} \log\frac{\Pi\circ T}{\Pi} d\nu.
    \end{aligned}
    $$
    Equality $(a)$ follows from property $1$ of our chosen $x$. Inequality $(b)$ follows from item $1$ of the claim. Inequality $(c)$ follows from item $2$ of the claim. Equality $(d)$ follows from property $2$ and $3$ of our chosen $x$.
\end{proof}

\begin{proposition}\label{Pesinblock}
Given a Pliss set $P_{\chi,n_0}$, for all $\gamma>0$ and $\varepsilon>0$, there exist $C>0, c_0>0$, $0<h<h_{\rm{top}}(f)$, and  a Pesin block $\hat{\Lambda}_{C,\varepsilon,\chi}$, such that for all $\hat{\nu}\in \mathbb{P}^h_{\rm{erg}}(\hat{f})$,
$$
\hat{\nu}(\hat{\mathcal{M}}\setminus \hat{\Lambda}_{C,\varepsilon,\chi})\leq 4\max\{ \frac{c_0}{\varepsilon},1 \}\cdot \hat{\nu}(\hat{\mathcal{M}}\setminus P_{C,\chi,n_0})+\gamma.
$$
\end{proposition}
\begin{proof}
    We prove the proposition by applying Lemma \ref{temper} for a suitable function $\Pi$. 

    First fix some $\gamma>0$ and $0<\varepsilon<1$.

    Recall the definition of $\operatorname{NUE}_{\chi}$ via $\eqref{eqn:NUE}$. We define a function $\Pi: \operatorname{NUE}_{\chi}\rightarrow (0,\infty)$ by
    $$
    \Pi(\hat{x}):=\sup_{k\geq 0} \{ ||d_{\hat{x}}\hat{f}^{-k} ||\cdot e^{\chi\cdot k} \}.
    $$
    Note that 
    $$
    \begin{aligned}
    \Pi\circ\hat{f}(\hat{x})&=\sup_{k\geq 0} \{ ||d_{\hat{f}\hat{x}}\hat{f}^{-k} ||\cdot e^{\chi\cdot k} \}\\
    &\leq\sup\left(\{ m(d_{x_0}f)^{-1}||d_{\hat{x}}\hat{f}^{-k+1}||\cdot e^{\chi\cdot(k-1+1)}: k\geq 1 \}\cup\{1\}\right)\\
    &=m(d_{x_0}f)^{-1}e^{\chi}\cdot \sup_{k\geq 0} \{ ||d_{\hat{x}}\hat{f}^{-k} ||\cdot e^{\chi\cdot k} \}=m(d_{x_0}f)^{-1}e^{\chi}\cdot \Pi(\hat{x}).
    \end{aligned}
    $$
    Similarly, we have
    $$
    \Pi(\hat{x}) \leq ||d_{x_0} f || e^{-\chi} \Pi(\hat{f}\hat{x}) \leq || d f ||_{\infty} e^{-\chi} \Pi(\hat{f}\hat{x}).
    $$
    This gives $\chi - \log ||df ||_{\infty} \leq g:=\log \frac{\Pi\circ \hat{f}}{\Pi} \leq \chi-\log m(d_{x_0}f)$.
    Apply the arguments in the proof of Proposition \ref{prop:uniformPliss} (the bound of $\int -\log m(d_xf)d\nu$ by $\int -\log|\operatorname{Jac}(x)|d\nu$), we get that uniform integrability of $(\mathcal{M},f)$ (Corollary \ref{cor:UI}, recall that $(\mathcal{M},f)$ is always assumed to be SEE) implies that there exist $c>0, 0<h<h_{\rm{top}}(f)$ such that for all $\hat{\nu}\in \mathbb{P}_{\rm{erg}}^{h}(\hat{f})$,
    \begin{equation}
    \label{eqn:UISPR}
    \frac{2}{\varepsilon}\int_{\{\hat{x}:g>c\}} g d\hat{\nu}<\frac{\gamma}{4}.
    \end{equation}
    In particular, $g\in L^1(\hat{\mu})$ for all $\hat{\nu}\in \mathbb{P}^{h}_{\rm{erg}}(\hat{f})$.

    Define $C:=\max_{0\leq k\leq n_0}\{ ||d_{\hat{x}}f^{-k}|| : \hat{x}\in P_{\chi,n_0}\}$, note that by definition of $P_{\chi,n_0}$, $C<\infty$. Moreover,
    $$
    P_{\chi,n_0}\subseteq \{ \hat{x}: \Pi(\hat{x})\leq C \}.
    $$

    Now consider the $\varepsilon$-envelope of $\Pi$. One observes that the Pesin block can be identified as follows
    $$
    \Lambda_{C,\varepsilon,\chi}=\{\hat{x}: \Pi_{\varepsilon}(\hat{x})\leq C  \}.
    $$
    By item $1)$ of Lemma \ref{temper}, $\Pi_{\varepsilon}(\hat{x})<\infty$ for $\hat{\mu}$-a.e. $\hat{x}$.    

    Then, for all $\nu\in \mathbb{P}^h_{\rm{erg}}(f)$,
    $$
    \begin{aligned}
    \hat{\nu}(\hat{\mathcal{M}}\setminus\Lambda_{C,\varepsilon,\chi})&=\hat{\nu}(\operatorname{NUE}_{\chi}\setminus \Lambda_{C,\varepsilon,\chi})\\
    &=\hat{\nu}(\{ \hat{x}\in \operatorname{NUE}_{\chi}: \Pi_{\varepsilon}>C \})\\
    \overset{(a)}{&}\leq 4\max\{ \frac{c_0}{\varepsilon},1 \}\hat{\nu}( \{ \hat{x}\in \operatorname{NUE}_{\chi}: \Pi(\hat{x})>C \})+\gamma\\
    &\overset{(b)}{\leq} 4\max\{ \frac{c_0}{\varepsilon},1 \}\hat{\nu}( \hat{\mathcal{M}}\setminus P_{C, \chi,n_0})+\gamma.
    \end{aligned}
    $$
    Inequality $(a)$ follows from item $2)$ of Lemma \ref{temper} and $\eqref{eqn:UISPR}$. Inequality $(b)$ follows from the $\chi$-expanding of $\mu$, guaranteed by Lemma \ref{lem:chi-expan}, and the fact that $\hat{\nu}(\operatorname{NUE}_{\chi})=1$ for all $\nu\in \mathbb{P}_{\rm{erg}}(f)$, $\chi$-expanding (\cite{LDP}).
    
    This finishes the proof of our proposition.
\end{proof}

With all the preparations, we may proceed to the proof of Theorem \ref{thm:main1}.

\begin{proof}[Proof of Theorem \ref{thm:main1}]
    Let $\chi_0 > 0$ be given by Proposition \ref{prop:uniformPliss}.
    
    Then, take one $\gamma>0$ and $\varepsilon>0$.
    For this $\gamma, \varepsilon$, Proposition \ref{Pesinblock} gives $C>0$, $c_0>0$ and $0<h_0<h_{\rm{top}}(f)$. 

    Now take $\tau$ such that $1-\tau<\frac{\gamma}{8}\min\{ \frac{\varepsilon}{c_0}, 1 \}$. 

    With respect to this $\tau$, Proposition \ref{prop:uniformPliss} shows that there exist $C_1>0$, $0<h_1<h_{\rm{top}}(f)$ and an expanding set $P_{C_1,\chi_0}$ such that for all $\hat{\nu}\in \mathbb{P}^{h_1}_{\rm {erg}}(f)$,
    $$
    \hat{\nu}(P_{C_1,\chi_0})>\tau.
    $$

    Define $n_0:=\frac{\log C_1}{\chi_0-\chi}$. Then $P_{C_1,\chi_0}\subseteq P_{C_1,n_0,\chi}$.

    Denote $h:=\max\{ h_0,h_1 \}$. Finally, we may apply Proposition \ref{Pesinblock} with error being $\frac{\gamma}{2}$, to see that for all $\hat{\nu}\in \mathbb{P}^h_{\rm{erg}}(f)$,
    $$
    \begin{aligned}
    1-\hat{\nu}(\Lambda_{C,\varepsilon,\chi})&\leq 4\max\{ \frac{c_0}{\varepsilon},1 \}(1-\hat{\nu}(P_{C_1,n_0,\chi}))+\frac{\gamma}{2}\\
    &\leq \frac{\gamma}{2}+\frac{\gamma}{2}=\gamma,
    \end{aligned}
    $$
    where the second inequality follows from the choice of $\tau$. Since $\gamma$ is arbitrarily small, $\tau$ can be chosen arbitrarily close to $1$.
    \end{proof}

\section{Implications of SPR}\label{sec:3}

In this section we prove the following theorems, which will conclude the first item of Theorem \ref{thm:main2}.

\begin{theorem}\label{thm:EM}
    Let $(\mathcal{M},f)$ be an $C^r$ expanding-SPR system, $r>1$, and is topologically transitive. Then 
    \begin{itemize}
        \item[1.] there exist finitely many measures of maximal entropy of $f$,
        \item[2.] there exists $q>0$ such that for any $\mu$, MME of $f$, each $f^q$-ergodic component $\mu'$ is $f^q$-exponentially mixing. If $f$ is topologically mixing, then $q=1$. 
    \end{itemize}
\end{theorem}




To conclude Theorem \ref{thm:EM}, we will maneuver through the symbolic dynamics. Heuristically, the idea could be summarized as follows. We consider the symbolic model which encodes the properties of $(\M, f)$ we are interested in (Theorem \ref{thm:coding-LDP}). Then, we show that the expanding-SPR property implies good statistical behavior of this symbolic model. Finally, we translate back these properties from the symbolic model to $(\mathcal{M},f)$ (concluding Theorem \ref{thm:EM}). 

We first review certain definitions in symbolic dynamics. Let $\mathcal{G}=(\mathcal{V},\mathcal{E})$ be a directed graph, with $\mathcal{V}$ the set of vertices being countable; $\mathcal{E}$ the set of edges. We denote an edge connecting $v,u\in \mathcal{V}$ by $v\rightarrow u$. The Markov shift generated by $\mathcal{G}$ is a pair $(\Sigma,\sigma)$, where
$$
\Sigma:=\{ \underline{v}:=(v_k)_{k\in \mathbb{Z}}\in \mathcal{V}^{\mathbb{Z}}: v_k\rightarrow v_{k+1} \},
$$
and $\sigma$ is the left shift. We denote by $\Sigma^{\#}$ the subset of $\Sigma$, consisting of $\underline{v} = (v_k)$ such that there exist $u_0,w_0\in \mathcal{V}$ and subsequences $(v_{n_k})$, $n_k\uparrow \infty$, $(v_{n_l})$, $n_l\downarrow -\infty$ with $v_{n_k}=u_0$, $v_{n_l} = w_0$ for every $k,l$. An irreducible component of $(\Sigma,\sigma)$ is the Markov shift $(\Sigma',\sigma)$ generated by the maximal subset $\mathcal{V}'$ with
$$
\forall v,v'\in \mathcal{V}', \exists \underline{v}\in \Sigma, n\geq 1 \text{ such that } v_0=v, v_n=v'.
$$


\begin{definition}\label{def:1-sSPR}[SPR for Markov shifts]
    A Markov shift $(\Sigma,\sigma)$ is \emph{strongly positively recurrent}, if there exists a finite number of cylinders $C(v):=\{ \underline{v}: v_0=v \}$ and $h>0$, $\tau>0$, such that for all $\mu\in \mathbb{P}_{\rm{erg}}(\sigma)$ with $h(\sigma,\mu)>h$,
    $$
    \mu(\bigcup C(v))>\tau.
    $$
\end{definition}

\begin{theorem}[\cite{BCS3}, Theorem B.12]\label{thm:Markov}
    Let $(\Sigma,\sigma)$ be an irreducible SPR Markov shift. Then
    \begin{itemize}
        \item[1.] there exists a unique measure $\mu\in \mathbb{P}_{\rm{erg}}(\sigma)$ such that $h(\mu)=\sup_{\nu\in \mathbb{P}_{\rm{erg}}(\sigma)}\{ h(\nu) \}$, which we call a measure of maximal entropy,
        \item[2.] the measure of maximal entropy $\mu$ is Bernoulli,
        \item[3.] the measure of maximal entropy $\mu$ is exponentially mixing.
    \end{itemize}
\end{theorem}

Finally, recall that a measure $\mu\in \mathbb{P}_{\rm{erg}}(\sigma)$ is an MME if $h(\mu)=\sup\{ h(\nu): \nu\in \mathbb{P}_{rm{erg}}(\sigma) \}$.

For general reviews on Markov shifts and their SPR theory, see \cite{SMarkov}, \cite{S3}, part II of \cite{BCS3}.

\bigskip



\subsection{SPR coding: statements and comparisons 
}\label{sec:SPRstate}

The symbolic model $\left(\hat{\Sigma}^{\#},\sigma\right)$ that we will use is a variation of the one in \cite{LDP}. This process of building symbolic models to encode complicated geometric systems by symbolic ones, usually called a \emph{coding}, has a long history which could be traced back to \cite{Bo} in the 1970s. Usually, as the coding is not bijective, topological mixing of $(\mathcal{M},f)$ does not imply irreducibility of $(\hat{\Sigma}^{\#},\sigma)$. So as a first step, we will have to identify the invariant subsets in $(\mathcal{M},f)$ corresponding to irreducible components of $(\hat{\Sigma}^{\#},\sigma)$. For uniformly hyperbolic systems, they are homoclinic classes \cite{Bowen}. For non-uniformly hyperbolic systems, they are topological homoclinic classes for measures \cite{BCS1} \cite{LDP} \cite{Li2}.

\bigskip

Throughout the remainder of this section, we will be considering the system $(\mathcal{M},f)$ with $f$ a $C^r$-map, $r>1$, and $h_{\rm{top}}(f)>0$.

Take $0<\chi<h_{\rm{top}}(f)$. Recall the definition of $\operatorname{NUE}_{\chi}\subseteq \hat{\mathcal{M}}$ via $\eqref{eqn:NUE}$. For each $\hat{x}\in \operatorname{NUE}_{\chi}$, there exist well-defined local stable/unstable sets $W^{s/u}_{\rm{loc}}(\hat{x})\subseteq \hat{\mathcal{M}}$. Since $\operatorname{NUE}_{\chi}$ consists of points $\hat{x}$ which are $\chi$-non-uniformly expanding in all direction, we have $W^u_{\rm{loc}}(\hat{x})$ being homeomorphic to a $d$-disk at $x_0\in \M$; and $W^s_{\rm{loc}}(\hat{x})=\vartheta_0^{-1}(x_0)$ \cite{LDP} \cite{Li2}.

\begin{definition}\label{def:2-u/ssets}
    The stable/unstable set at $\hat{x}\in \operatorname{NUE}_{\chi}$ is defined to be
    $$
    \begin{aligned}
        & W^s(\hat{x}):=\{ \hat{y}\in \hat{\M}: \limsup_{n\rightarrow\infty}\frac{1}{n}\log d(y_n,x_n)<0 \},\\
        & W^u(\hat{x}):=\{ \hat{y}\in \hat{\M}: \limsup_{n\rightarrow\infty}\frac{1}{n}\log d(y_{-n},x_{-n})<0 \}.
    \end{aligned}
    $$
\end{definition}

Then, we have
$$
    \begin{aligned}
        & W^s(\hat{x})=\bigcup_{n\geq 0} \hat{f}^{-n}\left(W^s_{\rm{loc}}(\hat{f}^n\hat{x})\right),\\
        & W^u(\hat{x})=\bigcup_{n\geq 0} \hat{f}^n \left( W^u_{\rm{loc}}(\hat{f}^{-n}\hat{x})\right).
    \end{aligned}
$$
We remark that $W^s(\hat{x})=\{ \hat{y}=(y_n)_{n\in \mathbb{Z}}: \exists n\geq 0 \text{ such that } y_n=x_n \}$.

With well-defined stable/unstable sets, we define homoclinic relation between measures.

\begin{definition}\label{def:homoclinic}
Given $\mu_1,\mu_2\in \mathbb{P}_{\rm{erg}}(f)$, hyperbolic measures, we say that $\mu_1$ \emph{precedes} $\mu_2$, denoted by $\mu_1\overset{h}{\leq}\mu_2$, if there are measurable sets $\hat{A}_1,\hat{A}_2\subseteq \hat{\M}$, $\hat{\mu_1}(\hat{A}_1)>0$, $\hat{\mu}_2(\hat{A}_2)>0$, such that for all $(\hat{x}_1,\hat{x}_2)\in \hat{A}_1\times \hat{A}_2$, there exists $\hat{z}\in \hat{\M}$ such that
\begin{itemize}
\item[1.]\emph{(intersection)} $\hat{z}\in W^u(\hat{x}_1)\cap W^s(\hat{x}_2)$,
\item[2.]\emph{(transverse)} $\hat{z}\notin \bigcup_{n\in \mathbb{Z}} \hat{f}^n \vartheta^{-1}(C(f))$. 
\end{itemize}
We say {that} two hyperbolic measures $\mu_1, \mu_2\in \mathbb{P}_{\rm{erg}}(f)$ are homoclinic related, denoted by $\mu_1\overset{h}{\sim}\mu_2$, if $\mu_1\overset{h}{\leq}\mu_2$ and $\mu_2\overset{h}{\leq}\mu_1$.
\end{definition}

We briefly comment again how this definition differs from those in \cite{BCS1} \cite{BCS2} and \cite{LDP}. For diffeomorphisms, hyperbolic points admit non-degenerate directions that are non-uniformly contracting. In our setup, points in $\operatorname{NUE}_{\chi}$ expand in all directions, making stable sets 'degenerate'. So the transverse condition is replaced here by 'non-criticality' (item 2 of Definition \ref{def:homoclinic}). This is again the analogy between 'criticality' and 'non-transverse' we mentioned in Section \ref{sec:1-idea}.

\begin{proposition}[Proposition 4 \cite{LDP}]
    The homoclinic relation on measures is an equivalence relation.
\end{proposition}

\begin{definition}\label{def:1-uHC}
    Given $\mu\in \mathbb{P}_{\rm{erg}}(f)$, hyperbolic, the topological homoclinic class of $\mu$ is the set
    $$
    \operatorname{HC}(\mu):=\overline{\bigcup\{ \operatorname{supp}(\hat{\nu}): \nu\in \mathbb{P}_{\rm{erg}}(f), \text{ hyperbolic, }\nu\overset{h}{\sim} \mu \}}.
    $$
\end{definition}

Recall that $n$ is the period of a measure $\mu\in \mathbb{P}(f)$ if there exists $A\neq \emptyset$ such that $f^nA=A$, $f^jA\cap A=\emptyset$, $j=1,\cdots, n-1$ and $\mu(\bigcup_{j=0}^{n-1}A)=1$. When it exists, the largest period is \emph{the period of} $\mu$. When $(\mathcal{M},f)$ is topologically mixing, the period of $\mu$ is $1$.

With this preparation, we state the symbolic model that we will use. We first introduce a parameter that will be used later. For $\hat{x}\in \operatorname{NUH}_{\chi}$, one may define a linear map $C_{\chi}(\hat{x}):\mathbb{R}^d\rightarrow T_{\hat{x}}\hat{\mathcal{M}}$, such that 
\begin{equation}
\label{eqn:C_chi}
||C_{\chi}(\hat{x})^{-1}||=\left(\max_{v\in T_{\hat{x}}\hat{\mathcal{M}},||v||=1} 2\cdot\sum_{n=0}^\infty e^{2n\chi}||d_{\hat{x}}\hat{f}^{-n}v||^2\right)^{\frac{1}{2}}.
\end{equation}
Define $Q_{\varepsilon}(\hat{x}):=\varepsilon^{\frac{6}{\beta}}\min\{ ||C_{\chi}(\hat{x})^{-1}||^{-\frac{48}{\beta}}, d(\{ x_{-1},x_0,x_1 \}, C(f))^{\frac{96a}{\beta}} \}$, where $a$ only depends on $f$ and comes from conditions $(A1)-(A7)$ in \cite{LDP}, which we explain in the following Remark \ref{re-Acondi}. Then define
\begin{equation}
\label{eqn:q}
q_{\varepsilon}(\hat{x}):=\varepsilon\cdot \min\{ e^{\varepsilon|n|} Q_{\varepsilon}(\hat{f}^n\hat{x}): n\leq 0 \}.
\end{equation}
For details regarding these parameters, see \cite{ALP}, Section 3.

\begin{theorem}[Theorem A, Theorem 3.1, Theorem 4.2 \cite{LDP}]\label{thm:coding-LDP}
Let $f:\mathcal{M} \rightarrow \mathcal{M}$ be a $C^r$-map, $r>1$, $h_{\rm{top}}(f)>0$, with singularities and satisfying conditions $(A1)$-$(A7)$ in \cite{LDP}. For every $\chi>0$, there exist a locally compact countable state (two-sided) Markov shift $(\hat{\Sigma},\sigma)$, induced by a graph $(\mathcal{R}, \mathcal{E})$, and a H\"older continuous map $\hat{\pi}: \hat{\Sigma}\rightarrow \hat{\mathcal{M}}$ such that $\hat{f}\circ \hat{\pi}=\hat{\pi}\circ \sigma$ and:
\begin{itemize}
\item[1.] the restriction of $\hat{\pi}$ to $\hat{\Sigma}^{\#}$ is finite-to-one,
\item[2.] for all $\mu\in \mathbb{P}_{\rm{erg}}(f)$, $\chi$-expanding, $\hat{\mu}(\hat{\pi}(\hat{\Sigma}^{\#}))=1$, and there exists $\overline{\mu}\in \mathbb{P}_{\rm{erg}}(\sigma)$, such that $\hat{\pi}_{\ast}\overline{\mu}=\hat{\mu}$, $h_{\sigma}(\overline{\mu})=h_f(\mu)$; for all $\overline{\nu}\in \mathbb{P}_{\rm{erg}}(\sigma)$, $\hat{\nu}:=\hat{\pi}_{\ast}\overline{\nu}$ is $\chi/3$-expanding, $h_{\sigma}(\overline{\nu})=h_{\hat{f}}(\hat{\nu})$,
\item[3.] let $\hat{K}\subseteq \operatorname{NUE}_{\chi}$ and $\inf\{ \hat{x}\in \hat{K}: q_{\varepsilon}(\hat{x}) \}\geq q_0>0$, then there exist finitely many cylinders $C(v)$ such that $\hat{K}\subseteq\hat{\pi}(\bigcup C(v))$,
\item[4.] take $\mu\in \mathbb{P}_{\rm{erg}}(f)$, $\chi$-expanding, then there exists $\hat{\Sigma}'$, an irreducible component of $\hat{\Sigma}$, such that 
    \begin{itemize}
        \item[4.a.] for all $\nu\in \mathbb{P}_{\rm{erg}}(f)$, $\nu\overset{h}{\sim}{\mu}$, there exists $\overline{\nu}$, supported on $\hat{\Sigma}'$, such that $\hat{\pi}_{\ast}\overline{\nu}=\hat{\nu}$, $h_{\sigma}(\overline{\nu})=h_f(\hat{\nu})$,
        \item[4.b.] for all $\overline{\nu}\in \mathbb{P}_{\rm{erg}}(\sigma)$, supported on $\hat{\Sigma}'$, $\hat{\pi}_{\ast}\overline{\nu}$ is a lift to natural extension of $\nu\in \mathbb{P}_{\rm{erg}}(f)$, and $\nu\overset{h}{\sim} \mu$,
        \item[4.c.] there exists at most one MME on $(\hat{\Sigma}',\sigma)$. When it exists, the MME is Bernoulli times a finite rotation with period equal to the period of $(\hat{\Sigma}',\sigma)$.
    \end{itemize}
\end{itemize}
\end{theorem}


\begin{remark}\label{re-Acondi}
    Additional assumptions considering the critical set are made in \cite{LDP} ($(A1)$-$(A7)$). This is partly due to the fact that singularities are allowed in \cite{LDP}. In a work under preparation \cite{BLY}, we build a coding as in Theorem \ref{thm:coding-LDP}, removing the assumptions $(A1)$-$(A7)$ altogether. We note that Viana maps all satisfy the additional regularity conditions \cite{LDP}, so imposing them will not affect Theorem \ref{thm-Viana}. It will bring additional conditions to Theorem \ref{thm:main1}. We will proceed to the remainder of the proof with these conditions. We hope the forthcoming paper \cite{BLY} will be ready soon enough so that we will refer to it instead of \cite{LDP} for Theorem \ref{thm:coding-LDP}.
\end{remark}

Our definition of Pesin blocks (Definition \ref{def:Pesinblock1}) are slightly different to that in \cite{LDP}, which they denote by $\Lambda_{C',\varepsilon',\chi,l}$, in Section 4.2. We will conclude this section by showing that the two definitions are compatible.

\begin{lemma}\label{lem:discon}
    Let $\M$ be a closed Riemannian manifold and $f:\M\rightarrow \M$ be a $C^{1+\beta}$-map, $\beta > 0$. Then there exists $H>0$ such that for all $x\in \M\setminus C(f)$,
    $$
    d(x,C(f))>H \cdot m(d_x f)^{\frac{1}{\beta}}.
    $$
\end{lemma}
\begin{proof}
    By basic linear algebra, let $A,B:V\rightarrow V$ be two linear morphisms of a normed vector space,
    $$
    |m(A)-m(B)|\leq ||A-B||.
    $$

    Take $x\in \M\setminus C(f)$. For all $c\in C(f)$, $m(d_c f)=0$. So
    $$
    \begin{aligned}
        m(d_x f)&=|m(d_xf)-m(d_c f)|\\
        &\leq ||d_xf-d_cf||\\
        &\leq || d^{1+\beta} f ||_{\infty} \cdot d(x,c)^{\beta}.
    \end{aligned}
    $$
    And $d(x,c)\geq || d^{1+\beta} f ||_{\infty}^{-1/\beta} \cdot m(d_xf)^{\frac{1}{\beta}}$.
\end{proof}

\begin{lemma}\label{lem:twoblocks}
    Let $\Lambda_{C,\varepsilon,\chi}$ be an expanding-Pesin block as defined in Definition \ref{def:Pesinblock1}. Then there exists a Pesin block $\Lambda_{C',\varepsilon',\chi,d}$, as defined in \cite{LDP}, such that 
    $$
    \Lambda_{C,\varepsilon,\chi} \subseteq \Lambda_{C',\varepsilon',\chi.d}.
    $$
\end{lemma}
\begin{proof}
    Items $(PS1)$, $(PS2)$, $(PS4)$ from \cite{LDP} are automatically satisfied, as we take $l=d$. Apply Lemma \ref{lem:discon}, we get item $(PS3)$ from $\eqref{eqn:criticalappro}$ while shrinking $C$ and $\varepsilon$ if necessary.
\end{proof}



\subsection{SPR coding and outcomes}\label{sec:SPREM}

In this section we prove Theorem \ref{thm:EM}, hence Theorem \ref{thm:main2}.

\begin{lemma}\label{lem:Q-lb}
    $q_{\varepsilon}$ is lower bounded on $\Lambda_{C,\varepsilon,\chi',d}$ (defined in \cite{LDP}) for $\chi'>\chi$.
\end{lemma}
\begin{proof}
    Recall that $q_{\varepsilon}(\hat{x})$ is defined via $C_{\chi}(\hat{x})$ and $d(\{x_{-1}, x_0,x_1 \}, C(f))$ in $\eqref{eqn:q}$. By condition $(PS1)$ of Pesin block in \cite{LDP}, for all $\hat{x}\in \Lambda_{C,\varepsilon,\chi',d}$,
    $$
    ||d_{\hat{x}}\hat{f}^{-k}||\leq Ce^{-k\chi'}. 
    $$
    So $||C_{\chi}(\hat{x})^{-1}||^2\leq C^2\cdot\frac{2}{1-e^{2(\chi'-\chi)}}$.

    By condition $(PS3)$ of Pesin block in \cite{LDP}, $d(\{x_{-1}, x_0,x_1 \}, C(f))>Ce^{-\varepsilon}$. So $Q_{\varepsilon}(\hat{x})$ is lower-bounded on $\Lambda_{C,\varepsilon,\chi',d}$. So will be $q_{\varepsilon}(\hat{x})$. 
\end{proof}

As in \cite{ALP} \cite{LDP}, the local stable/unstable sets $W^{s/u}_{\rm{loc}}(\hat{x})$ depends on $q_{\varepsilon}(\hat{x})$ continuously. As a result, Lemma \ref{lem:Q-lb} implies the following fact.

\begin{lemma}\label{lem:su-con}
    On $\Lambda_{C,\varepsilon',\chi',d}$, $W^{u/s}_{\rm{loc}}(\hat{x})$ varies continuously with respect to $\hat{x}$.
\end{lemma}

\begin{proposition}\label{prop:5-cSPR}
    Let $(\mathcal{M},f)$ be an expanding-SPR system. Let $\mu\in \mathbb{P}_{\rm{erg}}(f)$ be an ergodic measure, $\chi$-expanding. Then there exists a coding $\hat{\pi}:\hat{\Sigma}\rightarrow \hat{\M}$ built in Theorem \ref{thm:coding-LDP} for which the irreducible component $\hat{\Sigma}'\subseteq \hat{\Sigma}$, which lifts all measures $\nu\in \mathbb{P}_{\rm{erg}}(f)$ with $\nu\overset{h}{\sim}\mu$ (Item 4, Theorem \ref{thm:coding-LDP}), is SPR (Definition \ref{def:1-sSPR}).
\end{proposition}
\begin{proof}
    Fix $\varepsilon$ sufficiently small. By definition of expanding-SPR property, there exist $C,\chi>0$, $h_0 > 0$,$\tau>0$, such that the Pesin block $\Lambda_{C,\varepsilon,\chi}$ has $\mu(\Lambda_{C,\varepsilon,\chi})>\tau$ for all $\mu\in \mathbb{P}_{\rm{erg}}(f)$, $h(\mu)>h_0$.

    By Lemma \ref{lem:twoblocks}, there exists a Pesin block $\Lambda_{C',\varepsilon',\chi,d}$ which contains $\Lambda_{C,\varepsilon,\chi}$. Consider the coding $(\hat{\Sigma}^{\#},\sigma)$ on $\operatorname{NUE}_{\frac{\chi}{2}}$. By Lemma \ref{lem:Q-lb}, $q_{\varepsilon}$ is lower bounded on $\Lambda_{C,\varepsilon,\chi}$. Item 3 of Theorem \ref{thm:coding-LDP} then implies that there exists a finite collection of cylinders $C(v)\subseteq \hat{\Sigma}$ such that $\hat{\pi}^{-1}(\Lambda_{C',\varepsilon',\chi,d})\subseteq \bigcup C(v)$.

    Only consider the cylinders $C(v)$ which are in the irreducible component $\hat{\Sigma}'$. Lift $\mu$ to $\hat{\Sigma}'$ by Item $4.a$ of Theorem \ref{thm:coding-LDP} and we see that this set is non-empty. All ergodic measures $\overline{\nu}$ on $\hat{\Sigma}'$ can be seen as ergodic measures on $\hat{\Sigma}^{\#}$. Thus, if $h_{\sigma}(\overline{\nu})>h_0$, Item $4.b$ of Theorem \ref{thm:coding-LDP} implies that $h_{\hat{f}}(\hat{\pi}_{\ast}\overline{\nu})>h_0$. The SPR property of $(\mathcal{M},f)$ implies that $\hat{\pi}_{\ast}\overline{\nu}(\Lambda_{C,\varepsilon,\chi,d})>\tau$. So
    $$
    \overline{\nu}(\bigcup_{\text{finite}}C(v))\geq \overline{\nu} (\hat{\pi}^{-1}(\Lambda_{C,\varepsilon,\chi,d}))
    > \tau.
    $$
    This concludes our proof.
\end{proof}

\begin{proof}[Proof of Theorem \ref{thm:EM}]

Let $\{ \mu \}$ be the set of MMEs of $(\mathcal{M},f)$. Assume that this set is non-empty, or there is nothing left to prove.

We first prove the finiteness. Since $(\mathcal{M},f)$ is expanding-SPR, there exist $0<h_0<h_{\rm{top}}(f)$ and a Pesin block $\Lambda_{C,\varepsilon,\chi,d}$ such that for all $\nu\in \mathbb{P}_{\rm{erg}}(f)$ with $h_f(\nu)>h_0$, $\nu(\Lambda_{C,\varepsilon,\chi,d})>\tau$. 
We prove by contradiction. Assume that $\sharp\{\mu\}=\infty$. Then compactness of $\Lambda_{C,\varepsilon,\chi,d}$ implies that we may find a sequence of MMEs $\mu_k$ and $\hat{x}_k\in \operatorname{supp}(\mu_k)\cap \Lambda_{C,\varepsilon,\chi,d}$ such that $\hat{x}_k\rightarrow \hat{x}_0\in \Lambda_{C,\varepsilon,\chi,d}$. Take a neighborhood $U_0$ of $\hat{x}_0$ such that for all $\hat{y},\hat{z}\in U_0$, we have transverse intersection of the local stable/unstable sets at $\hat{y}$ and $\hat{z}$ respectively. This is possible due to Lemma \ref{lem:su-con}. For $k$ large enough, $\hat{x}_k\in U_0$, and $\mu_k(U_0)>0$. By Definition \ref{def:homoclinic}, this implies $\mu_k\overset{h}{\sim}\mu_{k'}$ for all $k,k'$ latge enough.  However, MMEs cannot be homoclinically related, since if so, they are lifted to different measures to the irreducible symbolic model $(\hat{\Sigma}',\sigma)$ over one of the topological homoclinic class $\operatorname{HC}(\mu)$ by Proposition \ref{prop:5-cSPR}, contradicting uniqueness of MME on $(\hat{\Sigma}',\sigma)$ (Theorem \ref{thm:Markov}, item 1 ).

Now we proceed to the mixing rate. Take $\mu$, an ergodic MME of $(\mathcal{M},f)$. 

By Definition \ref{def:1-uHC}, we know that $\operatorname{supp}(\hat{\mu})\subseteq\operatorname{HC}(\mu)$. Thus the system $(\operatorname{HC}(\mu),\hat{f})$ has topological entropy $ h_{\rm{top}}(f)=h(\mu)\leq h_{\rm{top}}(f|_{\operatorname{HC}(\mu)})\leq h_{\rm{top}}(f)$. 

Let $\hat{\pi}:\hat{\Sigma}'\rightarrow \hat{\mathcal{M}}$ be the symbolic coding built in Theorem \ref{thm:coding-LDP} (item 4). 
The model $(\hat{\Sigma}',\sigma)$ is irreducible, but may still be periodic. In this case, it has a symbolic spectral decomposition $\hat{\Sigma}'=\bigcup_{j=0}^{p-1}\sigma^j\hat{\Sigma}'_0$, with $p$ being the \emph{period of} $\hat{\Sigma}'$. 
We know that the Gurevich measure $\overline{\mu}$ on $\hat{\Sigma}'$ pushes to $\hat{\mu}$ by $\hat{\pi}$: $\hat{\pi}_{\ast}\overline{\mu}=\hat{\mu}$. Lemma 12.2 in \cite{BCS3} also implies $q|p$, where $q$ is the period of $\mu$.

Item $4.c$ of Theorem \ref{thm:coding-LDP} shows that when $q=1$, $p=1$ also and $\mu$ is mixing. If not, since $\overline{\mu}$ is Bernoulli, so will be $\hat{\mu}$. We consider $(\hat{\mathcal{M}},\hat{f}^q)$, and $\hat{\mu}$ has $q$-ergodic components, $\hat{\mu}_i:=\hat{f}^i_{\ast}\hat{\mu}$, $i=0,\cdots, q-1$. It is left to show that for each $\hat{\mu}_i$, we have exponential decay of correlations of H\"older observables for $\hat{f}^q$. Since the coding $\hat{\pi}$ is H\"older continuous in itself, it suffices to show that each $\hat{\mu}_i$ is the pushforward of some $\overline{\mu}_j$, an ergodic component of the MME $\overline{\mu}$ with respect to $(\hat{\Sigma}',\sigma^p)$.
The proof of this is the same as that of Lemma 12.3 in \cite{BCS3}.
\end{proof}

As mentioned in Theorem \ref{thm:main2} and below, there are more implications of expanding-SPR, as for the SPR in \cite{BCS3}.
Since their proofs, once we have Theorem \ref{thm:EM} and Proposition \ref{prop:5-cSPR}, are identical to the corresponding ones in \cite{BCS3}, we omit them.

\bigskip

Let $\mu\in \mathbb{P}_{\rm{erg}}(f)$ and $\psi:\mathcal{M}\rightarrow \mathbb{R}$ be in $L^2(\mu)$. Denote by $\psi_n:=\sum_{k=0}^{n-1}\psi\circ f^k$. Recall that the variance of $\psi$ with respect to $(\mathcal{M}, \mu)$ is
$$
\operatorname{Var}_{\mu}(f):=\int \left( \psi-\int \psi d\mu \right)^2d\mu. 
$$
And the \emph{asymptotic invariance} of $\psi$ with respect to $(\mathcal{M},f,\mu)$ is
$$
\sigma^2_{\psi}:=\lim_{n\rightarrow \infty}\frac{1}{n}\log\operatorname{Var}_{\mu}(\psi_n).
$$

\begin{theorem}[Theorem 11.1, \cite{BCS3}]\label{thm:AV-BCS}
    Let $f:\mathcal{M}\rightarrow \mathcal{M}$ be a $C^r$-map, $r>1$, and expanding-SPR. Let $\mu\in \mathbb{P}_{\rm{erg}}(f)$ be an MME of $(\mathcal{M},f)$. Let $\psi:\mathcal{M}\rightarrow \mathbb{R}$ be H\"older continuous. Then
    \begin{itemize}
        \item[1.]  Green-Kubo formula: let $p$ be the period of $\mu$, then
        $$
        \sigma^2_{\psi}=\frac{1}{p}\left( \operatorname{Var}_{\mu}(\psi)+2\sum_{n=1}^{\infty}\operatorname{Cov}_{\mu}(\psi_p,\psi_p\circ f^{np}) \right),
        $$
        where
        $$
        \operatorname{Cov}_{\mu}(\varphi,\psi)=\int \varphi\psi d\mu-\left( \int \varphi d\mu \right)\left( \int \psi d\mu \right),
        $$
        \item[2.] Asymptotic Laplace transform: suppose $\int \psi d\mu=0$, then for all $z\in \mathbb{C}$,
        $$
        \int e^{z\psi_n/\sqrt{n}} d\mu\overset{n\rightarrow \infty}{\longrightarrow} e^{\frac{1}{2}z^2\sigma_{\psi}^2}.
        $$
    \end{itemize}
\end{theorem}

\begin{corollary}[Central limit theorem, Corollary 11.14, \cite{BCS3}]\label{cor:CLT-BCS}
    Let $f:\mathcal{M}\rightarrow \mathcal{M}$ be a $C^r$-map, $r>1$, and expanding-SPR. Let $\mu\in \mathbb{P}_{\rm{erg}}(f)$ be an MME of $(\mathcal{M},f)$. Let $\psi:\mathcal{M}\rightarrow \mathbb{R}$ be H\"older continuous, with $\int \psi d\mu=0$. Then for every $a<b$,
    $$
    \lim_{n\rightarrow \infty }\hat{\mu}\left(\{ \hat{x}: \frac{\psi_n(x)}{\sqrt{n}}\in (a,b) \}\right)=
    \begin{cases}
        \frac{1}{\sqrt{2\pi\sigma^2_{\psi}}} \int_a^b e^{-t^2/2\sigma^2_{\psi}} dt & \text{ if $\sigma_{\psi}\neq  0$}\\
        1_{(a,b)}(0) & \text{ if $\sigma_{\psi}=0$, $ab\neq 0$. }
    \end{cases}
    $$
\end{corollary}

\section{Proof of the continuity of the exponents}

For any finite Borel measure $\rho$ on $\M$ and $x \in \M$, we let
$$
\Lambda_f(\rho) = \int \log |\Jac f| d \rho \;\;\; \text{and} \;\;\; \Lambda_f(x) = \limsup\limits_{n \to +\infty} \frac{1}{n} \log |\Jac_x f^n|.
$$
For any $i \in [\![ 1, d ]\!]$, we denote by $\lambda_f^i(x)$ the $i$-th Lyapunov exponent of $f$ at $x$.
We introduce these notations where we specify the map since we will consider a sequence of dynamical systems in this section.\\


If $(f_k : \M \to \M)_{k \in \bbN}$ is a sequence of $C^r$-maps, we say that it converges $C^r$-weakly to $f$ if it converges in the $C^1$-topology and if the sequence $(f_k)$ is bounded for the $C^r$ norm.
In this section, we prove the following statement:

\begin{theorem}[Entropic continuity of the Lyapunov exponent, $C^r$ pertubative version]
\label{th:main-cont}
Let $f : \M \to \M$ be a $C^r$-map, with $r > 1$.
Let $(f_k : \M \to \M)$ be a sequence of $C^r$-maps converging $C^r$-weakly to $f$.
Let $(\nu_k)$ be a sequence of Borel probability measures.
Assume that $\nu_k \overset{*}{\longrightarrow} \mu$ and that each $\nu_k$ is ergodic for $f_k$ and has $d$ positive Lyapunov exponents.
For any $\alpha > \left ( d-1 + \frac{d^2}{r} \right ) R(f)$ and $\alpha' > \frac{d}{r} R(f)$, there exist $\mu_0, \mu_1$ $f$-invariant probability measures and $\beta \in [0,1]$ such that $\mu = \beta \mu_1 + (1-\beta) \mu_0$ and
\begin{itemize}
\item[1)] if $\beta > 0$, then $\mu_1$ is expanding: $\lambda_f^d(x) > 0$,
\item[2)] $\limsup\limits_{k \to +\infty} h_{f_k}(\nu_k) \leq \beta h_f(\mu_1) + (1-\beta) \alpha + \alpha'$,
\item[3)] $\limsup\limits_{k \to +\infty} \Lambda_{f_k}(\nu_k) \geq \beta \Lambda_f(\mu_1)$,
\item[4)] NCS for $\beta > 0$: let $\mathcal{K}$ be a sequence of integers along which $(h_{f_k}(\nu_k))$ converges to its limsup. Then $\beta > 0$ if and only if $\liminf\limits_{\mathcal{K} \ni k \to +\infty} \Lambda_{f_k}(\nu_k) > \alpha$.
\end{itemize}
\end{theorem}

Notice that this statement is empty if $r \leq d(d+1)$, since this gives $\alpha+\alpha' > d R(f)$ and taking $\beta = 0$ suffices.
We prove that this implies Theorem \ref{thm:main-con-1} from the introduction at the end of the paper.


\subsection{Decomposition of the limit measure}
\label{sec:decomp-meas}

In this section, we consider a fixed ergodic system $(\M, f, \nu)$.
We will apply the results of this section to each $f_k$ in the sequence $(f_k)$ in Theorem \ref{th:main-cont}.
We do not give any proof in this section, since they are all the same as their counterparts in \cite{B} \cite{A} \cite{DL}.

\subsubsection{Hyperbolic times}

For $x \in \M$, let $\psi(x) = \log |\Jac_x f| - \frac{d}{r+d-1}\log^+ |\Jac_x f| - \frac{d (r-1)}{r(r+d-1)} \log || df ||_{\infty} \leq \log |\Jac_x f|$ and $\psi^n = \sum\limits_{k=0}^{n-1} \psi \circ f^k$.
Let $\delta > 0$.

\begin{definition}[Hyperbolic times]
A $\delta$-hyperbolic time for $x$ is a $\delta$-Pliss time for the sequence $(\psi^n(x))_{n \in \mathbb{N}}$.
Meaning that for any $k \in [\![0, n-1 ]\!]$, we have
$$
\sum\limits_{i=k}^{n-1} \psi(f^i x) \geq (n-k) \delta.
$$
We denote by $E_{\delta}(x)$ the set of $\delta$-hyperbolic times for $x$.
\end{definition}

For $x \in \M$ and $n \in \bbN$, we let $m_n(x) = \max E_{\delta}(x) \cap [\![0, n-1 ]\!]$ and $m_n^+(x) = \min E_{\delta}(x) \backslash [\![0 , n-1 ]\!]$.
For $E \subset \bbN$, we denote its \textit{lower-density} by
$$
\underline{d}(E) = \liminf\limits_{n \to +\infty} \frac{1}{n} \#  \left ( E \cap [\![0, n-1 ]\!] \right ).
$$

\begin{lemma}[Lemmas 2.1 and 2.3 from \cite{DL}]
\label{lem:density-HB}
If $\nu$ is ergodic and $\int \psi d \nu > \delta$, then there is $\beta > 0$ such that for $\nu$-a.e. $x$, $\underline{d}(E_{\delta}(x)) \geq \beta$. We also have $m_n^+(x) \underset{n \to +\infty}{\sim} m_n(x) \underset{n \to +\infty}{\sim} n$.
\end{lemma}

\begin{remark}[Remark 2.2 from \cite{DL}]
\label{rk:unif-beta}
In the proof of Lemma \ref{lem:density-HB} in \cite{DL}, the lower bound $\beta_{inf}$ we obtain on $\underline{d}(E_{\delta}(x))$ can be chosen to be
$$
\beta_{inf} = \frac{\int \psi d \nu - \delta}{\log || \Jac f ||_{\infty} - \delta}.
$$
We will use this remark in the proof of Theorem \ref{th:main-cont} to obtain a non-zero $\beta$.
\end{remark}

Until Section \ref{sec:ent-exp-HB}, we fix $\delta > 0$ such that $\int \psi d \nu > \delta$ and $\delta < \frac{r-1}{r+d-1} \log 2$.
We will use this second bound in the proof of Lemma \ref{lem:repara}.

\subsubsection{Defining the unstable and neutral component}
\label{ssec:def-compo}

For $E \subset \bbN$ and $n,M \in \bbN$, we let $E_n^M = \bigcup\limits_{
\substack{n > a,b \in E\\
\text{s.t. } |b-a| \leq M}} [\![ a , b [\![
$.

Then, consider the connected components of $E_n^M$ and write $E_n^M = \bigsqcup\limits_{j} [\![ a_j , b_j [\![$.
For $m \in \mathbb{N}$, we let
$$
E_n^{M,m} = \bigsqcup\limits_j [\![ a_j, c_j [\![
$$
where $c_j$ is the largest element of $E$ such that $c_j \leq b_j - m +1$.
Notice that $\partial E_n^M \subset E$, this gives $E_n^{M,1} = E_n^M$.
For $q \in \mathbb{N}$, we also define $E_n^{M,m,q} = E_n^{M,\max (m,q)}$.

\begin{lemma}[Lemma 4.1 from \cite{A}]
\label{lem:EnM}
For any $n,M,m \in \bbN$, we have
\begin{itemize}
\item[1)] $\partial E_n^{M,m} \subset E$,
\item[2)] $\# \partial E_n^{M,m} \leq \frac{2(n+M)}{M}$,
\item[3)] $\limsup\limits_{n \to +\infty} d_n(E_n^{M,m}) \leq \frac{2}{M}$.
\end{itemize}
\end{lemma}

\begin{lemma}[Section 4 from \cite{A}]
\label{lem:inc-EnM}
For any $E \subset \bbN$ and $n, m \in \mathbb{N}$, the sequence $\left ( E_n^{M,m} \right )_M$ is nondecreasing for the inclusion.
\end{lemma}

By using the fact that $\psi \leq \log |\Jac f|$, we have the following lemma.

\begin{lemma}[Lemma 4.2 from \cite{A}]
\label{lem:prop-hb}
Let $x \in \M$ and $n,M,m \in \bbN$.
We have:
\begin{itemize}
\item[1)] for $l \in E(x)$ and $k \leq l$, we have $\log |\Jac_{f^k x} f^{l-k}| \geq (l-k) \delta$,
\item[2)] for any connected component $[\![ k,l [\![$ of $E_n^{M,m}(x)$, we have $\log |\Jac_{f^k x} f^{l-k}| \geq (l-k) \delta$,
\item[3)] for any $[\![ k,l [\![ \subset E_n^{M,m}(x)$, we have $\log |\Jac_{f^k x} f^{l-k}| \geq -M \log || \Jac f ||_{\infty}$.
\end{itemize}
\end{lemma}

From the above item $3)$, the orbit of $x$ is away from the critical points at times of $E_n^{M,m}(x)$, uniformly in $n$ and in $x$ in the set $\{y \in \M : E_n^{M,m}(y) = E_n^{M,m}(x) \}$.

For $x \in \M$, we call the connected components of $E^M_n(x)$ the \textit{$M$-hyperbolic blocks} of $x$.
We let
$$
\xi^{M,m,q}_{\nu,n}:= \frac{1}{n} \int \sum_{i\in E^{M,m,q}_n(x)}\delta_{f^i x}d\nu(x) \;\;\; \text{and} \;\;\; \beta_{\nu,n}^{M,m,q} = \frac{1}{n} \int \# E_n^{M,m,q}(x) d \nu(x).
$$
We let $\xi^{M,m}_{\nu,n}=\xi^{M,m,1}_{\nu,n}$ and  $\xi^M_{\nu,n}=\xi^{M,1,1}_{\nu,n}$,
this is the hyperbolic component of the measure $\nu$.
We let $\beta^{M,m}_{\nu,n}=\beta^{M,m,1}_{\nu,n}$ and $\beta^M_{\nu,n}=\beta^{M,1,1}_{\nu,n}$.
We let $\eta_{\nu,n}^{M,m,q} = \nu - \xi_{\nu,n}^{M,m,q}$, $\eta_{\nu,n}^{M,m} = \eta_{\nu, n}^{M,m,1}$ and $\eta_{\nu,n}^M = \eta_{\nu, n}^{M,1,1}$, this is the \textit{neutral component} of the measure $\nu$.
By Cantor's diagonal argument, there is a sequence of integers $\mathcal{n}$ such that for all $M,m,q \in \mathbb{N}^*$
\begin{align*}
\xi_{\nu,n}^{M,m,q} \underset{\mathcal{n} \ni n \to +\infty}{\longrightarrow} \xi_{\nu}^{M,m,q},\\
\eta_{\nu,n}^{M,m,q} \underset{\mathcal{n} \ni n \to +\infty}{\longrightarrow} \eta_{\nu}^{M,m,q},\\
\beta_{\nu,n}^{M,m,q} \underset{\mathcal{n} \ni n \to +\infty}{\longrightarrow} \beta_{\nu}^{M,m,q}.
\end{align*}
We also let $\overline{\xi_n^{M,m,q}}$ be the normalization of the measure $\xi_n^{M,m,q}$, so that $\xi_n^{M,m,q} = \beta_n^{M,m,q} \overline{\xi_n^{M,m,q}}$.
Similarly, we define $\overline{\xi^{M,m}}$, $\overline{\xi_n^{M,m}}$ and $\overline{\xi_n^M}$.\\

\subsubsection{Decomposition of the entropy}

\begin{definition}[$\lambda$-partition]
For a measure $\lambda$ on $\M$, a $\lambda$-partition is a countable collection of Borel subsets of $\M$ whose union has full $\lambda$-measure and whose pairwise intersections have zero $\lambda$-measure.
\end{definition}

For a $\lambda$-partition $\mathcal{P}$, we denote by $\mathcal{P}_{\lambda}$ the set $\{ P \in \mathcal{P} : \lambda(P) > 0 \}$.\\

Let $\iota > 0$ and let $\mathcal{P}$ be a finite partition such that $\nu(\partial \mathcal{P}) = 0$ and 
$$
\lim\limits_{n \to +\infty} \frac{1}{n} H_{\nu}(\mathcal{P}^n) \geq h(\nu) - \iota.
$$
For a $\nu$-partition $\mathcal{F}$ and $n,M,m \in \mathbb{N}^*$, we denote by $\mathcal{F}^{E^{M,m}_n}$ the $\nu$-partition for which two points $x$ and $y$ are in the same atom if $E_n^{M,m}(x) = E_n^{M,m}(y)$ and if $f^l x$ and $f^l y$ are in the same atom of $\mathcal{F}$ for every $l \in E_n^{M,m}(x) = E_n^{M,m}(y)$.\\

For $n,M,m$, let $\mathcal{E}_n^{M,m}$ be the partition of $\M$ whose atoms are the sets $\{ x \in \M : E_n^{M,m}(x) = E \}$ where $E \subset \mathbb{N}$ are such that the associated atoms are of positive $\nu$-measure.
To shorten notations, we will denote by $E$ both the atom of $\mathcal{E}_n^{M,m}$ and the set $E_n^{M,m}(x)$ for $x$ in this atom. Then for $E \in \mathcal{E}_n^{M,m}$, we write
$$
\nu_E = \frac{\nu(E \cap \cdot)}{\nu(E)} \; \; \; \text{and} \; \; \; \nu_E^E = \frac{1}{\# E} \sum\limits_{i \in E} f^i_* \nu_E.
$$
The main proposition of this section is the following entropy inequality, which splits the entropy of the measure $\nu$ into the sum of the entropy of the hyperbolic component of $\nu$ and the entropy conditioned to the hyperbolic component of $\nu$, with some error terms.

\begin{proposition}[Proposition 2.9 from \cite{DL}]
\label{prop:entropy-ineq}
Let $\mathcal{Q}$ be a $\nu$-partition into disjoint sets such that
\begin{itemize}
\item[i)] $\forall n,M,m \in \bbN$, $\forall E \in \mathcal{E}_n^{M,m}, \forall i \in E, f^{-i} \mathcal{Q}$ is a $\nu_E$-partition,
\item[ii)] $\forall M,m, \# \{ Q \in \mathcal{Q}^m : \exists n, \exists E \in \mathcal{E}_n^{M,m}, \nu_E^E(Q) > 0 \} < +\infty$,
\item[iii)] $\forall m,M \in \mathbb{N}^*, \overline{\xi_{\nu}^{M,m}}(\partial \mathcal{Q}^m)=0$.
\end{itemize}
Then for any $M$ large enough and $m \in \mathbb{N}^*$,
\begin{align*}
h(\nu) \leq &\beta_{\nu}^{M,m} \frac{1}{m} H_{\overline{\xi_{\nu}^{M,m}}}(\mathcal{Q}^m) + \limsup\limits_{n \to +\infty} \frac{1}{n} H_{\nu}(\mathcal{P}^n \mid \mathcal{Q}^{E_n^{M,m}}) + \frac{3 \log M}{M} + \iota  \\
&+\limsup\limits_{n \to +\infty} \frac{m}{n}\sum\limits_{E \in \mathcal{E}_n^{M,m}} \nu(E) \log \left ( \#\mathcal{Q}_{\nu_E^E} \right ) \#\partial E.  
\end{align*}
\end{proposition}

Let $\mathtt{F}_{\nu} \subset \M$ be a set of positive $\nu$-measure where the convergence $\frac{m_n^+(x)}{n} \underset{n \to +\infty}{\longrightarrow} 1$ is uniform.
Let $\zeta$ be the normalized restriction of $\nu$ to $\mathtt{F}_{\nu}$.
By relating the entropy of $\nu$ to the entropy of $\zeta$, we obtain the following:




\begin{proposition}[Proposition 2.11 from \cite{DL}]
\label{prop:entropy-final}
If a partition $\mathcal{Q}$ satisfies the hypotheses of Proposition \ref{prop:entropy-ineq}, then for any $M$ large enough and $m \in \mathbb{N}^*$, we have
\begin{align*}
h(\nu) \leq &\beta_{\nu}^{M,m} \frac{1}{m} H_{\overline{\xi_{\nu}^{M,m}}}(\mathcal{Q}^m) + \limsup\limits_{n \to +\infty} \frac{1}{n} H_{\zeta}(\mathcal{P}^n \mid \mathcal{Q}^{E_n^{M,m}})\\
&+ \nu(\M \backslash \mathtt{F}_{\nu}) \log \# \mathcal{P} + \frac{3 \log M}{M} + \iota  \\
&+\limsup\limits_{n \to +\infty} \frac{m}{n}\sum\limits_{E \in \mathcal{E}_n^{M,m}} \nu(E) \log \left ( \#\mathcal{Q}_{\nu_E^E} \right ) \#\partial E.
\end{align*}
\end{proposition}

\subsubsection{Defining the adapted partition}
\label{ssec:def-partition}

In this section, we define a partition satisfying the hypotheses of Proposition \ref{prop:entropy-ineq} and such that we control the distortion on each atom.\\

Given any $q \in \bbN^*$ and $a \in \left ( -\frac{1}{q}, 0 \right )$, we define
$$
I_q = \{ \left [ \frac{i}{q} + a, \frac{i+1}{q} + a \right ) : i \in \bbZ \}.
$$
Define the partition
$$
\mathcal{Q}_q = \{ (\log |\Jac f|)^{-1}(J) : J \in I_q \}.
$$
We also require small enough atoms.
Hence for $\varepsilon > 0$, we consider a finite partition $\M(\varepsilon)$ whose atoms have diameter smaller than $\varepsilon$.
Finally, let
$$
\mathcal{R}_{q, \varepsilon} = \mathcal{Q}_q \vee \M(\varepsilon).
$$
We choose $a$ and $\M(\varepsilon)$ such that the border of $\mathcal{R}_{q, \varepsilon}^m$ has measure zero for $\xi^{M,m}$ for any $M,m \in \bbN^*$.
We will show in Lemma \ref{lem:entropy-R} that this partition satisfies the hypotheses of Proposition \ref{prop:entropy-ineq}.

\subsection{Entropy of the neutral component}

In this section, we consider a fixed ergodic system $(\M, f, \nu)$.
We will apply the results of this section to each $f_k$ in the sequence $(f_k)$ in Theorem \ref{th:main-cont}.

\subsubsection{General bounds}

By the same proof as Lemma 9 form \cite{B}, we have:

\begin{lemma}[Lemma 9 from \cite{B}]
\label{lem:reparamization-partition}
Let $\mathcal{Q}$ be a $\nu$-partition which satisfies the hypotheses of Proposition \ref{prop:entropy-ineq}.
Assume that for every atom $Q$ of $\mathcal{Q}^{E_n^{M,m}}$, the set $Q \cap \mathtt{F}_{\nu}$ is covered by a family $\Psi_Q$ of continuous maps $\theta : [0,1]^d \to \M$ that are $C^1$ on $(0,1)^d$ and such that for any $k \in [\![0 , n-1 ]\!]$, we have $|| d(f^k \circ \theta) ||_{\infty} \leq 1$.
Then, we have
$$
\limsup\limits_{n \to +\infty} \frac{1}{n} H_{\zeta}(\mathcal{P}^n \mid \mathcal{Q}^{E_n^{M,m}})
\leq \limsup\limits_{n \to +\infty} \frac{1}{n}\sum\limits_{Q \in \mathcal{Q}^{E_n^{M,m}}}\zeta(Q) \log \# \Psi_Q.
$$
\end{lemma}

Thus, the main statement of this section is the following.
This is an analog of Proposition 1 from \cite{B} and Lemma 3.2 from \cite{DL}. 
We prove it in the next subsection.
For any Borel set $A \subset \M$ such that $\nu(A) > 0$, we denote by $\nu_A$ the normalized restriction of $\nu$ to $A$.

\begin{lemma}
\label{lem:repabound}
Let $q,m, M,n \in \mathbb{N}$ be large enough.
There exist $\gamma_{n,M,m,q}(f) > 0$ and $\varepsilon_q > 0$ satisfying the following.
Let $\mathcal{R}_q = \mathcal{R}_{q, \varepsilon_q}$.
Let $F_n$ be the intersection between an atom of $\mathcal{R}_q^{E_n^{M,m}}$ and $\mathtt{F}_{\nu}$.
Then $F_n$ is covered by a family $\Psi_{F_n}$ of continuous maps $\theta: [0,1]^d \to \M$ that are $C^1$ on $(0,1)^d$ and such that for any $k \in [\![ 0, n-1 ]\!]$ we have $|| d(f^k \circ \theta) ||_{\infty} \leq 1$, and such that
\begin{align*}
\frac{1}{n}\log \sharp \Psi_{F_n}\leq &\left( 1-\frac{\# E^{M,m}_n}{n} \right)C(f) + \frac{d}{r} \log ||df ||_{\infty}\\
&+\frac{d}{r-1}\left( \int \frac{\log^+| \Jac f^q|}{q}d\xi^{M,m,q}_{n,\nu,F_n}-\int\frac{\log|\Jac f^q|}{q}d\xi^{M,m,q}_{n,\nu,F_n} \right)\\
&+\gamma_{n,M,m,q}(f)
\end{align*}
where $\xi_{n,\nu, F_n}^{M,m,q}$ is the measure $\xi_{\rho, n}^{M,m,q}$ with $\rho = \nu_{F_n}$ and $\gamma_{n,M,m,q}(f)$ satisfies for any $K>0$
$$
\limsup\limits_{q \to +\infty} \limsup\limits_{M \to +\infty} \limsup\limits_{n \to +\infty} \sup \{ \gamma_{n,M,m,q}(f) : f : \M \to \M \text{ is } C^r \text{ and } || df ||_{\infty} \leq K \} = 0.
$$
\end{lemma}

\subsubsection{Semi-local Reparametrization lemma}

This section contains the main difference between our proof and those in \cite{B} \cite{DL}.
The main statement of this section is Lemma \ref{lem:repara}, which is a high-dimensional version of Lemma 15 from \cite{B} and Lemma 3.8 from \cite{DL}.
These were obtained by iterating one-dimensional results obtained by using Yomdin theory, respectively Lemma 12 in \cite{B} and Lemma 3.5 in \cite{DL}.
Here, we obtain Lemma \ref{lem:repara} by iterating Lemma \ref{lem:iter}.
The latter is an adaptation of Lemma 4.3 from \cite{BB} for non-invertible maps with arbitrary critical points.\\

We first state some definitions and results from \cite{BB}.\\

For $r \geq 0$, when we consider possible orders of differentiation, we denote by $[1,r]$ the set $[\![1 , \lfloor r \rfloor ]\!] \cup \{r \}$, we define $[0,r]$ the same way and we denote by $(1,r]$ the set $[\![2 , \lfloor r \rfloor ]\!] \cup \{r \}$.

\begin{definition}[Bounded couple, Definition 3.1 from \cite{BB}]
\label{def:bounded-couples}
A couple $(\sigma, \theta)$ of $C^r$-maps $\sigma : [0,1]^d \to \M$ and $\theta : [0,1]^d \to [0,1]^d$ is a bounded couple if the restriction $\theta : (0,1)^d \to (0,1)^d$ is a homeomorphism onto its image and if
$$
\forall s \in [1, r-1], || d^s \left (t \mapsto \Jac_{\theta(t)} \sigma \right ) ||_{\infty} \leq \frac{1}{10 d} || \Jac_{\theta (\cdot)} \sigma ||_{\infty}.
$$
We say that it is $\varepsilon$-bounded if additionally $|| d(\sigma \circ \theta) ||_{r-1} \leq \varepsilon$.
\end{definition}

Bounded couples satisfy the following distortion inequality.

\begin{lemma}[Distortion inequality, Lemma 3.4 from \cite{BB}]
\label{lem:dist-dim}
Let $(\sigma, \theta)$ be a bounded couple.
For any $s,t \in [0,1]^d$, we have
$$
\frac{|\Jac_{\theta(t)} \sigma|}{|\Jac_{\theta(s)} \sigma|} \leq \sqrt{2}.
$$
\end{lemma}

We say that a set is semialgebraic if it is defined by polynomial inequalities.

\begin{lemma}[Algebraic Lemma]
\label{lem:AL}
\cite{SB_1985-1986__28__225_0,Burguet2008,PilaWilkie,BinyaminiNovikov}
Let $P : [0,1]^d \to \bbR^d$ be a polynomial map with $\deg(P) \leq r$.
Let $Y \subset \bbR^d$ be a bounded semialgebraic set of $\bbR^d$.
Then there is a constant $B_{r,d}$ depending only on $r,d,\deg(Y), \diam(Y)$ and semialgebraic analytic injective maps $( \theta_i : (0,1)^d \to (0,1)^d)_{i \in I}$ such that
\begin{itemize}
\item[1)] $\# I \leq B_{r,d}$,
\item[2)] $||  d \theta_i ||_{r-1}, || d(P \circ \theta_i) ||_{r-1} \leq \frac{1}{100 d}$,
\item[3)] $\bigsqcup\limits_{i \in I} \im(\theta_i) = P^{-1}(Y)$, the equality and disjoint union hold modulo sets of dimension at most $d-1$.
\end{itemize}
\end{lemma}

Notice that each $\theta_i$ can be continuously extended to $[0,1]^d$, since $||d \theta_i ||_{\infty} \leq 1$.

\begin{lemma}[General Leibniz rule]
\label{Leibniz}
Let $U \subset \bbR^d$ be an open set.
Let $f,g : U \to \bbR$ be $C^r$-maps with $r > 1$.
We have 
$$
|| f \times g ||_r \leq C_{r,d} || f ||_r \times || g ||_r
$$
where $C_{r,d}$ is a constant depending only on $r$ and on the dimension $d$.
\end{lemma}

\begin{lemma}[Fa{\'a} di Bruno's formula, \cite{EM_FDB}]
\label{FDB}
Let $U \subset \bbR^n$ and $V \subset \bbR^m$ be open sets.
Let $g : U \to V$ and $f : V \to \bbR^d$ be $C^r$-maps with $r > 1$.
Assume that $|| dg ||_{r-1} \leq 1$, then we have
$$
|| d (f \circ g)||_{r-1} \leq C_r || df ||_{r-1} || dg ||_{r-1}
$$
where $C_r$ is a constant depending only on $r$.
\end{lemma}

Let $g$ be an iterate of $f$. 
The input of the following lemma is an $\varepsilon$-bounded couple $(\sigma, \theta)$ and the output is a family $\Phi$ such that each $(g \circ \sigma, \theta \circ \phi)$ is a bounded couple, for $\phi \in \Phi$ (item $3)$).
Additionally, the cardinality $\#\Phi$ is controlled (item $5)$) and the set $\bigcup\limits_{\phi \in \Phi} \im(\sigma \circ \theta \circ \phi)$ covers the set where the Jacobian of $g$ takes values in some fixed interval $[a, e \times a]$ (item $2)$).
Let us first state the lemma, then give an idea of its proof before giving the complete proof.\\


\begin{lemma}
\label{lem:iter}
Let $p \in \bbN^*$ and $g = f^p$.
There exists $\varepsilon_0(g) > 0$ satisfying the following.
Let $\varepsilon \in (0, \varepsilon_0)$ and $(\sigma, \theta)$ be an $\varepsilon$-bounded couple.
There exists a chart $\Gamma : \bbR^d \to \M$ such that for any $a \in \bbZ$, there exists a family $\Phi$ of maps $(\phi: [0,1]^d \to [0,1]^d)$ such that
\begin{itemize}
\item[1)] each $\phi$ is the continuous extension of an analytic homeomorphism from $(0,1)^d$ onto its image,
\item[2)] $(\sigma \circ \theta)^{-1}\{ x : \lfloor \log |\Jac_x g| \rfloor = a \} \subset \bigcup\limits_{\phi} im(\phi)$,
\item[3)] $(\Gamma^{-1} \circ g \circ \sigma, \theta \circ \phi)$ is a bounded couple,
\item[4)] $|| d \phi ||_{\infty} \leq \frac{1}{C_{r,d}} e^{-\frac{a^-}{r-1}}|| dg ||^{-1/r}$, where $C_{r,d}$ here is chosen at the end of the proof of Lemma \ref{lem:repara}.
\item[5)] $\# \Phi \leq C_{r,d} e^{\frac{d a^-}{r-1}}|| dg ||_{\infty}^{d/r}$,
\item[6)] $|| d^r (\Gamma^{-1} \circ g \circ \sigma \circ \theta \circ \phi) ||_{\infty} \leq \frac{1}{C_{r,d}}$,
\item[7)] $\Gamma = \widetilde{\Gamma} \left ( \varepsilon \times \cdot \right )$ with $|| d \widetilde{\Gamma}^{\pm 1} ||_{r-1} \leq C$ where $C$ is a universal constant,
\item[8)] the quantity $\varepsilon_0(g)$ depends only on the $C^r$ norm of $g$.
\end{itemize}
\end{lemma}

As explained in the introduction of this subsection, we adapt the proof of Lemma 4.3 from \cite{BB}. We explain the strategy.

In \hyperref[RL:step1]{\textbf{Step 1}}, we chose $\varepsilon > 0$ such that, after a change of scale $\varepsilon$, the high order derivatives of $g$ and of its Jacobian are small compared to $1$.

In \hyperref[RL:step2]{\textbf{Step 2}}, we explicit the choice of a chart $\Gamma$ so that the $\varepsilon$-bounded couple $(\sigma, \theta)$ gives a $1$-bounded couple $(\varepsilon^{-1} \Gamma^{-1} \circ \sigma, \theta \circ \varphi)$, where $\varphi$'s are affine contractions whose rate depends only on $r$ and $d$ (these will in fact not appear since we will do all the calculations up to a constant depending on $r$ and $d$).

In \hyperref[RL:step3]{\textbf{Step 3}}, we show that the high order derivatives of the Jacobian of $g \circ \sigma$ are small compared to $e^{a^+} \times || \Jac \sigma ||_{\infty}$, where $a^+ = \max(a, 0)$ and $a$ is approximately the value of the Jacobian of $g$.
Then, we compose by the main affine contraction to make the high order derivatives of the Jacobian of $g \circ \sigma$ small compared to $e^a \times || \Jac \sigma ||_{\infty}$.
One of the main features of Yomdin theory is that we only need to make the $r$-th order derivative small, and not all the orders.
Thus, the cost in cardinality of these affine reparametrizations is $e^{d \frac{a^-}{r-1}}$.

In \hyperref[RL:step4]{\textbf{Step 4}}, we approach the Jacobian of $g \circ \sigma$ by a polynomial, to which we apply the Algebraic Lemma \ref{lem:AL}.
This gives a family of reparametrizations, we show that they cover the desired set.

In \hyperref[RL:step5]{\textbf{Step 5}}, we show that the reparametrizations from \hyperref[RL:step4]{\textbf{Step 4}} give a family of maps $\Delta : [0,1]^d \to [0,1]^d$ such that each $(\Gamma^{-1} \circ g \circ \sigma, \theta \circ \Delta )$ is a bounded couple.

In \hyperref[RL:step6]{\textbf{Step 6}}, we compose by a last affine contraction to obtain item $6)$.\\

We do all the calculations up to a constant depending on $r$ and $d$.
By composing by additional affine contractions, one can always make this constant disappear without changing the final bound.
We will write $X \lesssim Y$ when the inequality $X \leq C_{r,d} Y$ holds for some constant $C_{r,d}$ depending only on $r$ and $d$.

\begin{proof}
\textbf{Step 1:} Choice of $\varepsilon_0(g)$.
\phantomsection
\label{RL:step1}\\

Let $r' = \min(r,2)$.
Let $\varepsilon_0(g,r,d,\M) > 0$ be small enough so that
\begin{itemize}
\item[i)] $\varepsilon_0 < r_{inj}(\M) / || dg ||_{\infty}$, where $r_{inj}(\M)$ is the injectivity radius of $\M$,
\item[ii)] $\varepsilon_0 < \left ( C_{r,d} \max\limits_{k \in [ 0 , r-1 ]} || d^s (\Jac g) ||_{\infty} \right )^{-1}$,
\item[iii)] $\varepsilon_0 < \left( C_{r,d} \max\limits_{k \in [1 , r]} || d^s g ||_{\infty} \right)^{-\frac{1}{r'-1}}$,
\item[iv)] $\varepsilon_0 < \left( C_{r,d} || dg ||_{\infty}^{d-1} || d^{r'} g ||_{\infty} \right)^{-\frac{1}{r'-1}}$,
\end{itemize}
where $C_{r,d}$ is chosen in the rest of the proof. In particular, we choose it such that the following holds for any $0 < \varepsilon < \varepsilon_0$:
\begin{itemize}
\item[1)] for any local chart $\alpha : \bbR^d \to U$ where $\diam(U) < \varepsilon$, the set $g(U)$ is contained in a local chart,
\item[2)] if $\alpha, U$ are as in $1)$ and if $\beta : \bbR^d \to V$ is a local chart such that $g(U) \subset V$ and $G = \varepsilon^{-1} \beta^{-1} \circ g \circ \alpha(\varepsilon \cdot)$, then for any $k \in [1,r]$ we have
$
|| d^k (\Jac G) ||_{\infty} \leq 1,
$
\item[3)] with the notations of $2)$, we have $|| d G ||_{\infty} \leq C_{r,d} || dg ||_{\infty}$ and for any $k \in ( 1,r ]$ we have $|| d^k G ||_{\infty} \leq 1$.
\end{itemize}
\vspace{1em}

\textbf{Step 2:} Defining the chart $\Gamma$.
\phantomsection
\label{RL:step2}\\

We introduce some notations, illustrated in Figure \ref{fig-lem-iter}.
Let $U \subset \M$ be a local chart containing $\im (\sigma \circ \theta)$.
Let $V \subset \M$ be a local chart containing $\im (g \circ \sigma \circ \theta)$.
Let $\alpha : \bbR^d \to U$ and $\widetilde{\Gamma} : \bbR^d \to V$ be associated charts.
We may assume that their $C^r$ norm and that of their Jacobian are bounded by some universal constant $C$.
Let $\Sigma := \varepsilon^{-1} \times (\alpha^{-1} \circ \sigma) : \sigma^{-1}(U) \subset [0,1]^d \to \bbR^d$.
Let $G = \varepsilon^{-1} \times \widetilde{\Gamma}^{-1} \circ g \circ \alpha ( \varepsilon \cdot) : 
(\alpha \circ \varepsilon)^{-1}(U) \subset \bbR^d \to \bbR^d$.
Hence, if we let $\Gamma^{-1} = \varepsilon^{-1} \times \widetilde{\Gamma}^{-1} : V \to \bbR^d$, we have $G \circ \Sigma = \Gamma^{-1} \circ g \circ \sigma : \sigma^{-1}(U) \to \bbR^d$.\\

\begin{figure}[h!]
\centering
\includegraphics[width=0.9\linewidth]{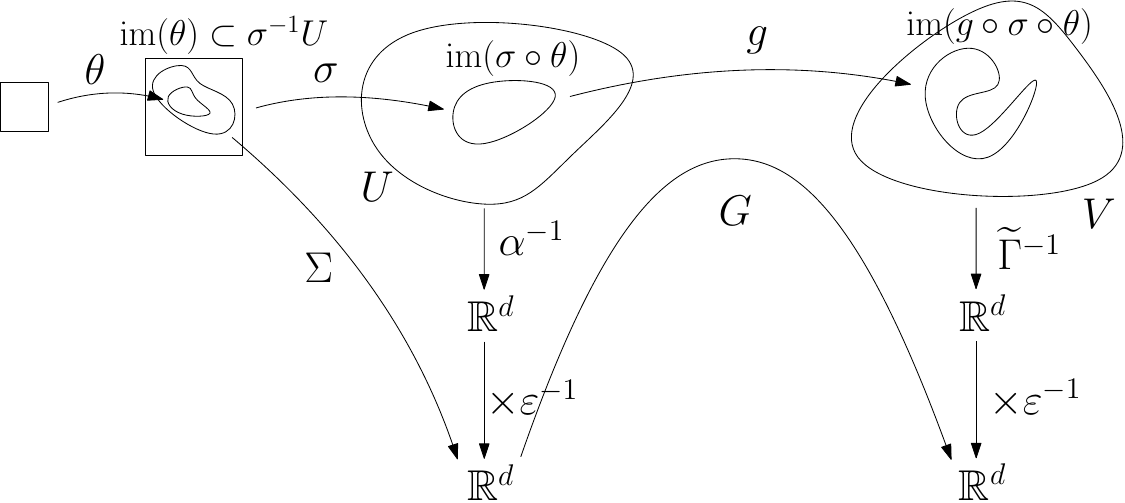}
\vspace{0.5em}
\caption{Notations for Step 2 of Lemma \ref{lem:iter}}
\label{fig-lem-iter}
\vspace{1.5em}
\end{figure}

For $t \in [0,1]^d$, we let $\Omega(t) = \Jac_{\theta(t)} (G \circ \Sigma)$.
Let $J : \begin{cases}
(\alpha \circ \varepsilon)^{-1}(U) \subset \bbR^d \to \bbR\\
x \mapsto \Jac_x G
\end{cases}$.
Item $2)$ from \hyperref[RL:step1]{\textbf{Step 1}} gives
$$
\forall k \in [1, r-1], || d^k J ||_{\infty}\leq 1.
$$
From the fact that $(\sigma, \theta)$ is $\varepsilon$-bounded and by using Lemma \ref{FDB}, we have for any $s \in [1,r]$
$$
|| d^s (\Sigma \circ \theta) ||_{\infty} = \frac{1}{\varepsilon} || d^s (\alpha^{-1} \circ \sigma \circ \theta) ||_{\infty} \lesssim 1.
$$
Therefore, Lemma \ref{FDB} gives
\begin{align*}
|| d^s (\Jac_{\Sigma \circ \theta( \cdot)} G) ||_{\infty}
&= || d^s (J \circ \Sigma \circ \theta) ||_{\infty}\\
&\lesssim \left ( \max\limits_{k \in [ 1 , r-1 ] } || d^k J ||_{\infty} \right) \times \left ( \max\limits_{k \in [ 1 , r-1 ] } || d^k (\Sigma \circ \theta) ||_{\infty} \right )\\
&\lesssim 1.
\end{align*}
We now show that $(\Sigma, \theta)$ is bounded:
\begin{align*}
|| d^s ( (\Jac \Sigma) \circ \theta) ||_{\infty}
&= \varepsilon^{-d} || d^s [ (\Jac_{\sigma \circ \theta} \alpha^{-1}) \times (\Jac_{\theta} \sigma)] ||_{\infty}\\
\text{Lemma \ref{Leibniz}} \hspace{1em} &\lesssim \varepsilon^{-d} \left (\max_{k \in [ 1 , r-1 ]} || d^k [(\Jac\sigma) \circ\theta] \right)\\
(\sigma, \theta) \text{ is a bounded couple}\hspace{1em}&\lesssim \varepsilon^{-d} || (\Jac \sigma) \circ \theta ||_{\infty}\\
&\lesssim || \left (\Jac (\varepsilon^{-1} \times \alpha^{-1} \circ \sigma)\right ) \circ \theta||_{\infty}\\
&\lesssim \frac{1}{10d} || (\Jac \Sigma) \circ \theta||_{\infty}.
\end{align*}
Thus, for any $k \in [ 1 , r-1 ]$, Lemma \ref{lem:dist-dim} gives
\begin{align*}
|| d^k (\Jac_{\theta (\cdot)} \Sigma) ||_{\infty}
\leq || \Jac_{\theta (\cdot)} \Sigma ||_{\infty}
\leq \sqrt{2} \inf \left | \Jac_{\theta (\cdot)} \Sigma \right |.
\end{align*}
\vspace{1em}

\textbf{Step 3:} Main affine reparametrization.
\phantomsection
\label{RL:step3}
\vspace{1em}

From the inequalities in \hyperref[RL:step2]{\textbf{Step 2}} and item $2$ from \hyperref[RL:step1]{\textbf{Step 1}}, we have
\begin{align*}
|| d^{r-1} \Omega||_{\infty} 
&= || d^{r-1} \left ( (\Jac(G \circ \Sigma)) \circ \theta \right)  ||_{\infty}\\
\text{Lemma \ref{Leibniz}}\hspace{1em}&\lesssim \left ( \max\limits_{k \in [ 0 , r-1 ]} || d^k \left ( \Jac_{\Sigma \circ \theta (\cdot)} G \right ) ||_{\infty} \right ) \times \left ( \max\limits_{k \in [ 0 , r-1 ]} || d^k \left ( \Jac_{\theta ( \cdot )} \Sigma \right ) ||_{\infty} \right )\\
&\lesssim \max ( || \Jac_{\Sigma \circ \theta} G ||_{\infty}, 1) \inf \left | \Jac_{\theta (\cdot)} \Sigma \right |.
\end{align*}
Recall that $r' = \min(r,2)$.
Then for $s,t \in [0,1]^d$, we have
\begin{align*}
|\Jac_{\Sigma \circ \theta (s)} G - \Jac_{\Sigma \circ \theta(s)} G)|
\leq ||s-t||^{r' - 1}.
\end{align*}
Therefore, for any $y \in \im (\Sigma \circ \theta)$, we have
$$
|| d^{r-1} \Omega ||_{\infty} \lesssim \max \left ( |\Jac_y G| + \sqrt{d}, 1 \right ) \inf \left | \Jac_{\theta (\cdot)} \Sigma \right |.
$$
To simplify notations, we let $\inf \Sigma = \inf \left | \Jac_{\theta (\cdot)} \Sigma \right |$.
Let
$$
H(a) = \{ t \in [0,1]^d : \log |\Jac_{\Sigma \circ \theta (t)} G| \in a + \left [ - C_{r,d} , C_{r,d}\right] \}
$$
where $C_{r,d}$ is such that $(\sigma \circ \theta)^{-1}(\{ x : \lfloor \log |\Jac_x g |\rfloor = a \}) \subset H(a)$.
Hence taking $y \in H(a)$ gives
$$
|| d^{r-1} \Omega ||_{\infty} \lesssim \frac{1}{10} e^{a^+} \inf \Sigma.
$$
Then, we consider a family $\Psi = (\psi : [0,1]^d \to [0,1]^d)$ of affine maps such that
\begin{itemize}
\item[i)] for any $\psi \in \Psi$, $|| d \psi ||_{\infty} \leq e^{-\frac{a^-}{r-1}}$,
\item[ii)] for any $\psi \in \Psi$, we have $\im(\psi) \cap H(a) \neq \emptyset$,
\item[iii)] the union $\bigcup\limits_{\psi \in \Psi} \im(\psi)$ contains $H(a)$,
\item[iv)] $\#\Psi \leq e^{\frac{d a^-}{r-1}}$.
\end{itemize}


Thus, for any 
$\psi \in \Psi$, we have:
\begin{align*}
|| d^{r-1} (\Omega \circ \psi) ||_{\infty} \leq \frac{1}{10 C_{r,d}} e^a \inf \Sigma.
\end{align*}
\vspace{1em}

\textbf{Step 4:} Polynomial interpolation.
\phantomsection
\label{RL:step4}
\vspace{1em}

Let 
$\psi \in \Psi$.
Let $P_{\psi}$ be the interpolation polynomial of $\Omega \circ \psi$ at $0$ of degree $\lceil r \rceil-2$, and we let $\widetilde{P_{\psi}} = \frac{P_{\psi}}{e^a \inf \Sigma}$.
We let $Y = B(0, 2C_{r,d}) \backslash B\left(0, \frac{1}{2 C_{r,d}}\right)$, with $C_{r,d}$ chosen below, and we apply the Algebraic Lemma \ref{lem:AL} to $\widetilde{P}_{\psi}$ and $Y$, which gives us a family $\left ( \xi_i \right )_{i \in I(\psi)}$ such that
$$
\bigsqcup\limits_{ i \in I(\psi)} \im(\xi_i) = \widetilde{P_{\psi}}^{-1}(Y)
$$


We only keep the $\xi_i$'s such that $\im(\psi \circ \xi_i) \cap H(a) \neq \emptyset$.
We show that the union $\bigcup\limits_{\psi \in \Psi} \bigcup\limits_{j \in I(\psi)} \im(\psi \circ \xi_j)$ covers $H(a)$.
For $t \in H(a)$, we have
\begin{align*}
| \Omega (t) |
&= | \Jac_{\theta (t)} (G \circ \Sigma) |\\
&= | \Jac_{\Sigma \circ \theta (t)} G | \times | \Jac_{\theta(t)} \Sigma |\\
&\in e^a \inf \Sigma \times \left [ \frac{1}{C_{r,d}} , C_{r,d} \right ].
\end{align*}
Then, by definition of $\Psi$, there exist $\psi \in \Psi$ and $s \in [0,1]^d$ such that $\psi(s) = t$.
Therefore we have
\begin{align*}
|\widetilde{P_{\psi}}(s)|
&= \left|\frac{P_{\psi}(s)}{e^a \inf \Sigma}\right|\\
&\in \frac{|\Omega \circ \psi(s)|}{e^a \inf \Sigma} + \left [ \frac{- 1}{10 C_{r,d}}, \frac{1}{10 C_{r,d}} \right ]\\
&\subset \left [ \frac{9}{10 C_{r,d}} , \frac{11}{10} C_{r,d} \right ]\\
&\subset Y.
\end{align*}
\vspace{1em}

\textbf{Step 5:} Proof of the bounded property.
\phantomsection
\label{RL:step5}
\vspace{1em}

Let $\psi \in \Psi$ and $\xi$ be in $(\xi_i)_{i \in I(\psi)}$.
Let $\Delta = \theta \circ \psi \circ \xi$.
We show that $(G \circ \Sigma, \Delta)$ is bounded.
For $k \in [ 1 , r-1 ]$ we have
\begin{align*}
|| d^k (\Jac_{\Delta(\cdot)} (G \circ \Sigma)) ||_{\infty}
&= || d^k (\Omega \circ \psi \circ \xi) ||_{\infty}\\
&\leq || d^k (P_{\psi} \circ \xi) ||_{\infty} + || d^k ((\Omega \circ \psi - P_{\psi})\circ \xi) ||_{\infty}\\
&\leq \frac{1}{100d} e^a \inf \Sigma
+\frac{1}{100 d} || d^{r-1} (\Omega \circ \psi) ||_{\infty}\\
&\leq \frac{2}{100d} e^a \inf \Sigma.
\end{align*}
Then, let $t \in \im(\psi \circ \xi) \cap H(a)$, which is nonempty by construction.
Thus $\log |\Jac_{\Sigma \circ \theta(t)} G| \in a + [- C_{r,d}, C_{r,d}]$ and
\begin{align*}
|\Jac_{\Delta(t)}(G \circ \Sigma)| &= |\Jac_{\Sigma \circ \Delta(t)} G| \times |\Jac_{\Delta(t)} \Sigma| \gtrsim e^a \inf \Sigma.
\end{align*}
Therefore
$$
|| d^k (\Jac_{\Delta(\cdot)} (G \circ \Sigma)) ||_{\infty} \lesssim \frac{1}{10 d} || \Jac_{\Delta( \cdot )}(G \circ \Sigma)||_{\infty}.
$$
\vspace{1em}

\textbf{Step 6:} Final affine reparametrization for item $6)$.
\phantomsection
\label{RL:step6}
\vspace{1em}

By item $3)$ from \hyperref[RL:step1]{\textbf{Step 1}}, we have $|| d^s G ||_{\infty} \leq 1$ for $s > 1$ and $|| d G ||_{\infty} \leq C_{r,d} || dg ||_{\infty}$, and \hyperref[RL:step2]{\textbf{Step 2}} gave for any $k \in [1 , r]$
\begin{align*}
|| d^k (\Sigma \circ \Delta) ||_{\infty}
\lesssim 1.
\end{align*}
Hence Lemma \ref{FDB} gives
\begin{align*}
|| d^r (G \circ \Sigma \circ \Delta) ||_{\infty}
&\lesssim || dg ||_{\infty}.
\end{align*}

Let $(\gamma_i : [0,1]^d \to [0,1]^d)_{i \in N_1}$ be a family of affine maps such that
\begin{itemize}
\item[i)] for any $i \in N_1$, $|| d \gamma_i ||_{\infty} = (C_{r,d} || d g ||_{\infty})^{-1/r}$,
\item[ii)] $\bigsqcup\limits_{i \in N_1} \im (\gamma_i ) = [0,1]^d$,
\item[iii)] $\# N_1 \leq C_{r,d} ||dg ||_{\infty}^{d/r}$.
\end{itemize}
Let $\gamma$ be one of the $\gamma_i$.
We let $\phi = \Delta \circ \gamma$.
Hence $(G \circ \Sigma, \phi)$ is bounded and
$$
|| d^r (G \circ \Sigma \circ \phi) ||_{\infty} \leq 1/C_{r,d}.
$$
\end{proof}

For $k \in \bbN^*$, we define $\varepsilon(f,k)$ by induction by letting $\varepsilon(f,k) = \min\limits_{l < k} \left ( \varepsilon_0(f^k), \varepsilon(f,l), || df ||_{\infty}^{-k} \right )$.
Let $\varepsilon_q \leq \varepsilon(f,q)/4$.
Let $\mathcal{R}_q = \mathcal{Q}_q \vee \M(\varepsilon_q)$ as defined in Section \ref{ssec:def-partition}.
\\

Let $F_n$ be the intersection between an atom of $\mathcal{R}_q^{E_n^{M,m}}$ and $\mathtt{F}_{\nu}$.
Thus there exists $E_n^{M,m} \subset \bbN$ such that $E_n^{M,m}(x) = E_n^{M,m}$ for any $x \in F_n$.
We define the set $\widetilde{E}_n^{M,m,q} \subset \bbN$ to be the set $E_n^{M,m}$ to which we remove the last $q-1$ elements of each connected component.
We assume that $m \geq q$, so that for any $x \in F_n$, we have
$$
\widetilde{E}_n^{M,m,q} \subset E_n^{M,m} = E_n^{M,m,q}(x)
$$
these sets differ by at most $q \times \# \partial E_n^M / 2$ elements.
For $x \in \M$, $i,l \in \bbN$, we let
\begin{align*}
&k_{f^l, i+1}(x) = \lfloor \log |\Jac_{f^i x} f^l| \rfloor,\\
&k_{f^l, i+1}'(x) = \lfloor \log^- |\Jac_{f^i x} f^l| \rfloor.
\end{align*}

\begin{lemma}[Lemma 3.6 from \cite{DL}]
\label{lem:q-control}
There exists $c \in [\![ 0, q-1]\!]$ such that for any $x \in F_n$,
$$
\sum\limits_{l \in (c+q\bbN) \cap \widetilde{E}_n^{M,m,q}} k'_{f^q}(f^l x) \leq \frac{n}{q} + n \left ( \int \frac{\log^+ |\Jac f^q|}{q} d \xi_{n, \nu, F_n}^{M,m,q} - \int \frac{\log |\Jac f^q|}{q} d \xi_{n, \nu, F_n}^{M,m,q} \right ).
$$
\end{lemma}


Let $c$ be given by Lemma \ref{lem:q-control}.
Define
\begin{align*}
&\partial_l E^{M,m}_n = E_n^{M,m} \backslash (E_n^{M,m} + 1),\\
&\mathcal{A}_n =\{ 0=a_1<a_2< ... <a_{\mathcal{m}} \} = \partial_l E^{M,m}_n\cup \left((c+q\mathbb{N})\cap[\![0,n-1 ]\!]\right)\cup \left([\![0,n-1]\!] \backslash E^{M,m}_n\right).
\end{align*}
Then, let $b_i =a_{i+1}-a_i \leq q$ for $i \in [\![ 1, {\mathcal{m}}-1 ]\!]$ and $b_{\mathcal{m}} =n-a_{\mathcal{m}} \leq q$.\\

Given ${\bf{k}}:=(k_l,k_l')_{l\in \mathcal{A}_n}$, $\overline{E}\subset [\![0,n-1 ]\!]$, $m_n\in [\![0,n - 1 ]\!]$, define $F_n^{({\bf{k}}, \overline{E}, m_n)}$ to be the set of $x\in F_n$ satisfying:
\begin{itemize}
    \item[1)] $\overline{E}=E_{\delta}(x) \cap [\![  0 , n- 1 ]\!] \setminus E^{M,m}_n$,
    \item[2)] $k_{a_i}=k_{f^{b_i}}(f^{a_i}x)$, $k'_{a_i}=k'_{f^{b_i}}(f^{a_i}x)$, for any $i \in [\![ 1 , {\mathcal{m}}-1 ]\!]$,
    \item[3)]$m_n(x)=m_n$.
\end{itemize}

The number of possible such sets is subexponential in $n$, as given by the following lemma. Indeed, the term $\log(\log ||\Jac f ||_{\infty})$ will be negligible when taking large iterates of $f$.

\begin{lemma}[Lemma 3.7 from \cite{DL}]
\label{lem:combinatorics}
    For every $m^2 \leq M\leq \frac{n}{2}$ and $M \geq 16$, we have
    $$
    \begin{aligned}
    \frac{1}{n} \log \# \{ ({\bf{k}}, \overline{E}, m_n): F_n^{({\bf{k}}, \overline{E}, m_n)} \neq \emptyset \} \leq &\frac{\log n}{n} + H\left ( \frac{2}{\sqrt{M}} \right ) + \frac{1}{q} + \frac{1}{M}
    + \sup\limits_{x \in \mathtt{F}_{\nu}} \frac{m_n^+(x)-n}{n}\\ &+ \left ( 1 - \frac{\# E_n^{M,m}}{n} \right )(1+2\log(1+\log || \Jac f ||_{\infty})).
    \end{aligned}
    $$
\end{lemma}

By iterating Lemma \ref{lem:iter}, we obtain the following adaptation of Lemma 15 from \cite{B} and of Lemma 3.8 from \cite{DL}.

\begin{lemma}[Semi-local reparametrization lemma]
\label{lem:repara}
Let $(k_{a_i}, k_{a_i}')_{i \in [\![ 1, \mathcal{m} ]\!]}$, $\overline{E} \subset [\![0, n-1 ]\!]$ and $m_n \in [\![ 0, n-1 ]\!]$.
Let $(\sigma, \theta)$ be an $\varepsilon_q$-bounded couple.
There exist families of maps $(\Theta_i)_{i \in [\![ 1, \mathcal{m} ]\!]}$ such that
\begin{itemize}
\item[1)] the families $\Theta_1, ..., \Theta_{\mathcal{m}}$ have a tree structure, where maps in $\Theta_{i+1}$ are children of maps in $\Theta_i$,
\item[2)] for any $i \leq \mathcal{m}$, $\sigma^{-1}(F_n^{({\bf{k}}, \overline{E}, m_n)}) \subset \bigcup\limits_{\theta \in \Theta_i} im(\theta)$,
\item[3)] for any $\theta \in \Theta_i$ and $j \leq i$, $(f^{a_j} \circ \sigma, \theta)$ is an $\varepsilon_q$-bounded couple,
\item[4)] for any $i \in [\![ 1, \mathcal{m} ]\!]$, $\theta \in \Theta_i$ and $j < i$, there exists $\theta_j^i \in \Theta_j$ such that $\theta = \theta_j^i \circ \theta_{j \to i}$ with
$$
||\Jac \theta_{j \to i}||_{\infty} \leq \prod\limits_{j \leq l < i} \left ( \frac{1}{C} e^{- \frac{d k_{a_l}'}{r-1}} || df^{b_l} ||_{\infty}^{-d/r} \right )
$$
where $C$ depends only on $r,d, \delta$ and is chosen at the end of the proof,
\item[5)] for any $i \in [\![ 1, \mathcal{m} ]\!]$,
$$
\# \Theta_i \leq \max\left ( 1, || df ||_{\infty}^d \right )^{\# \overline{E} \cap [\![ 1 , a_i ]\!]} \times \prod\limits_{1 \leq j \leq i} C_{r,d} \exp \left ( (d-1)\log ||df || + \frac{d}{r} \log || df^{b_j} ||_{\infty} + \frac{d k_{a_i}'}{r-1} \right )
$$
\end{itemize}
\end{lemma}

\begin{proof}
We let $F = F_n^{({\bf{k}}, \overline{E}, m_n)}$.
We do an induction on $i \leq \mathcal{m}$.
Assume that $\Theta_i$ has been built up to $i < \mathcal{m}$.
Let $\theta_i \in \Theta_i$, we define the children of $\theta_i$ in $\Theta_{i+1}$.
Let $\Psi$ be the family given by Lemma \ref{lem:iter} for the $\varepsilon_q$-bounded couple $(f^{a_i} \circ \sigma, \theta_i)$, $a=k_{a_i}$ and the map $g = f^{b_i}$, and let $\Gamma : \bbR^d \to \M$ be the associated local chart.
We indeed have $\varepsilon_q \leq \varepsilon_0(f^{b_i})$ since $b_i \leq q$.
Let $\Theta_{i+1}' = \{ \theta_i \circ \psi : \theta_i \in \Theta_i, \psi \in \Psi \}$ and $\theta_i \circ \psi \in \Theta_{i+1}'$.
Thus $(\Gamma^{-1} \circ f^{a_{i+1}} \circ \sigma, \theta_i \circ \psi)$ is a bounded couple.
To build $\Theta_{i+1}$, and obtain $\varepsilon_q$-bounded couples, we compose by a new family of maps.\\

\textbf{Case 1:} $a_{i+1} \in E_n^{M,m}$. Then the set $f^{a_{i+1}} F$ is contained in an atom of $\mathcal{R}_q$, which we recall has a diameter smaller than $\varepsilon_q$.
Thus, there is ball $\widetilde{B} \subset \bbR^d$ of diameter $D = C_{r,d}$ containing $\Gamma^{-1}(f^{a_{i+1}}F)$.
Indeed, item $7)$ from Lemma \ref{lem:iter} gives that $\Gamma = \widetilde{\Gamma}(\varepsilon_q \times \cdot)$ with $||d \widetilde{\Gamma}^{-1} ||_{\infty} \leq C$, where $C$ is a universal constant.
Let $B \subset \bbR^d$ be the ball with the same center as $\widetilde{B}$ but with twice the radius. 
Let $Q$ be the approximation polynomial of $\Gamma^{-1} \circ f^{a_{i+1}} \circ \sigma \circ \theta_i \circ \psi$ at $0$ of degree $\lceil r \rceil - 1$.
Therefore, item $6)$ from Lemma \ref{lem:iter} gives
\begin{align*}
|| Q - \Gamma^{-1} \circ f^{a_{i+1}} \circ \sigma \circ \theta_i \circ \psi ||_r
\leq || d^r \left(\Gamma^{-1} \circ f^{a_{i+1}} \circ \sigma \circ \theta_i \circ \psi \right ) ||_{\infty}
\leq \frac{1}{C_{r,d}}
\leq \frac{D}{2}.
\end{align*}
Let $(\xi_j)_{j \in J}$ be given by the Algebraic Lemma for $Q$ and $B$, so that $\bigcup\limits_{j \in J} im(\xi_j) = Q^{-1}B$ and $\# J \leq C_{r,d}$.
Therefore $(\Gamma^{-1} \circ f^{a_{i+1}} \circ \sigma \circ \theta_i \circ \psi)^{-1} \widetilde{B} \subset \bigcup\limits_{j \in J} im(\xi_j)$.
It remains to show that $(f^{a_{i+1}} \circ \sigma, \theta_i \circ \psi \circ \xi_j)$ is an $\varepsilon_q$-bounded couple.
We first show that it is a bounded couple.
For $s \in [1,r-1]$ we have
\begin{align*}
|| d^s \left ( \Jac_{\theta_i \circ \psi \circ \xi_j} \left ( f^{a_{i+1}} \circ \sigma \right ) \right ) ||_{\infty}
&= || d^s \left ( \Jac_{\theta_i \circ \psi \circ \xi_j} \left ( \Gamma \circ  \Gamma^{-1} \circ f^{a_{i+1}} \circ \sigma \right ) \right ) ||_{\infty}\\
\text{Lemmas } \ref{FDB} \text{ and } \ref{Leibniz} \hspace{1em}&\lesssim \varepsilon_q^d \max\limits_{l \in [0,r-1]} || d^l \left ( \Jac_{\theta_i \circ \psi \circ \xi_j} (\Gamma^{-1} \circ f^{a_{i+1}} \circ \sigma)\right ) ||_{\infty}\\
\text{$(\Gamma^{-1} \circ f^{a_{i+1}} \circ \sigma, \theta_i \circ \psi)$ is bounded} \hspace{1em}
&\lesssim \varepsilon_q^d || \Jac_{\theta_i \circ \psi} (\Gamma^{-1} \circ f^{a_{i+1}} \circ \sigma) ||_{\infty}\\
\text{Lemma } \ref{lem:dist-dim}\hspace{1em}&\lesssim \varepsilon_q^d || \Jac_{\theta_i \circ \psi \circ \xi_j} (\Gamma^{-1} \circ f^{a_{i+1}} \circ \sigma) ||_{\infty}\\
&\lesssim \frac{1}{10d} || \Jac_{\theta_i \circ \psi \circ \xi_j} (f^{a_{i+1}} \circ \sigma) ||_{\infty}
\end{align*}
Then we show that it is an $\varepsilon_q$-bounded couple.
The idea is to show that after the change of scale $\Gamma$, we have a unit size.
For $s \in [1,r]$, we have
\begin{align*}
|| d^s (f^{a_{i+1}} \circ \sigma \circ \theta_i \circ \psi \circ \xi_j) ||_{\infty}
&\leq || d^s \left(\Gamma \circ \Gamma^{-1} \circ f^{a_{i+1}} \circ \sigma \circ \theta_i \circ \psi \circ \xi_j \right ) ||_{\infty}\\
\text{Lemma \ref{FDB}} \hspace{1em}&\lesssim \varepsilon_q \max\limits_{l \in [1,r]} || d^l \left(\Gamma^{-1} \circ f^{a_{i+1}} \circ \sigma \circ \theta_i \circ \psi \circ \xi_j \right ) ||_{\infty}\\
&\lesssim \varepsilon_q \max\limits_{l \in [1,r]} \left ( || d^l ( Q \circ \xi_j ) ||_{\infty} \right.\\
&\hspace{1em}\left.+ || d^l \left ( Q \circ \xi_j - \Gamma^{-1} \circ f^{a_{i+1}} \circ \sigma \circ \theta_i \circ \psi \circ \xi_j \right ) ||_{\infty} \right )\\
&\lesssim \varepsilon_q \left ( \frac{1}{100d} + \frac{1}{ C_{r,d}}  \right )\\
&\leq \varepsilon_q.
\end{align*}

\textbf{Case 2:} $a_{i+1} \in \overline{E}$. We have $b_{i+1} = 1$, and given the expression of the bound in item $5)$, we may multiply the amount of required reparametrizations by $\max(1, || df ||_{\infty}^d)$. In particular, we simply cut $[0,1]^d$ into $\lceil || df ||_{\infty}^d \rceil$ pieces and reparametrize according to these pieces, which concludes.\\

\textbf{Case 3:} $a_{i+1} \notin E_n^{M,m} \cup \overline{E}$. We have $b_{i+1} = 1$.
For surface diffeomorphisms and one-dimensional maps, we could show that $f^{a_{i+1}} \circ \sigma \circ \theta_i \circ \psi$ was already $\varepsilon$-bounded.
We cannot expect the same thing here, since the ``neutral time'' condition does not imply ``no expansion''.
However, it implies ``expansion in at most $d-1$ directions''.
Therefore, we expect that the number of reparametrizations required to obtain $\varepsilon_q$-bounded couples will be bounded by a $(d-1)$-dimensional quantity, such as $(d-1) R(f)$.

Let $\overline{\sigma} = \Gamma^{-1} \circ f^{a_{i+1}} \circ \sigma \circ \theta_i \circ \psi$ and $Q$ be the polynomial approximation of $\overline{\sigma}$ of degree $\lceil r \rceil - 1$ at 0.
Consider a partition $\mathcal{C}$ of $\mathbb{R}^d$ into squares of side $1$.
For each $C \in \mathcal{C}$, we apply the Algebraic Lemma \ref{lem:AL} to $Q$ at $C$, and denote by $\Theta_C$ the associated family.
We consider two types of squares in $\mathcal{C}$: the ones near the border of the image of $\overline{\sigma}$ and the ones far from it, in the interior.
More precisely, let $\mathcal{C}_{out} = \{ C \in \mathcal{C} : C \cap \partial \textup{im}(\overline{\sigma}) \neq \emptyset \}$ and $\mathcal{C}_{in} = \{ C \in \mathcal{C} \backslash \mathcal{C}_{out} : C \cap \textup{im}(\overline{\sigma}) \neq \emptyset \}$. 
We denote by $\Theta_{i+1}$ the reparametrizations of the form $\theta_i \circ \psi \circ \overline{\theta}$ for $\psi \in \Psi$ and $\overline{\theta}$ in some $\Theta_C$, with $C \in \mathcal{C}_{in} \cup \mathcal{C}_{out}$ (we are abusing notations, since we are defining the children of $\theta_i$ only, not all of $\Theta_{i+1}$).

We denote by $\overline{T}$ (resp. $\underline{T}$) the reparametrizations in $\Theta_{i+1}$ given by a $\overline{\theta}$ from $\mathcal{C}_{in}$ (resp. $\mathcal{C}_{out}$).
By Lemma C.1 from \cite{BB}, we have
\begin{align*}
\underline{T}
&\leq \sum\limits_{\psi \in \Psi} C_{r,d} \times \#\mathcal{C}_{out}\\
&\leq \sum\limits_{\psi \in \Psi} C_{r,d} \max\limits_{l \in [\![ 0 , d-1 ]\!]} || \Lambda^l d_{\cdot} \overline{\sigma} ||_{\infty}\\
&\lesssim \#\Psi \times \max\limits_{l \in [\![ 0 , d-1 ]\!]} || \Lambda^l d_{\cdot} f ||_{\im(\sigma \circ \theta_i \circ \psi)}\\
\text{Lemma \ref{lem:iter}.5)} \hspace{1em}&\lesssim \exp \left ( (d-1 + \frac{d}{r}) \log || df ||_{\infty} + \frac{d k_{a_i}'}{r-1} \right )
\end{align*}
On the other hand we have
\begin{align*}
\# \overline{T}
\leq \sum\limits_{\psi \in \Psi} ||\Jac \overline{\sigma} ||_{\infty}
\end{align*}
Let $\psi \in \Psi$ and let us estimate $||\Jac \overline{\sigma} ||_{\infty}$.
\begin{align*} 
||\Jac \overline{\sigma} ||_{\infty}
&\leq || \Jac \Gamma^{-1} ||_{\infty} \times || \Jac (f^{a_{i+1}} \circ \sigma \circ \theta_i \circ \psi) ||_{\infty}\\
&\lesssim \frac{1}{\varepsilon_q^d} || \Jac (f^{a_{i+1}} \circ \sigma \circ \theta_i \circ \psi) ||_{\infty}.
\end{align*}

Let $K_x = \max \{ k < a_{i+1} : k \in E_{\delta}(x) \}$ be the last hyperbolic time before $a_{i+1}$.
Therefore $]\!] K_x , a_{i+1} [\![$ is made of $a_l$'s only.
Since $a_{i+1}$ is not in $E_{\delta}(x)$, and $K_x \in E_{\delta}(x)$, there is $j \in [\![ K_x , a_{i+1} ]\!]$ such that $\psi^{a_{i+1} - j}(f^j x) < \delta (a_{i+1} - j)$.
Thus $j=a_l$ for some $l < i+1$ and $a_{i+1}-a_l = (i+1)-l$, and for $x = \theta_i \circ \psi(t)$ we have
\begin{align*}
|| \Jac (f^{a_{i+1}} \circ \sigma \circ \theta_i \circ \psi) ||_{\infty}
&\leq 2 |\Jac_t (f^{a_{i+1}} \circ \sigma \circ \theta_i \circ \psi) |\\
&=2 |\Jac_{f^j x} f^{a_{i+1} - j} | \times |\Jac_{\theta_{l \to i} \circ \theta (t)} (f^j \circ \sigma \circ \theta_l)| \times |\Jac_t (\theta_{l \to i} \circ \psi)|\\
\text{Induction with item $4)$}\hspace{1em}&\leq 2 \varepsilon_q^d \times \left( \prod\limits_{s=j}^{a_{i+1} - 1} |\Jac_{f^s x} f| \right) \times \left ( \prod\limits_{s=l}^{i-1} \frac{1}{C} e^{-\frac{d k'_{a_s}}{r-1}} || df ||_{\infty}^{-d/r} \right ) \times |\Jac_t \psi|\\
\text{item $4)$ from Lemma \ref{lem:iter}}\hspace{1em}&\leq 2 \varepsilon_q^d \prod\limits_{s=j}^{a_{i+1}-1} \left ( |\Jac_{f^s x} f| \frac{1}{C} e^{-\frac{d k'_s}{r-1}} || df ||_{\infty}^{-d/r} \right ).
\end{align*}

Then for any $t > 0$, we have $\log(t)-\frac{d}{r-1} \log^-(t) =\frac{r+d-1}{r-1}\left ( \log(t) - \frac{d}{r+d-1} \log^+(t) \right )$. Hence
\begin{align*}
|| \Jac (f^{a_{i+1}} \circ \sigma \circ \theta_i \circ \psi) ||_{\infty}
&\leq \frac{2}{C^{a_{i+1}-j}} \varepsilon_q^d \exp\left ( \frac{r+d-1}{r-1} \sum\limits_{s=j}^{a_{i+1}-1} \log |\Jac_{f^s x} f| \right . \\
&\hspace{1em}\left. - \frac{d}{r+d-1}\log^+ |\Jac_{f^s x} f| - \frac{d (r-1)}{r(r+d-1)} \log || df ||_{\infty} \right )\\
&= \frac{2}{C^{a_{i+1}-j}} \varepsilon_q^d \exp\left (\frac{r+d-1}{r-1}\psi^{a_{i+1}-j}(f^j x)\right)\\
&\leq \varepsilon_q^d \left (\frac{2 \exp \left (\frac{r+d-1}{r-1}\delta \right )}{C} \right)^{a_{i+1}-j}\\
&\leq \frac{1}{C_d}\varepsilon_q^d
\end{align*}
where we obtain the last inequality by choosing $C=4 C_d$ and $\delta \leq \frac{r-1}{r+d-1} \log 2$.
This gives $\# \overline{T} \leq C_{r,d} \times \# \Psi \lesssim \exp\left(\frac{d}{r} \log ||df || + \frac{d k_{a_i}'}{r-1}\right)$, so
$$
\# T \leq C_{r,d} \exp \left ( (d-1 + \frac{d}{r}) \log || df ||_{\infty} + \frac{d k_{a_i}'}{r-1} \right ).
$$
\end{proof}

The following lemma gives a part of the main term of Lemma \ref{lem:repabound}.

\begin{lemma}[Lemma 3.10 form \cite{DL}]
\label{lem:firstterm}
$$\frac{d}{r-1}\sum_{i \text{ s.t. } \ m_n>a_i\notin \widetilde{E}_n^{M,m,q}} k'_{a_i}\leq \left(n-\sharp \widetilde{E}_n^{M,m,q}\right) \frac{d^2}{r} \log^+ \| df \|_\infty.
$$ 
\end{lemma}
\begin{proof}
Fix $x\in F_n$. 
Write
$$
\{i: m_n>a_i\notin \widetilde{E}_n^{M,m,q} \} = \bigcup\limits_{j = 1}^s [\![ c_j , d_j [\![
$$
where $[\![c_j , d_j [\![$ are the consecutive $i$'s such that $a_i \notin \widetilde{E}_n^{M,m,q}$. Notice that $a_{d_j}\in E(x)$ for $j=1,\ldots, s$.
Then for $a,b >0$, we have $\log^+ (ab) \leq \log^+ a + \log^+ b$, thus
$$
\begin{aligned}
0 &\leq \sum_{i, \ m_n>a_i\notin \widetilde{E}_n^{M,m,q}} \psi_{b_i}(f^{a_i} x)\\
&\leq\sum_{i, \ m_n>a_i\notin \widetilde{E}_n^{M,m,q}}\left( \log|\Jac_{f^{a_i} x} f^{b_i}|-\frac{d}{r+d-1}\log^+| \Jac_{f^{a_i} x} f^{b_i} | \right.\\ 
&\hspace{8em}\left.- \frac{d(r-1)}{r(r+d-1)} \log || df ||_{\infty} \right )\\
&\leq\sum_{i, \ m_n>a_i\notin \widetilde{E}_n^{M,m,q}} \left ( \left ( 1 - \frac{d}{r+d-1} \right ) \log^+ |\Jac_{f^{a_i} x} f^{b_i}| \right.\\
&\hspace{8em}\left.- \log^- |\Jac_{f^{a_i} x} f^{b_i}| - \frac{d(r-1)}{r(r+d-1)} \log || df ||_{\infty} \right )\\
&\leq \left ( d \left (1 - \frac{d}{r+d-1} \right ) - \frac{d(r-1)}{r(r+d-1)} \right ) \left ( n - \# \widetilde{E}_n^{M,m,q} \right ) \log || df ||_{\infty} - \sum_{i, \ m_n>a_i\notin \widetilde{E}_n^{M,m,q}} k_{a_i}'.
\end{aligned}
$$
Since we have $\frac{d}{r-1} \left ( d \left (1 - \frac{d}{r+d-1} \right ) - \frac{d(r-1)}{r(r+d-1)} \right ) \leq \frac{d^2}{r}$, we obtain
\begin{align*}
\frac{d}{r-1}\sum_{i, \ m_n>a_i\notin \widetilde{E}_n^{M,m,q}} k_{a_i}'
\leq (n-\# \widetilde{E}_n^{M,m,q}) \frac{d^2}{r} \log^+||df||_{\infty}.
\end{align*}
\end{proof}

Another part of the main term of Lemma \ref{lem:repabound} is given by the following:

\begin{lemma}
\label{lem:error-dim-sup}
$$
\sum\limits_{j=1}^{\mathcal{m}} \log ||df^{b_j} ||_{\infty} \leq n \log ||df ||_{\infty}.
$$
\end{lemma}

We conclude this section by the proof of Lemma \ref{lem:repabound}.

\begin{proof}[Proof of Lemma \ref{lem:repabound}]
We let $\Psi_{F_n}$ be a family $\Theta_{\mathcal{m}}$ of reparametrizations as given in Lemma \ref{lem:repara}.
Then we decompose the bound on $\sharp \Theta_{\mathcal{m}}$ given by item $5)$ of Lemma \ref{lem:repara} into four terms:
\begin{itemize}
\item[1)] by Lemma \ref{lem:firstterm}, 
$$\frac{d}{r-1}\sum_{i, \ m_n>a_i\notin \widetilde{E}_n^{M,m,q}} k'_{a_i}\leq \left(n-\sharp \widetilde{E}_n^{M,m,q}\right) \frac{d^2}{r} \log^+ \| df \|_\infty.$$
\item[2)] Let $x \in F_n^{({\bf{k}}, \overline{E}, m_n)}$, so that $m_n=m_n(x)=\max E(x)\cap [\![0,n[\![$.
Therefore, we have
\begin{align*}
\hspace{-3em}\sum_{i, m_n\leq a_i} \frac{k'_{a_i}}{r-1} &\leq\sum_{k=m_n}^{n-1} \frac{1}{r-1} \log^- |\Jac_{f^k x} f|\\
&\leq \frac{1}{r-1}\sum_{k=m_n}^{n-1} \left(\log|| \Jac f||_{\infty} -\log|\Jac_{f^k x} f| \right)\\
&=\frac{1}{r-1}\left((n-m_n)\log ||\Jac f||_{\infty} \right .\\
&\hspace{1em}\left.-\left( \log| \Jac_{f^{m_n} x} f^{m_n^{+}(x)-m_n}|- \log| \Jac _{f^n x} f^{m_n^+(x)-n}|\right)\right)\\
m_n^+(x) \in E_{\delta}(x) \hspace{1em} &\leq\frac{1}{r-1}\left( (n-m_n)\log ||\Jac f||_{\infty}+(m^+_n(x)-n)\log ||\Jac f||_{\infty}\right)\\
&=\frac{1}{r-1}(m^+_n(x)-m_n)\log || \Jac f||_{\infty}.
\end{align*}
\item[3)] By Lemma \ref{lem:q-control},
$$
\frac{d}{r-1}\sum_{i,a_i\in \widetilde{E}_n^{M,m,q}\cap (c+q\mathbb{N})} k'_{a_i}
\leq n\frac{d}{r-1}\left( \frac{1}{q}+ \int \frac{\log^+|\Jac f^q|}{q}d\xi^{M,m,q}_{n,\nu,F_n}-\int\frac{\log|\Jac f^q|}{q}d\xi^{M,m,q}_{n,\nu,F_n}
 \right).$$
\item[4.] if $a_i\in \widetilde{E}^{M,m,q}_n\setminus (c+q\mathbb{N})$, then $a_i\in \partial_l E^{M,m}_n$ and $[\![a_i , a_i+b_i [\![ \subset E_n^{M,m}$, therefore Lemma \ref{lem:prop-hb} gives
$$
\begin{aligned}
\frac{d}{r-1}\sum_{i,a_i\in \widetilde{E}^{M,m,q}_n\setminus (c+q\mathbb{N})} k'_{a_i}
&\leq \frac{d}{r-1}\sum_{i,a_i\in \partial_l E^{M,m}_n} k'_{a_i}\\
&\leq \frac{d}{r-1}\sum_{i,a_i\in \partial_l E^{M,m}_n} \log^{-} |(f^{b_i})'(f^{a_i} x)|\\
&\leq \frac{d}{r-1}\frac{\sharp\partial E^{M,m}_n}{2}\cdot M\cdot \log|| \Jac f ||_{\infty}.
\end{aligned}
$$
\end{itemize}
We define
$$
\begin{aligned}
&\gamma_{n,M,m,q}(f):=
\left [ 
d^2 \left ( \frac{M}{r-1} + \frac{q}{r} \right ) \frac{\# \partial E_n^{M,m}}{2n} + \left ( \frac{1}{q} + \frac{1}{M} \right ) C_{r,d} +
\frac{d^2}{r-1}\sup\limits_{x \in \mathtt{F}_{\nu}} \frac{m_n^+(x) - m_n(x)}{n}
\right ] \cdot \log|| df ||_{\infty}\\
&\hspace*{6em}+ \frac{\log n}{n} + H\left(\frac{2}{\sqrt{M}}\right)+ \left (\frac{1}{q} + \frac{1}{M} \right ) \log C_{r,d} +\sup_{x\in \mathtt{F}_{\nu}}\frac{m_n^+(x)-n}{n},\\
&C(f):=\left( d-1 + \frac{d^2}{r} \right )\log^{+}|| df ||_{\infty} + \log C_{r,d} + 1 + 2\log (1+d\log || df ||_{\infty}).
\end{aligned}
$$
Then, by noting that $\sharp \mathcal{A}_n \leq n - \sharp E_n^{M,m} + \frac{n}{q} + \frac{n}{M}$
and by Lemmas \ref{lem:combinatorics} and \ref{lem:error-dim-sup}, we obtain
\begin{align*}
\frac{1}{n}\log \# \Psi_{F_n}\leq &\left( 1-\frac{\# E^{M,m}_n}{n} \right)C(f) + \frac{d}{r} \log ||df ||_{\infty}\\
&+\frac{d}{r-1}\left( \int \frac{\log^+|(f^q)'|}{q}d\xi^{M,m,q}_{n,\nu,F_n}-\int\frac{\log|(f^q)'|}{q}d\xi^{M,m,q}_{n,\nu,F_n} \right)\\
&+\gamma_{n,M,m,q}(f).
\end{align*}
Then for $\theta \in \Psi_{F_n}$ and $i \in [\![ 1, \mathcal{m} ]\!]$, the choice of $\varepsilon_q \leq \varepsilon(f,b_i)$ gives
$$
|| d (f^{a_i} \circ \sigma \circ \theta) ||_{\infty} \leq \varepsilon(f, b_i) \leq \max \{ 1, || df ||_{\infty} \}^{- b_i}.
$$
Therefore, for $j \in [\![ 0 , n]\!]$, we have
$
|| d (f^j \circ \sigma \circ \theta) ||_{\infty} \leq 1.
$
\end{proof}

\subsection{Entropy and exponent of the hyperbolic component}
\label{sec:ent-exp-HB}

\subsubsection{Main Proposition}
\label{ssec:main-prop}

Let $f : \M \to \M$ be a $C^r$-map of positive topological entropy.
Let $(f_k : \M \to \M)$ be a sequence of maps converging $C^r$-weakly to $f$.
Let $(\nu_k)$ be a sequence such that each $\nu_k$ is ergodic for $f_k$, has $d$ positive Lyapunov exponents, and such that $(\nu_k)$ converges to some $\mu$ in the weak-$*$ topology.
The $f_k$'s will be the $f$'s of the previous sections.
Instead of proving Theorem \ref{th:main-cont}, we prove the following more technical statement.

\begin{theorem}
\label{th:main2}
Let $(f_k)$ and $(\nu_k)$ be as above.
Let $\mathcal{K}$ be a sequence of integers such that $(h_{f_k}(\nu_k))$ converges to its limsup along $\mathcal{K}$.
For all $\alpha > \left ( d-1 + \frac{d^2}{r} \right ) R(f)$ and $\alpha' > \frac{d}{r} R(f)$, there exist two $f$-invariant probability measures $\mu_1, \mu_0$ and $\beta \in [0,1]$, with $\mu = \beta \mu_1 + (1-\beta) \mu_0$ and
\begin{itemize}
\item[1)] $\beta > 0$ if and only if $\liminf\limits_{\mathcal{K} \ni k \to +\infty} \Lambda_{f_k}(\nu_k) > \alpha$,
\item[2)] $\limsup\limits_{k \to +\infty} h_{f_k}(\nu_k) \leq \beta h_f(\mu_1)+(1-\beta)\alpha + \alpha'$,
\item[3)] assume that $\beta>0$ and that $p \in \bbN^*$ is such that
$$
\liminf\limits_{\mathcal{K} \ni k \to +\infty} \Lambda_{f_k}(\nu_k)
> \alpha + \frac{d (dr+r-1)}{r(r+d-1)} \left ( \frac{1}{p} \log^+ || d f^p ||_{\infty} - R(f)\right ).
$$
Then, by letting $\delta=\min(\alpha - \frac{d(dr+r-1)}{r(r+d-1)} R(f), \frac{r-1}{p (r+d-1)} \log 2 )$, we have
\begin{itemize}
    \item[1')] for $\mu_1$-a.e. $x \in \M$, we have $\Lambda_f(x)\geq \delta$,
    \item[3')] $\liminf\limits_{\mathcal{K} \ni k \rightarrow \infty} \Lambda_{f_k}(\nu_k)\geq \beta\Lambda_f(\mu_1) + (1-\beta) \delta$.
\end{itemize}
\end{itemize}
\end{theorem}

Notice that if $0 < \beta < 1$, then from upper semi-continuity of $\mu \mapsto \Lambda_f(\mu)$, item \textit{3')} gives $\Lambda_f(\mu_0) \geq \delta$.
The proof that the above implies Theorem \ref{th:main-cont} is given at the end of the paper.\\

In order to prove Theorem \ref{th:main2}, we prove, in Section \ref{section:proof-th}, that some iterate of $f$ satisfies the hypotheses of the following main Proposition \ref{prop:main-prop}, whose proof will be completed in Section \ref{ssec:proof}.
Finally, in Section \ref{section:proof-th} as well, we show that this suffices to get Theorem \ref{th:main2}.

\begin{proposition}
\label{prop:main-prop}
Let $(f_k)_{k\in \mathbb{N}}$ and $(\nu_k)_{k\in \mathbb{N}}$ be as above.
Let $\mathcal{K}$ be a sequence of integers such that $(h_{f_k}(\nu_k))_k$ converges to its limsup along $\mathcal{K}$.
Let $\alpha' > \frac{d}{r} \log^+||df ||_{\infty}$.
Assume that
\begin{itemize}
\item[i)] $\liminf\limits_{\mathcal{K} \ni k \to +\infty} \Lambda_{f_k}(\nu_k) > \alpha + \frac{d (dr+r-1)}{r(r+d-1)} \left ( \log^+ || d f ||_{\infty} - R(f) \right )$,
\item[ii)] $\alpha > C(f)$, 
where $C(f) > \left ( d-1 + \frac{d^2}{r} \right ) R(f)$ is the quantity of Lemma \ref{lem:repabound}.
\end{itemize}
Then there exist $\beta, \mu_0, \mu_1$ as in Theorem \ref{th:main2} with $\beta > 0$ and $p=1$ in the definition of $\delta$.
\end{proposition}

In the rest of this section, we detail the notations of the previous sections now that we have the additional parameter $k$. We will also assume that $k$ always belongs to $\mathcal{K}$.
For $x \in \M$, define $\psi_k: \M\rightarrow \mathbb{R}\cup \{ -\infty \}$ by
$$
\psi_k = \log | \Jac f_k| - \frac{d}{r+d-1} \log^+ |\Jac f_k| - \frac{d(r-1)}{r(r+d-1)} \log || df_k ||_{\infty}.
$$
Hence, from Lemma \ref{lem:density-HB}, whenever $\delta_k$ is such that $0 < \delta_k < \int \psi_k d \nu_k$, then the set $E_k(x) = E_{k, \delta_k}(x)$ has lower density bounded away from 0 $\nu_k$-almost everywhere. 

We explain how to use Remark \ref{rk:unif-beta} to make this lower bound uniform in $k$.
By using hypothesis $i)$ from Proposition \ref{prop:main-prop} and the inequality $\psi_k \geq \log |\Jac f_k| - \frac{d}{r+d-1}\left ( d + \frac{r-1}{r} \right ) \log || d f_k ||_{\infty}$, we get
$$
\liminf\limits_{k \to +\infty} \int \psi_k d \nu_k > \alpha - \frac{d (dr+r-1)}{r(r+d-1)} R(f) > 0.
$$ 
Therefore, if we take $\delta = \alpha - \frac{d (dr+r-1)}{r(r+d-1)} R(f)$, then for $k$ large enough, for $\nu_k$-almost every $x \in \M$, the set $E_k(x)$ defined for $\delta_k = \delta$ has positive lower density independent from $k$ (and $x$).
In other words, the hypothesis of Proposition \ref{prop:main-prop} allows us to choose $\delta$ independent from $k$.
However, we mentioned in the proof of the semi-local Reparametrization Lemma \ref{lem:repara} that $\delta$ must satisfy an additional requirement, so we actually define
\begin{align*}
&\widetilde{\delta} = \alpha - \frac{d (dr+r-1)}{r(r+d-1)} R(f)\\
&\delta = \min \left ( \widetilde{\delta}, \frac{r-1}{r+d-1}\log 2 \right ).
\end{align*}

Since this still does not depend on $k$, Remark \ref{rk:unif-beta} gives a lower bound $\beta_{inf} > 0$ uniform in $k$.\\

We write $E_k(x)$, $E_{k,n}^{M,m,q}(x)$, $\xi_{k,n}^{M,m,q}$, $\beta_{k,n}^{M,m,q}$, $\overline{\xi_{k,n}^{M,m,q}}$ and $\eta_{k,n}^{M,m,q}$ for the objects defined in Section \ref{ssec:def-compo} for $f_k$, $\nu_k$.
From Cantor's diagonal argument, we may take some subsequences and assume the convergence to the following limits:
\begin{align*}
&\beta_{k,n}^{M,m,q} \underset{n \to +\infty}{\longrightarrow} \beta_k^{M,m,q} \underset{k \to +\infty}{\longrightarrow} \beta^{M,m,q} \underset{M \to +\infty}{\longrightarrow} \beta \geq \beta_{inf} > 0,\\
\beta_{k,n}^{M,m,q}\overline{\xi_{k,n}^{M,m,q}} = \; &\xi_{k,n}^{M,m,q} \underset{n \to +\infty}{\longrightarrow} \xi_k^{M,m,q} \underset{k \to +\infty}{\longrightarrow} \xi^{M,m,q} \underset{M \to +\infty}{\longrightarrow} \beta \mu_1,\\
&\eta_{k,n}^{M,m,q} \underset{n \to +\infty}{\longrightarrow} \eta_k^{M,m,q} \underset{k \to +\infty}{\longrightarrow} \eta^{M,m,q} \underset{M \to +\infty}{\longrightarrow} (1-\beta)\mu_0.
\end{align*}
The fact that the last column of limits does not depend on $m$ and $q$ comes from item $ii)$ of Lemma \ref{lem:EnM}.
Therefore, we have $\mu = \beta \mu_1 + (1-\beta) \mu_0$ and $\beta > 0$. We also write $\xi^M = \xi^{M,1,1}$.

\begin{lemma}[Lemma 10 from \cite{B}]
\label{lem:f-inv}
The measures $\mu_0$ and $\mu_1$ are $f$-invariant probability measures.
\end{lemma}



Then Lemma \ref{lem:inc-EnM} gives the following:

\begin{lemma}[Proposition 5.2 from \cite{A}]
\label{lem:monotone-xi}
For any $m \in \bbN^*$ and Borel set $B \subset \M$, we have $\xi^{M,m}(B) \underset{M \to +\infty}{\nearrow} \beta \mu_1(B)$.
\end{lemma}

\subsubsection{Entropy of the hyperbolic component}
\label{ssec:ent-HB}

For each $k$, there is $\varepsilon_{q,k} > 0$ given by Lemma \ref{lem:repabound}.
By item $8)$ from Lemma \ref{lem:iter} and the $C^r$-weak convergence of $(f_k)$ to $f$, we may assume that $\varepsilon_{q,k}$ is independant from $k$.
We denote it by $\varepsilon_q > 0$.
Then, let
$$
\mathcal{R}_{q,k} = \mathcal{Q}_{q,k} \vee \mathcal{M}(\varepsilon_q) \;\;\; \text{and} \;\;\; \mathcal{R}_q = \mathcal{Q}_q \vee \mathcal{M}(\varepsilon_q).
$$
where $\mathcal{M}(\varepsilon_q)$ is, as in Section \ref{ssec:def-partition}, a partition whose atoms have diameter smaller than $\varepsilon_q$.
Choose $a$ and $\mathcal{M}(\varepsilon_q)$ from Section \ref{ssec:def-partition} such that the border of $\mathcal{R}_{q,k}^m$ has zero $\mu_1$-, $\xi^{M,m}$- and $\xi_k^{M,m}$-measures, for any $k,M,m \in \bbN^*$.\\

The main result of this section is the following:

\begin{proposition}[Proposition 4.5 from \cite{DL}]
\label{prop:entropy-cv}
For any $q \in \bbN^*$, we have
$$
\lim\limits_{m \to +\infty} \lim\limits_{M \to +\infty} \lim\limits_{k \to +\infty} \frac{1}{m} H_{\overline{\xi_k^{M,m}}}(\mathcal{R}_{q,k}^m) = h_f(\mu_1, \mathcal{R}_{q}) \leq h_f(\mu_1).
$$
\end{proposition}

We also give some properties of the partitions defined above.
In particular, they satisfy the hypotheses of Proposition \ref{prop:entropy-ineq}.

\begin{lemma}[Proposition 6.6 from \cite{A} and Lemma 4.6 from \cite{DL}]
\label{lem:entropy-R}
For any $q,k$, the collections $\mathcal{R}_{q,k}$ and $\mathcal{R}_q$ satisfy:
\begin{itemize}
\item[1)] the collection $\mathcal{R}_q$ is a $\mu_1$-partition,
\item[2)] $H_{\mu_1}(\mathcal{R}_q) < +\infty$,
\item[3)] $\forall E \in \mathcal{E}_{k,n}^{M,m}, \forall i \in E, f_k^{-i} \mathcal{R}_{q,k}$ is a $(\nu_k)_E$-partition,
\item[4)] $\forall M,m,k, \sharp \{ R \in \mathcal{R}_{q,k}^m : \exists n, \exists E \in \mathcal{E}_n^{M,m}, (\nu_k)_E^E(R) > 0 \} < +\infty$,
\item[5)] $\forall m,M \in \bbN^*, \overline{\xi_k^{M,m}}(\partial \mathcal{R}_{q,k}^m)=0$,
\item[6)] $\limsup\limits_{k \to +\infty}\limsup\limits_{n \to +\infty} \frac{1}{n} \sum\limits_{E \in \mathcal{E}_{k,n}^{M,m}} \nu_k(E) \log \left ( \sharp(\mathcal{R}_{q,k})_{(\nu_k)_E^E} \right ) \sharp\partial E \underset{M \to +\infty}{\longrightarrow} 0$.
\end{itemize}
\end{lemma}

\subsubsection{Continuity of the exponent of the hyperbolic component}
\label{ssec:cont-expo-HB}

We follow the notations from Section \ref{sec:decomp-meas} and we let $\mathtt{F}_k = \mathtt{F}_{\nu_k}$ and we denote by $m_{k,n}(x)$ (resp. $m_{k,n}^+(x)$) the quantity $m_n(x)$ (resp. $m_n^+(x)$) for $f_k$.
Hence the convergence $\frac{m_{k,n}^+(x)}{n} \underset{n \to +\infty}{\longrightarrow} 1$ is uniform on $\mathtt{F}_k$. We also defined the measure $\zeta_k$ to be the normalized restriction of $\nu_k$ to $\mathtt{F}_k$.
We define $\xi_{k,n,\mathtt{F}_k}^M$ and $\eta_{k,n,\mathtt{F}_k}^M$ the same way as $\xi_{k,n}^M$ and $\eta_{k,n}^M$ in Section \ref{sec:decomp-meas}, but by pushing forward $\zeta_k$ instead of $\nu_k$.
When $n$ goes to infinity, we also define $\xi_{k,\mathtt{F}_k}^M$ (resp. $\eta_{k,\mathtt{F}_k}^M$) to be the limit of $\xi_{k,n,\mathtt{F}_k}^M$ (resp. $\eta_{k,n,\mathtt{F}_k}^M$).\\

From the hypotheses of Proposition \ref{prop:main-prop}, we have $\liminf\limits_{\mathcal{K} \ni k \to +\infty} \Lambda_{f_k}(\nu_k) > \alpha > (d-1) R(f)$.
Thus for $k$ large enough, the measure $\nu_k$ has only positive Lyapunov exponents, hence $\log |\Jac f_k|$ is $\nu_k$-integrable.
Therefore, we may assume that $\nu_k(\mathtt{F}_k)$ converges quickly enough to 1 so that
$$
\lim\limits_{k \to +\infty} |\Lambda_{f_k}(\nu_k) - \Lambda_{f_k}(\zeta_k)| = 0.
$$
Then, we have $\zeta_k = \xi_{k,n,\mathtt{F}_k}^M + \eta_{k,n,\mathtt{F}_k}^M$, so that
$$
\Lambda_{f_k}(\zeta_k) = \Lambda_{f_k}(\xi_{k,n,\mathtt{F}_k}^M) + \Lambda_{f_k}(\eta_{k,n,\mathtt{F}_k}^M),
$$
where we recall that $\Lambda_{f_k}(\xi_{k,n,\mathtt{F}_k}^M) = \int \log |\Jac f_k| d\xi_{k,n,\mathtt{F}_k}^M$ and $\Lambda_{f_k}(\eta_{k,n,\mathtt{F}_k}^M) = \int \log |\Jac f_k| d\eta_{k,n,\mathtt{F}_k}^M$.
In this section, we will see that $\Lambda_{f_k}(\xi_{k,n,\mathtt{F}_k}^M)$ converges to $\beta \Lambda_f(\mu_1)$.

\begin{proposition}[Proposition 4.10 from \cite{DL}]
\label{proposition:mu1-integrable}
The function $\log |\Jac f|$ is $\mu_1$-integrable.
\end{proposition}


\begin{proposition}[Proposition 4.11 from \cite{DL}]
\label{prop:conv-expo-HB}
We have
\begin{align*}
\int \log |\Jac f_k| d \xi_{k,n,\mathtt{F}_k}^{M} \underset{n \to +\infty}&{\longrightarrow}
\int \log |\Jac f_k| d \xi_{k,\mathtt{F}_k}^{M}\\
\underset{k \to +\infty}&{\longrightarrow}\int \log |\Jac f| d \xi^{M}\\
\underset{M \to +\infty}&{\longrightarrow} \beta \int \log |\Jac f| d \mu_1 = \beta \Lambda_f(\mu_1).
\end{align*}
\end{proposition}

\subsection{Proof of Theorem \ref{th:main2}}

\subsubsection{Estimating the exponent of the neutral component}
\label{sec:expo-neutral}

We give the following estimate of the quantity $\Lambda_{f_k}(\eta_{k,n,\mathtt{F}_k}^M)$.

\begin{proposition}[Proposition 5.1 from \cite{DL}]
\label{prop:expo-neutral}
For any $k \in \bbN$ and $M \in \bbN^*$
$$
\liminf\limits_{n \to +\infty} \int \log |\Jac f_k| d \eta_{k,n,\mathtt{F}_k}^M \geq (1 - \beta_k^M) \delta.
$$
\end{proposition}

\subsubsection{Proof of the Main Proposition}
\label{ssec:proof}

We assume hypotheses $i)$ and $ii)$ from Proposition \ref{prop:main-prop}.
We show that the definitions from Section \ref{ssec:main-prop} of the measures $\mu_0$ and $\mu_1$ and of $\beta$ satisfy the conclusion of Proposition \ref{prop:main-prop}. Recall that the $f$-invariance of $\mu_0$ and $\mu_1$ has been proved in Lemma \ref{lem:f-inv}.
We start by proving item $1')$ from Theorem \ref{th:main2}. 

\begin{lemma}[Lemma 10 from \cite{B}, Lemma 5.2 from \cite{DL} and Proposition 5.4 from \cite{A}]
\label{lem:mu1-HB}
For $\mu_1$-almost every $x \in \M$, we have 
$\Lambda_f(x) \geq \delta$.
\end{lemma}


\begin{lemma}
\label{lem:mu1-exp}
The measure $\mu_1$ has a non-zero expanding component.
More precisely, there exists an $f$-invariant measure $\mu_2$ such that
\begin{itemize}
\item[1)] $0 < \mu_2 \leq \mu_1$,
\item[2)] for $\mu_2$-almost every $x \in \M$, $\lambda_f^d(x) \geq \frac{d^2}{r} R(f) > 0$.
\end{itemize}
\end{lemma}

\begin{proof}
For this proof only, we consider hyperbolic times for $\phi_f(x) = \log |\Jac_x f|$.
We define $\phi_f^n(x) = \sum\limits_{i=0}^{n-1} \phi_f(f^i x)$.
We let
$$
\gamma = \alpha + \frac{d(dr+r-1)}{r(r+d-1)} \left ( \log^+ || df ||_{\infty} - R(f) \right ).
$$
By assumption, there exists $\iota > 0$ such that $\liminf\limits_{\mathcal{K} \ni k \to +\infty} \int \phi_{f_k} d \nu_k > \gamma + \iota$.
We denote by $H_k(x)$ the set of $(\gamma + \iota)$-hyperbolic times for $x$ for $\phi_{f_k}$.
We denote by $\mu_2$ the $f$-invariant measure we obtain by the same construction as $\mu_1$ by replacing $E_k$ by $H_k$.
By Lemma \ref{lem:density-HB}, for $k$ large enough, the set $H_k(x)$ has uniformly bounded lower density for $\nu_k$-almost every $x$, so $\mu_2$ is non-zero.
Then by the same proof as Lemma \ref{lem:f-inv}, $\mu_2$ is $f$-invariant.
By the same proof as Lemma \ref{lem:mu1-HB}, we have that $\Lambda_f(x) \geq \gamma+\iota \geq \alpha > \left ( d - 1 + \frac{d^2}{r} \right ) R(f)$ for $\mu_2$-almost every $x$, so $\lambda_d^f(x) \geq \Lambda_f(x) - (d-1) R(f) \geq \frac{d^2}{r} R(f)$.
It remains to show that $\mu_2$ is a component of $\mu_1$.
It suffices to show that $\phi_{f_k} - (\gamma+\iota) \leq \psi_k - \delta$, since this will imply $H_k(x) \subset E_k(x)$ for any $x$ and $k$.
Recall that $\delta \leq \alpha - \frac{d(dr+r-1)}{r(r+d-1)} R(f)$, which gives
\begin{align*}
\phi_{f_k} - \psi_k
&= \frac{d}{r+d-1} \log^+ |\Jac f_k| + \frac{d(r-1)}{r(r+d-1)} \log || df_k ||_{\infty}\\
&\leq \frac{d^2}{r+d-1} \log || d f_k || + \frac{d(r-1)}{r(r+d-1)} \log || df_k||_{\infty}\\
&= \frac{d(dr+r-1)}{r(r+d-1)} \log || d f_k ||_{\infty}\\
k \text{ large}\hspace{1em}&\leq \frac{d(dr+r-1)}{r(r+d-1)} \log || df ||_{\infty} + \iota\\
&=\gamma - \left ( \alpha - \frac{d(dr+r-1)}{r(r+d-1)} R(f)\right ) + \iota\\
&\leq \gamma + \iota - \delta.
\end{align*}
\end{proof}

Before proving item $2)$ from Theorem \ref{th:main-cont}, we state the following lemma:

\begin{lemma}[Lemma 5.3 from \cite{DL}]
\label{lem:term-II}
We have
$$
0 = \lim\limits_{q \to +\infty} \lim\limits_{m \to +\infty} \lim\limits_{M \to +\infty} \lim\limits_{k \to +\infty} \lim\limits_{n \to +\infty}
\frac{1}{q}\int \log^+ | \Jac f_k^q| d \xi_{k,n}^{M,m,q} - \frac{1}{q}\int \log |\Jac f_k^q| d \xi_{k,n}^{M,m,q}.
$$
\end{lemma}


We now prove item $2)$ from Theorem \ref{th:main-cont}.

\begin{proposition}
We have
$$
\limsup\limits_{k \to +\infty} h_{f_k}(\nu_k) \leq \beta h_f(\mu_1) + (1-\beta) \alpha + \alpha'.
$$
\end{proposition}

\begin{proof}
Let $k \in \mathcal{K}$, which is a sequence along which $(h_{f_k}(\nu_k))_k$ converges to its limsup. We start with some notations.
Let $\iota_k > 0$ be such that $\iota_k \to 0$.
Let $\mathcal{P}_k$ be a finite partition such that $\nu_k(\partial \mathcal{P}_k) = 0$ and such that
$$
\lim\limits_{n \to +\infty} \frac{1}{n} H_{\nu_k}(\mathcal{P}_k^n) \geq h(\nu_k) - \iota_k.
$$
We may assume that $\mathtt{F}_k$ from Section \ref{ssec:cont-expo-HB} has a $\nu_k$-measure sufficiently close to $1$ so that
$$
\nu_k(\M \backslash \mathtt{F}_k) \log \sharp\mathcal{P}_k \underset{k \to +\infty}{\longrightarrow} 0.
$$
Indeed, the sequence $(\sharp \mathcal{P}_k)_k$ may not be bounded.
Let $M, m \in \bbN^*$.
For $q \in \bbN^*$, we consider the partition $\mathcal{R}_{q,k}$ defined in Section \ref{ssec:ent-HB} by using the construction from Section \ref{ssec:def-partition}.
From Lemma \ref{lem:entropy-R}, the partition $\mathcal{R}_{q,k}$ satisfies the hypotheses of Proposition \ref{prop:entropy-ineq}.
Therefore, from Proposition \ref{prop:entropy-final}, we have for every $m \in \bbN^*$ and $M \in \bbN$ large enough
\begin{align*}
\limsup\limits_{k \to +\infty} h_{f_k}(\nu_k)
&\leq \limsup\limits_{k \to +\infty} \left ( \beta_k^{M,m} \frac{1}{m} H_{\overline{\xi_k^{M,m}}}(\mathcal{R}_{q,k}^m) + \limsup\limits_{n \to +\infty} \frac{1}{n} H_{\zeta_k}(\mathcal{P}_k^n \mid \mathcal{R}_{q,k}^{E_{k,n}^{M,m}}) \right .\\
&+ \nu(\M \backslash \mathtt{F}_k) \log \sharp \mathcal{P}_k + \frac{3 \log M}{M} + \iota_k  \\
&\left.+\limsup\limits_{n \to +\infty} \frac{m}{n}\sum\limits_{E \in \mathcal{E}_{k,n}^{M,m}} \nu_k(E) \log \left ( \sharp(\mathcal{R}_{q,k})_{(\nu_k)_E^E} \right ) \sharp\partial E \right ). 
\end{align*}

Then, by letting $M$ then $m$ go to infinity and applying Proposition \ref{prop:entropy-cv} and item $6)$ from Lemma \ref{lem:entropy-R}, we obtain for any $q \in \bbN^*$
\begin{align*}
\limsup\limits_{m \to +\infty}\limsup\limits_{M \to +\infty}\limsup\limits_{k \to +\infty} h_{f_k}(\nu_k) \leq \beta h_f(\mu_1) + \limsup\limits_{m \to +\infty} \limsup\limits_{M \to +\infty} \limsup\limits_{k \to +\infty} \limsup\limits_{n \to +\infty} \frac{1}{n} H_{\zeta_k}(\mathcal{P}_k^n \mid \mathcal{R}_{q,k}^{E_{k,n}^M}).
\end{align*}

By Lemma \ref{lem:reparamization-partition} and Lemma \ref{lem:repabound}, we have for any $k,M,m$
$$
\begin{aligned}
\limsup\limits_{n \to +\infty} \frac{1}{n}H_{\zeta_k}(\mathcal{P}_k^n \mid \mathcal{R}_{q,k}^{E_{k,n}^{M,m}})&\leq \limsup\limits_{n \to +\infty} \sum_{ Q\in \mathcal{R}^{E^{M,m}_n}_{q,k}} \zeta_k(Q)\frac{1}{n}\log\sharp\Psi_{Q}\\
&\leq \limsup\limits_{n \to +\infty} \int \left( 1-\frac{\# 
E_{k,n}^{M,m}(x)}{n} \right)C(f_k) d \zeta_k + \frac{d}{r}\log^+ ||df_k ||_{\infty}\\
&+\frac{1}{(r-1)\cdot \nu_k(\mathtt{F}_k)}\left( \int \frac{\log^{+}|\Jac f_k^q|}{q}d\xi^{M,m,q}_{k,n}-\int \frac{\log|\Jac f_k^q|}{q}d\xi^{M,m,q}_{k,n} \right)\\
&+\gamma_{n,k,M,m,q}(f_k).
\end{aligned}
$$
We then investigate the three terms separately:
$$
\begin{aligned}
    &\operatorname{(I)}= \int \left( 1-\frac{\# E^{M,m}_{k,n}(x)}{n} \right)C(f_k) d \zeta_k + \frac{d}{r}\log^+ ||d f_k ||_{\infty},\\
    &\operatorname{(II)}=\frac{1}{(r-1)\cdot \nu_k(\mathtt{F}_k)}\left( \int \frac{\log^{+}|\Jac f_k^q|}{q}d\xi^{M,m,q}_{k,n}-\int \frac{\log|\Jac f_k^q|}{q}d\xi^{M,m,q}_{k,n} \right),\\
    &\operatorname{(III)}=\gamma_{n,k,M,m,q}(f_k).
\end{aligned}
$$

Firstly, we have $\int \frac{\# E^{M,m}_{k,n}(x)}{n} d \zeta_k \rightarrow \beta$ when $m,M,k,n\rightarrow +\infty$. Thus by continuity of $f \mapsto C(f)$ in the $C^1$-topology, we have
$$
\limsup\limits_{m \to +\infty} \limsup\limits_{M \to +\infty} \limsup\limits_{k \to +\infty} \limsup\limits_{n \to +\infty}\operatorname{(I)}=(1-\beta) C(f) + \frac{d}{r}\log^+ ||df ||_{\infty}.
$$

Secondly, by Lemma \ref{lem:term-II},
$$
\limsup\limits_{q \to +\infty}\limsup\limits_{m \to +\infty} \limsup\limits_{M \to +\infty} \limsup\limits_{k \to +\infty} \limsup\limits_{n \to +\infty}\operatorname{(II)}=0.
$$

Thirdly, by Lemma \ref{lem:repabound}, we have
$$
\limsup\limits_{q \to +\infty}\limsup_{m\rightarrow +\infty}\limsup_{M\rightarrow +\infty}\limsup_{k\rightarrow+\infty}\limsup_{n\rightarrow +\infty}\operatorname{(III)}=0
$$
Therefore, the choice of $\alpha'$ and hypothesis $ii)$ from Proposition \ref{prop:main-prop} give
$$
\limsup\limits_{k \to +\infty} h_{f_k}(\nu_k) \leq \beta h_f(\mu_1) + (1-\beta) \alpha + \alpha'.
$$
\end{proof}

We are left to prove item $3')$ from Theorem \ref{th:main2} for $\delta$ defined with $p=1$.

\begin{proposition}
We have
$$
\liminf\limits_{\mathcal{K} \ni k \to +\infty} \Lambda_{f_k}(\nu_k) \geq \beta \Lambda_f(\mu_1) + (1 - \beta) \delta.
$$
\end{proposition}

\begin{proof}
Recall that in the beginning of section \ref{ssec:cont-expo-HB}, we explained that we can replace $\nu_k$ with $\zeta_k$ and we wrote
$$
\Lambda_{f_k}(\zeta_k) = \Lambda_{f_k}(\xi_{k,n,\mathtt{F}_k}^M) + \Lambda_{f_k}(\eta_{k,n,\mathtt{F}_k}^M).
$$
In that section, we also proved Proposition \ref{prop:conv-expo-HB}, which gives
$$
\lim\limits_{M \to +\infty} \lim\limits_{k \to +\infty} \lim\limits_{n \to +\infty} \Lambda_{f_k}(\xi_{k,n,\mathtt{F}_k}^M) = \beta \Lambda_f(\mu_1).
$$
Then, Proposition \ref{prop:expo-neutral} gives
$$
\liminf\limits_{M \to +\infty} \liminf\limits_{k \to +\infty} \liminf\limits_{n \to +\infty} \Lambda_{f_k}(\eta_{k,n,\mathtt{F}_k}^M) \geq (1 - \beta) \delta.
$$
Recall that $M,k,n$ actually belong to some subsequences given in Section \ref{ssec:main-prop} by Cantor's diagonal argument.
In particular, $k$ may not go through all of $\mathcal{K}$.
However, we can obtain the liminf in $k \in \mathcal{K}$ in the statement by first choosing a subsequence of $\mathcal{K}$ along which $(\Lambda_{f_k}(\nu_k))_{k \in \mathcal{K}}$ converges to its liminf, then use Cantor's diagonal argument again.  
This concludes the proof of item $3')$ of Theorem \ref{th:main2}, hence of Proposition \ref{prop:main-prop}.
\end{proof}

\subsubsection{Proof of Theorem \ref{th:main2}}
\label{section:proof-th}

We show how to obtain an iterate of $f$ which satisfies the hypotheses of Proposition \ref{prop:main-prop}.
We first deal with the cases where Theorem \ref{th:main2} is true for $\beta = 0$.
Recall that $\mathcal{K}$ is a sequence of integers such that $(h_{f_k}(\nu_k))_k$ converges to its limsup along $\mathcal{K}$.

\begin{lemma}
\label{th:small-liminf}
If $\liminf\limits_{\mathcal{K} \ni k \to +\infty} \Lambda_{f_k}(\nu_k) \leq \alpha$, then Theorem \ref{th:main2} works with $\beta = 0$.
\end{lemma}

\begin{proof}
We let $\mu_1$ be any $f$-invariant measure.
We recall that each $\nu_k$ has $d$ positive Lyapunov exponents.
In this setting, Ruelle's inequality \cite{Ruelle1978} gives 
\begin{align*}
\limsup\limits_{k \to +\infty} h_{f_k}(\nu_k)
&=\lim\limits_{\mathcal{K} \ni k \to +\infty} h_{f_k}(\nu_k)\\
&\leq \liminf\limits_{\mathcal{K} \ni k \to +\infty} \int \log |\Jac f_k| d \nu_k\\
&\leq \alpha.
\end{align*}
\end{proof}

We first give some notations to deal with iterates of $f_k$ and $f$.
If $p \geq 1$, we will consider $f_k^p$ and $f^p$.
Then, we want a converging sequence of $f_k^p$-ergodic measures.
Let $\nu_k^p$ be an ergodic component of $\nu_k$ for $f^p$, thus
$$
h_{f_k^p}(\nu_k^p) = p h_{f_k}(\nu_k) \; \; \; \text{and} \; \; \; \Lambda_{f_k^p}(\nu_k^p) = p \Lambda_{f_k}(\nu_k).
$$
By taking a subsequence, contained in $\mathcal{K}$, we may assume that $(\nu_k^p)_k$ converges to some $\mu^p$.
Notice that $\mu = \frac{1}{p}\sum\limits_{i=0}^{p-1} f^i_* \mu^p$.
Let $\alpha_p = p \alpha > p \left ( d - 1 + \frac{d^2}{r} \right ) R(f) = \left ( d - 1 + \frac{d^2}{r} \right ) R(f^p)$ and $\alpha_p' = p \alpha' > \frac{d}{r} R(f^p)$.\\


\begin{lemma}
\label{lem:main-prop-iterate}
If
$\liminf\limits_{\mathcal{K} \ni k \to +\infty} \Lambda_{f_k}(\nu_k) > \alpha$, then there exists $p$ such that $f_k^p$, $f^p$, $\nu_k^p$ and $\mu^p$ satisfy the hypotheses of Proposition \ref{prop:main-prop}. 
\end{lemma}

\begin{proof}
For $p$ large enough we have $\frac{d}{r} \log^+ || d f^p ||_{\infty} < \alpha_p'$.
Then there is $p$ large enough such that
$$
\liminf\limits_{\mathcal{K} \ni k \to +\infty} \Lambda_{f_k}(\nu_k) > 
 \alpha + \frac{d (dr+r-1)}{r(r+d-1)} \left ( \frac{1}{p}\log^+ || d f^p ||_{\infty} - R(f) \right ).
$$
Therefore, we have
$$
\liminf\limits_{\mathcal{K} \ni k \to +\infty} \Lambda_{f_k^p}(\nu_k^p) = p\liminf\limits_{\mathcal{K} \ni k \to +\infty} \Lambda_{f_k}(\nu_k) > \alpha_p + \frac{d (dr+r-1)}{r(r+d-1)} \left ( \log^+ || d f^p ||_{\infty} - R(f^p) \right ).
$$
which is hypothesis $i)$ from Proposition \ref{prop:main-prop}.
For item $ii)$, notice that $C(f^p)/p \underset{p \to +\infty}{\longrightarrow} \left (d-1 + \frac{d^2}{r} \right ) R(f) < \alpha$.
Therefore, we can assume that $p$ is large enough so that $C(f^p) < p \alpha = \alpha_p$.
\end{proof}

We are now ready to prove Theorem \ref{th:main2}.

\begin{proof}[Proof of Theorem \ref{th:main2}]
If
$\liminf\limits_{\mathcal{K} \ni k \to +\infty} \Lambda_{f_k}(\nu_k) \leq \alpha$,
then Lemma \ref{th:small-liminf} concludes.
Assume now that
$\liminf\limits_{\mathcal{K} \ni k \to +\infty} \Lambda_{f_k}(\nu_k) > \alpha$.
Let $p$ be given by Lemma \ref{lem:main-prop-iterate}.
We apply Proposition \ref{prop:main-prop} to the $p$-th iterate of $f$, which gives $\beta_p > 0, \mu_0^p$ and $\mu_1^p$ such that Theorem \ref{th:main2} holds for $f_k^p, f^p, \nu_k^p,\mu^p$ and $\delta_p = \min \left \{ \alpha_p - \frac{d(dr+r-1)}{r(r+d-1)} R(f^p), \frac{r-1}{r+d-1} \log 2 \right \}$.
Let $\beta = \beta_p$, $\mu_{0/1} = \frac{1}{p} \sum\limits_{i=0}^{p-1} f^i_* \mu_{0/1}^p$, $\delta = \frac{1}{p} \delta_p$.
We show that these quantities give Theorem \ref{th:main2} for $f,f_k,\nu_k$ and $\mu$.\\

Proof of \ref{th:main2}.1):
We chose $\beta = 0$ whenever the announced condition was not satisfied. Conversely, if this condition is satisfied, then we apply Proposition \ref{prop:main-prop} which gives $\beta_p > 0$, so $\beta = \beta_p$ concludes.\\

Proof of \ref{th:main2}.3)1'):
We have a set of full $\mu_1^p$-measure where $\Lambda_{f^p} \geq \delta_p = p \delta$.
Therefore, we have a set of full $\mu_1^p$-measure where $\Lambda_f \geq \delta$.
Since this property is $f$-invariant $\mu_1^p$-a.e., there is a set of full $\mu_1$-measure where $\Lambda_f \geq \delta$, which concludes.\\

Proof of \ref{th:main2}.3)3'):
We have
$$
\liminf\limits_{k \to +\infty} \Lambda_{f_k}(\nu_k) = \frac{1}{p}\liminf\limits_{k \to +\infty} \Lambda_{f_k^p}(\nu_k^p)
\geq \frac{1}{p} \left ( \beta_p \Lambda_{f^p}(\mu_1^p) + (1-\beta_p) \delta_p\right )
=\beta \Lambda_f(\mu_1) + (1-\beta) \delta.
$$

Proof of \ref{th:main2}.2):
We have
\begin{align*}
\limsup\limits_{k \to +\infty} h_{f_k}(\nu_k)
&= \frac{1}{p} \limsup\limits_{k \to +\infty} h_{f_k^p}(\nu_k^p)\\
&\leq \frac{1}{p}\left ( \beta_p h_{f^p}(\mu_1^p) + (1 - \beta_p) \alpha_p + \alpha_p' \right )\\
&= \beta h_f(\mu_1) + (1-\beta)\alpha + \alpha'.
\end{align*}
\end{proof}

\begin{proof}[Proof of Theorem \ref{th:main-cont}]
Let $\widetilde{\mu}_1, \widetilde{\mu}_0, \widetilde{\beta}$ be given by Theorem \ref{th:main2}.
Let $\beta \mu_1$ be the expanding component of $\widetilde{\beta} \widetilde{\mu}_1$, meaning that $\mu_1$ is an $f$-invariant probability measure, that $\beta$ is in $[0,1]$, and that $\lambda_d^f(x) > 0$ for $\mu_1$-almost every $x$ and $\lambda_d^f(x) \leq 0$ for $(\widetilde{\beta} \widetilde{\mu}_1 - \beta \mu_1)$-almost every $x$.
Let $(1 - \beta)\mu_0 = \mu - \beta \mu_1$.
Thus item $1)$ is satisfied.
To show item $3)$, we use item $3')$ and the facts that $\beta \mu_1$ is a component of $\widetilde{\beta} \widetilde{\mu}_1$ and that $\Lambda_f$ is nonnegative $\widetilde{\mu}_1$-almost everywhere, which comes from item $1')$.
Let $\widetilde{\nu}$ (resp. $\widetilde{\gamma}$) be the normalization (resp. mass) of the measure $\widetilde{\beta} \widetilde{\mu}_1 - \beta \mu_1$, thus $\widetilde{\gamma} \widetilde{\nu} = \widetilde{\beta} \widetilde{\mu}_1 - \beta \mu_1$ and $\widetilde{\gamma} = \widetilde{\beta} - \beta$.
To show item $2)$, we notice that Ruelle's inequality gives $h_f(\widetilde{\nu}) \leq (d-1) R(f)$.
Then we write
\begin{align*}
\limsup\limits_{k \to +\infty} h_{f_k}(\nu_k)
&\leq \widetilde{\beta} h_f(\widetilde{\mu}_1) + (1-\widetilde{\beta}) \alpha + \alpha'\\
&= \beta h_f(\mu_1) + \widetilde{\gamma} h_f(\widetilde{\nu}) + (1 - \widetilde{\beta}) \alpha + \alpha'\\
&\leq \beta h_f(\mu_1) + (\widetilde{\beta} - \beta )(d-1) R(f) + (1 - \widetilde{\beta}) \alpha + \alpha'\\
\alpha > (d-1) R(f) \hspace{1em}&\leq \beta h_f(\mu_1) + (1 - \beta) \alpha + \alpha'. 
\end{align*}
It remains to show item $4)$. By item $1)$ from Theorem \ref{th:main2}, and from the fact that $\widetilde{\beta} \geq \beta$, one direction of the equivalence is immediate.
Then assume that $\liminf\limits_{\mathcal{K} \ni k \to +\infty} \Lambda_{f_k}(\nu_k) > \alpha$.
Then we are in the setting where $\widetilde{\beta}, \widetilde{\mu_1}$ and $\widetilde{\mu_0}$ are obtained from the constructions from Section \ref{sec:ent-exp-HB}.
In particular, Lemma \ref{lem:mu1-exp} gives that $\widetilde{\mu_1}$ has a non-zero expanding component, concluding that $\beta > 0$.
\end{proof}

\begin{proof}[Proof of Theorem \ref{thm:main-con-1}]
Let $r_n \to +\infty$.
For any $n \in \bbN$, let $\alpha_n > \left ( d-1 + \frac{d^2}{r_n} \right ) R(f)$ and $\alpha_n' > \frac{d}{r_n} R(f)$.
Assume that $\alpha_n \to (d-1)R(f)$ and $\alpha_n' \to 0$.
For any $n$, let $\beta_n, \mu_0^n$ and $\mu_1^n$ be given by Theorem \ref{th:main-cont}.
By taking subsequences, we may assume that $\beta_n \to \beta, \mu_i^n \overset{\text{weak-}*}{\longrightarrow} \mu_i$, $i \in \{0, 1 \}$.
By upper semi-continuity of the metric entropy for $C^{\infty}$ maps \cite{NewhouseContEntropy}, letting $n$ go to infinity gives
$$
\limsup\limits_{k \to +\infty} h_f(\nu_k) \leq \beta h_f(\mu_1) + (1-\beta) (d-1) R(f).
$$
To obtain items $1)$ and $3)$, weak-$*$ convergence is not enough.
However, for any $n$, we have $\beta_n \mu_1^n \leq \mu$ and $\beta \mu_1 \leq \mu$.
By using the density of continuous functions in $L^1(\mu)$, one can obtain $\beta_n \mu_1^n (A) \underset{n \to +\infty}{\longrightarrow} \beta \mu_1(A)$ for any Borel set $A$.
This gives item $1)$, and one can then get the convergence on integrable test functions, which gives items $3)$.
\end{proof}


\vspace{2em}

\vspace{3em}

\noindent\textsc{Alexandre Delplanque}, Sorbonne Universit\'e and Universit\'e Paris Cit\'e, CNRS, Laboratoire de Probabilit\'es, Statistique et Mod\'elisation\\
F-75005 Paris, France\\
\textit{E-mail:} 
{\fontfamily{lmtt}\selectfont
delplanque@lpsm.paris
}

\vspace{1em}

\noindent\textsc{Hengyi Li}, Einstein Institute of Mathematics, The Hebrew University of Jerusalem,\\ Jerusalem 9190401, Israel \\
\textit{E-mail:}
{\fontfamily{lmtt}\selectfont
hengyi.li@mail.huji.ac.il
}

\end{document}